\documentclass{amsart} 
\usepackage[utf8]{inputenc}
\usepackage[T1]{fontenc}
\usepackage[english]{babel}
\usepackage{amssymb,amsthm}
\usepackage{mathtools}\mathtoolsset{showonlyrefs}
\usepackage{stmaryrd}
\usepackage{microtype}
\usepackage[inline]{enumitem}
\usepackage{mathrsfs}
\usepackage{booktabs} % beautiful tables
\usepackage[bottom=1in]{geometry}
\usepackage{subcaption} % used for subfigure
\usepackage{tikz-cd}
\usepackage{aliascnt}
\usepackage{appendix}

\usepackage{csquotes} % provides quotes
\usepackage{hyperref}
\AtBeginDocument{%
\let\oldref\ref
\renewcommand{\ref}[1]{\mbox{\oldref{#1}}%
}}

\tikzset{symbol/.style={draw=none, every to/.append style={%
,edge node={node [sloped,allow upside down,auto=false]{$#1$}}
}} }

\newcommand{\Z}{\mathbb Z}
\newcommand{\Q}{\mathbb Q}

\newcommand{\F}{\mathbb F}

\newcommand{\field}{\mathfrak F}
\newcommand{\Satake}{\mathcal{S}} % Satake homomorphism

\newcommand{\sm}{\mathrm{sm}}

\newcommand{\op}{\mathrm{op}}

\newcommand{\qis}{\mathrm{qis}}
\newcommand{\Chain}{\mathrm{C}} % category of complexes
\newcommand{\K}{\mathrm{K}} % homotopy category
\newcommand{\D}{\mathrm{D}} % derived category
\newcommand{\RR}{\mathrm{R}} % right derived functor
\newcommand{\LL}{\mathrm{L}} 
\renewcommand{\H}{\mathrm{H}} % cohomology
\newcommand{\bounded}{\mathrm{b}} % bounded derived category
\newcommand{\adm}{\mathrm{adm}} % admissible
\newcommand{\gadm}{\mathrm{a}} % globally admissible
\newcommand{\cpt}{\mathrm{c}} % compact

\newcommand{\one}{\mathbf{1}}
\newcommand{\pholder}{{\mathchoice%
{\raisebox{-1.5pt}{$-$}} % display style
{\raisebox{-1.5pt}{$-$}} % text style
{\scalebox{.7}{\raisebox{-1.5pt}{$-$}} } % script style
{\scalebox{.5}{\raisebox{-1.5pt}{$-$}} } % scriptscript style
}}

\newcommand{\eg}{\textit{e.g.}}
\newcommand{\ie}{\textit{i.e.}}
\newcommand{\loccit}{\textit{loc.\,cit.}}
\newcommand{\opcit}{\textit{op.\,cit.}}

\renewcommand{\setminus}{\smallsetminus}

\newcommand{\Hecke}{\mathcal H}
\newcommand{\Fil}{\mathcal{F}\kern-1.5pt{\mathit{il}} } % filtration
\renewcommand{\to}{\longrightarrow}
\newcommand{\To}{\Longrightarrow}

\newcommand{\ol}[1]{\overline{#1}}
\newcommand{\ul}[1]{\underline{#1}}
\newcommand{\alg}[1]{\mathbf{#1}}
\newcommand{\set}[2]{\left\{#1\,\middle|\,#2\right\}}

\newcommand{\cat}[1]{\mathscr{#1}}

\newcommand{\Ab}{\mathrm{(Ab)}}
\newcommand{\Sets}{\mathrm{(Sets)}}
\newcommand{\Vect}{\mathrm{Vect}}

\newcommand{\Fpbar}{\ol{\F}_p}

\newcommand{\longtwoheadrightarrow}%
{\longrightarrow\hspace{-1.2em}\rightarrow\hspace{.2em}}
\renewcommand{\twoheadrightarrow}%
{\rightarrow\hspace{-1.2em}\rightarrow\hspace{.2em}}
\newcommand{\longhookrightarrow}%
{\lhook\joinrel\relbar\joinrel\rightarrow}

\DeclareMathOperator{\diag}{diag}
\DeclareMathOperator{\End}{End}
\DeclareMathOperator{\Ext}{Ext}
\DeclareMathOperator{\q}{\mathbf{q}}
\DeclareMathOperator{\ii}{\mathbf{i}}
\DeclareMathOperator{\GL}{GL}
\DeclareMathOperator{\cores}{cores}
\DeclareMathOperator{\res}{res}
\DeclareMathOperator{\conj}{conj}
\DeclareMathOperator{\Char}{char}
\DeclareMathOperator{\Inf}{Inf}
\DeclareMathOperator{\val}{val}
\DeclareMathOperator{\Sp}{Sp} % Steinberg representation

\DeclareMathOperator{\Sym}{Sym}
\DeclareMathOperator{\Mod}{Mod}
\DeclareMathOperator{\ev}{ev}

\DeclareMathOperator{\Res}{Res}
\DeclareMathOperator{\Ord}{Ord}

\DeclareMathOperator{\Hom}{Hom}
\DeclareMathOperator{\RHom}{RHom}
\DeclareMathOperator{\iHom}{\ul{Hom}}
\DeclareMathOperator{\id}{id}
\DeclareMathOperator{\Nat}{Nat}
\DeclareMathOperator{\Ind}{Ind}
\DeclareMathOperator{\RInd}{RInd}
\DeclareMathOperator{\ind}{ind}
\DeclareMathOperator{\spann}{span}
\DeclareMathOperator{\colim}{colim}
\DeclareMathOperator{\Image}{Im}
\DeclareMathOperator{\Ker}{Ker}
\DeclareMathOperator{\Coker}{Coker}
\DeclareMathOperator{\pr}{pr}
\DeclareMathOperator{\Rep}{Rep}

\newtheorem{prop}[equation]{Proposition}
\newtheorem{thm}[equation]{Theorem}
\newtheorem{lem}[equation]{Lemma}
\newtheorem{cor}[equation]{Corollary}
\newtheorem{hyp}[equation]{Hypothesis}
\newtheorem{propintro}{Proposition}
\newtheorem{thmintro}[propintro]{Theorem}

\theoremstyle{definition}

\newtheorem{notation}[equation]{Notation}
\newtheorem*{defn}{Definition}
\newtheorem*{claim*}{Claim}
\newtheorem{warning}[equation]{Warning}
\newtheorem{ex}[equation]{Example}
\newtheorem{rmk}[equation]{Remark}
\newtheorem{explanation}[equation]{Explanation}

\theoremstyle{remark}
\newtheorem*{rmk*}{Remark}
\newtheorem*{ex*}{Example}
\newtheorem*{notation*}{Notation}

\counterwithin{equation}{subsection}

\title{The Left Adjoint of Derived Parabolic Induction}
\author{Claudius Heyer}
\address{Mathematisches Institut, Universit\"at M\"unster, Einsteinstra\ss{}e 62,
D-48149 M\"unster, Germany}
\email{cheyer@uni-muenster.de}

\begin{document}
\begin{abstract} 
We prove that the derived parabolic induction functor, defined on the unbounded
derived category of smooth mod $p$ representations of a $p$-adic reductive
group, admits a left adjoint $\mathrm{L}(U,-)$. We study the cohomology
functors $\mathrm{H}^i\circ \mathrm{L}(U,-)$ in some detail and deduce that
$\mathrm{L}(U,-)$ preserves bounded complexes and global admissibility in the
sense of Schneider--Sorensen. Using $\mathrm{L}(U,-)$ we define a derived Satake
homomorphism und prove that it encodes the mod $p$ Satake homomorphisms defined
explicitly by Herzig.
\end{abstract} 
\maketitle
\tableofcontents

\section{Introduction} 
\subsection{The setting} 
This article is devoted to the investigation of the left adjoint of the derived
parabolic induction functor in natural characteristic. Let $G$ be a $p$-adic
reductive group, that is, the group of $\field$-points of a connected reductive
group defined over a finite field extension $\field/\Q_p$. 
Let $k$ be a coefficient field of characteristic $p$ and denote $\Rep_k(G)$ the
category of smooth $G$-representations on $k$-vector spaces. Since several
important functors fail to be exact, like the functor of $K$-invariants for
compact open subgroups $K\subseteq G$, Schneider \cite{Schneider.2015}
introduced the derived category
\[
\D(G) \coloneqq \D(\Rep_k(G))
\]
of unbounded cochain complexes to remedy this problem. The main result of \opcit\
proves that, given a torsion-free open pro-$p$ subgroup $K\subseteq G$, taking
derived $K$-invariants induces a derived
equivalence of $\D(G)$ with the derived category $\D(\Hecke_K^\bullet)$ of
dg-modules over a certain differential graded algebra $\Hecke_K^\bullet$. This
strongly suggests that the study of smooth mod $p$ representations is
best done on the derived level.

Let $P = UM$ be a parabolic subgroup of $G$ with Levi subgroup $M$ and unipotent
radical $U$. The parabolic induction functor $i_M^G\colon \Rep_k(M)\to
\Rep_k(G)$ is defined as the composition
\[
\Rep_k(M) \xrightarrow{\Inf^M_P} \Rep_k(P) \xrightarrow{\Ind_P^G} \Rep_k(G),
\]
where $\Inf^M_P$ is the inflation along the canonical projection
$P\longtwoheadrightarrow M$ and $\Ind_P^G$ is the smooth induction functor. The
functors $\Inf^M_P$ and $\Ind_P^G$ are exact. Hence also $i_M^G$ is exact and
extends to a functor
\[
\RR i_M^G\colon \D(M)\to \D(G)
\]
on the derived categories: Here, $\RR i_M^G$ is given by applying $i_M^G$
termwise to a complex. (The different notation serves the only purpose of being
able to distinguish between underived and derived parabolic induction.) The
compatibility of $\RR i_M^G$ with
an analogous functor on $\D(\Hecke_K^\bullet)$ is explored in
\cite{Scherotzke-Schneider.2022}, which also discusses the existence of left and
right adjoints of $\RR i_M^G$, referring for the former to my unpublished notes. 

At this point one might ask why one should care about a left adjoint. I propose
two answers: for one, in the classical theory of smooth representations, the
left adjoint is exact and plays an essential role in the classification of
irreducible smooth representations in terms of parabolic induction, namely, the
irreducible representations which do not appear as a subrepresentation of a
parabolically induced representation are precisely those in the kernel of the
corresponding left adjoint. This strongly indicates that the left adjoint for
derived parabolic induction will take a fundamental position in understanding
the category $\D(G)$ as a whole. The second answer
is that the derived left adjoint will be an important tool to compute higher
$\Ext$-groups between parabolically induced representations. There are already
some computations which use Emerton's higher ordinary parts functor,
\cite{EmertonII, Hauseux.2016, Hauseux.2018}. But the strongest of these results
are conditional to an open conjecture of Emerton, which states that the higher
ordinary parts functor is a right derived functor. In the follow-up paper
\cite{Heyer.2023} I use the derived left adjoint to unconditionally compute these
$\Ext$-groups.

Let me briefly discuss why it is not immediate that $\RR i_M^G$ admits a left
adjoint. It is well-known that the Jacquet-functor $\Rep_k(G)\to \Rep_k(M)$,
$V\longmapsto V_U$, is left adjoint to $i_M^G$. The left
derived functor of $(\pholder)_U$ would be a promising candidate for the left
adjoint of derived parabolic induction. Unfortunately, the category $\Rep_k(G)$
does not have enough projectives, since $\Char(k) = p$. (It is a folklore
conjecture that there exist no non-zero projectives at all; first evidence in
this direction is contained in \cite[Rmk.~2.20]{Dupre-Kohlhaase.2023}, which proves the statement for compact $G$; see also \cite[Thm.~3.1]{Chirvasitu-Kanda.2023} for the more general statement about profinite groups.)
Therefore, one cannot use projective resolutions to prove that the left derived
functor exists. 
A second approach is to show that $\RR i_M^G$ commutes with arbitrary small
products and to then use a variant of Brown representability to deduce the
existence of a left adjoint. It is easy to show that $i_M^G$ preserves
arbitrary products, but since products in $\D(M)$ and $\D(G)$ are not computed
by taking products of representing complexes, it is not obvious that $\RR i_M^G$
should have the same property. 

In view of these concerns it was surprising that $\RR i_M^G$ does in fact
commute with small products. The initial proof was rather technical
and relied on an analysis of Hochschild--Serre spectral sequences associated
with $K_P\cap U \trianglelefteq K_P$, where $K_P$ is a compact open subgroup of
$P$, and analyzing how these behave with respect to restriction to smaller
subgroups. The more conceptual proof presented here is by appealing
to one of the main results of \cite{Balmer.2016}. The key ingredient is the
observation that the derived functor $\RR\H^0(U\cap K_P,\pholder)\colon \D(K_P)\to
\D(K_P/U\cap K_P)$ preserves compact objects provided $K_P$ and $K_P/U\cap K_P$
are torsion-free.

\subsection{Main results} 
Note that $\RR i_M^G = \RR\Ind_P^G \circ
\RR\Inf^M_P$ so that it suffices to determine the left adjoints of
$\RR\Ind_P^G$ and $\RR\Inf^M_P$ separately. The left adjoint of $\RR\Ind_P^G$ is
easily seen to be the restriction functor $\RR\Res^G_P\colon \D(G)\to \D(P)$,
and hence it remains to show that $\RR\Inf^M_P$ admits a left adjoint.

Granted the existence, the next task would then be to find an explicit
description of said left adjoint. In view of this it is convenient to study the
derived inflation
functor $\RR\Inf^T_S\colon \D(T)\to \D(S)$ along a surjective open monoid
homomorphism $f\colon S\longtwoheadrightarrow T$ with kernel $U$, where $S$ and
$T$ are open submonoids of some $p$-adic Lie groups. 
The most relevant example is the following: Fix a compact
open subgroup $K_U$ of $U$ and consider the positive monoid $T = M^+$
consisting of those elements $m\in M$ with $mK_Um^{-1} \subseteq K_U$ as well as
the open submonoid $S = K_UM^+$ of $P$ (in which case the kernel of $S\to T$ is
$K_U$). This added flexibility is useful to
deduce precise statements about the left adjoint of $\RR\Inf^M_P$. 

In this general setting, the first main result is:

\begin{thmintro}[Theorem~\ref{thm:Inf-prod-general}] 
\label{thm:A}
The functor $\RR\Inf^T_S\colon \D(T)\to \D(S)$ admits a left adjoint $\LL_U$. In
particular, $\RR i_M^G$ admits a left adjoint $\LL(U,\pholder)$.
\end{thmintro} 

We deduce the theorem from a variant of Brown representability by proving that
$\RR\Inf^T_S$ commutes with small products. Observing that restriction to
compact open subgroups is conservative and product preserving, we reduce to the
case where $S$ and $T$ are torsion-free compact $p$-adic Lie groups. We then
verify that the categories $\D(T)$ and $\D(S)$ are rigidly-compactly generated
(in the sense of \cite[1.2.~Hyp.]{Balmer.2016}) and that the right adjoint
$\RR\H^0(U,\pholder)$ of $\RR\Inf^T_S$ preserves compact objects. With these
hypotheses in place, the theorem follows from \cite[3.3.~Thm.]{Balmer.2016}. In
fact, under these restrictions on $S$ and $T$, the shifted invariants functor
$\RR\H^0(U,\pholder)[\dim U]$ is a left adjoint of $\RR\Inf^T_S$.

As a special case of Theorem~\ref{thm:A} we obtain that $\RR\Inf^M_P$ admits a
left adjoint, which we denote by $\LL_U$. 
The next task then is to try to explicitly compute $\LL_U$, or at least its
cohomology functors $\LL^i_U$. A formal argument shows $\LL_U^i \cong 0$, for
$i>0$,
and that $\LL_U^0$ is the usual functor of $U$-coinvariants (\ie, the left
adjoint of $\Inf^M_P$). In order to analyze $\LL_U^i$ for $i<0$, we employ
$p$-adic monoids. To motivate the idea, we recall a different description of the
(underived) coinvariants functor $\LL_U^0$. The unipotent radical $U$ is the
union of compact open subgroups, say $U = \bigcup_{r\in \Lambda}H_r$. It is
a well-known and easy fact that $\LL_U^0 \cong \colim_{r\in\Lambda}
\LL_{H_r}^0$. Although this isomorphism describes $\LL_U^0$ in terms of
functors of coinvariants with respect to compact open subgroups (which should be
easier to understand), it is in this formulation difficult to extend to the
derived setting. This is where the positive monoid enters the stage: Fixing a
torsion-free compact open subgroup $K_U$ of $U$, the positive monoid $M^+$ of
elements $m\in M$ with $mK_Um^{-1}\subseteq K_U$ acts naturally on
$\LL_{K_U}^0(V)$, for $V\in \Rep_k(P)$. It is well-known that there exists a
central positive element $z\in M$ such that $U = \bigcup_{i\ge0} z^{-i}K_Uz^i$.
There is a natural isomorphism $\colim_{i} \LL_{z^{-i}K_Uz^i}^0(V) \cong
\colim_z\LL_{K_U}^0(V)$, where the transition maps in the latter colimit are
given by multiplication with $z$. 
In fact, $\colim_z$ is the (exact) functor $\Rep_k(M^+)\to \Rep_k(M)$ given by
inverting the action of $z$. 
It is left adjoint to the restriction functor $\Res^{M}_{M^+}$ and also denoted
$\ind_{M^+}^M$. Thus, there is a natural isomorphism
$\LL_U^0(V)\cong \ind_{M^+}^M\LL_{K_U}^0(V)$, which can easily be extended to
the derived categories. Concretely, writing $P^+\coloneqq K_UM^+$ (it is an open
submonoid in $P$), the derived inflation $\RR\Inf^{M^+}_{P^+}$ admits a left
adjoint $\LL_{K_U}\colon \D(P^+)\to \D(M^+)$, and we find a natural isomorphism
\[
\RR\ind_{M^+}^M \LL_{K_U} \RR\Res^P_{P^+} \xRightarrow{\cong} \LL_U
\]
of functors $\D(P)\to \D(M)$ (Proposition~\ref{prop:LU-positive}). The group
$K_U$ is a compact, torsion-free $p$-adic Lie group and hence has finite
cohomological dimension by a result of Lazard and Serre
(see~\cite[Cor.~(1)]{Serre.1965}). 
We deduce from this fact in Corollary~\ref{cor:RH(KU,-)=LKU} that $\LL_{K_U}$ is
naturally isomorphic to $\RR\H^0(K_U,\omega_{P^+}\otimes_k\pholder)$, where
$\omega_{P^+}$ is a certain smooth character of $P^+$ sitting in cohomological
degree $-\dim U$, and $\RR\H^0(K_U,\pholder)$ is the right derived functor of
$\H^0(K_U,\pholder)\colon \Rep_k(P^+)\to \Rep_k(M^+)$ where  the $M^+$-action on each
$\H^0(K_U,V)$ is the ``Hecke action'' from~\cite{EmertonI}.
To summarize:

\begin{thmintro}[Corollary~\ref{cor:LU-explicit}]\label{thm:B} 
One has an isomorphism $\ind_{M^+}^M \RR\H^0(K_U, \omega_{P^+}\otimes_k\pholder)
\Res^{P}_{P^+} \cong \LL_{U}$ of functors $\Rep_k(P)\to \Rep_k(M)$.
\end{thmintro} 

An immediate consequence of Theorem~\ref{thm:B} is that $\LL_U$ restricts
to the bounded derived categories. 

Similar results trivially also follow for the left adjoint $\LL(U,\pholder)\coloneqq
\LL_U \circ\RR\Res^G_P$ of derived parabolic induction $\RR i_M^G$. It is
well-known that $i_M^G$ and the Jacquet-functor $\LL^0(U,\pholder)$
preserve admissible representations, and it is a natural question whether the
same still holds on the derived level. The correct notion of admissibility for
$\D(G)$ was introduced by Schneider--Sorensen in \cite{Schneider-Sorensen.2023a}:
A complex $X$ in $\D(G)$ is called \emph{globally admissible} if for some
torsion-free compact open pro-$p$ subgroup $K\subseteq G$ and all $i\in\Z$ the
cohomology groups $\H^i(K,X)$ are finite-dimensional $k$-vector spaces. They
then proceed to show that globally admissible complexes are precisely the
reflexive objects for a natural duality functor on $\D(G)$. In this direction,
we prove:

\begin{thmintro}[Theorem~\ref{thm:globally-admissible}]\label{thm:C} 
The functors $\RR i_M^G$ and $\LL(U,\pholder)$ preserve globally admissible complexes.
\end{thmintro} 

We finally compute $\LL^i(U,V)$ for the irreducible smooth
$\Fpbar[\GL_2(\Q_p)]$-representations $V$ and all $i\in\Z$, where $U$ consists
of the upper triangular unipotent matrices. For $i\in \{-1,0\}$ the
representations $\LL^i(U,V)$ are listed in Table~\ref{tab:irreducibles} on
page~\pageref{tab:irreducibles}, while for $i\notin\{-1,0\}$ these
representations are zero. 

The proof uses a derived variant of the Satake homomorphism, which I will now
describe. Let $K$ be a compact open subgroup of $G$ satisfying $G=PK$. Given
$V\in \D(K)$, there is a natural isomorphism
$\LL(U,\ind_K^GV) \cong \ind_{K\cap M}^M \LL(K\cap U,V)$, where we write
$\LL(K\cap U,\pholder) \coloneqq \LL_{K\cap U}\RR\Res^K_{K\cap P}$. Hence, the functor
$\LL(U,\pholder)$ induces a $k$-algebra homomorphism
\[
\Satake_V\colon \End_{\D(G)}(\ind_K^GV) \to
\End_{\D(M)}\bigl(\ind_{K\cap M}^M \LL(K\cap U,V)\bigr)
\]
which we call the derived Satake homomorphism.\footnote{In fact, $\Satake_V$ is
the $0$-th cohomology of the morphism induced by $\LL(U,\pholder)$ on
$\RHom$-complexes, but we will not make use of this generality.} 
In a different context, a version
of the derived Satake homomorphism was defined and studied by Ronchetti,
\cite{Ronchetti.2019}. The relation between these homomorphisms is still
unclear: The fact that $\LL(K\cap U,V)$ is a proper complex (even when $V =
\one$ is the trivial representation) makes a comparison difficult. However, it
is possible to relate $\Satake_V$ with the (underived) mod $p$ Satake
homomorphisms introduced by Herzig, \cite[Thm.~2.6 and \S2.3]{Herzig.2011b}.
Composing $\Satake_V$ with the $(-n)$-th cohomology functor $\H^{-n}$ yields a
$k$-algebra homomorphism
\[
\Satake_V^n\colon \End_{\Rep_k(G)}(\ind_K^GV) \to \End_{\Rep_k(M)}\bigl(
\ind_{K\cap M}^M \LL^{-n}(K\cap U,V)\bigr),
\]
and we prove:
\begin{thmintro}[Theorem~\ref{thm:Satake}]\label{thm:D} 
The mod $p$ Satake homomorphisms constructed in \cite{Herzig.2011b} coincide
with $\Satake_V^0$ and $\Satake_V^{\dim U}$. 
\end{thmintro} 

In particular, Theorem~\ref{thm:D} shows that the mod $p$ Satake homomorphisms
appearing in the literature are induced by explicit functors. For
$\Satake_V^0$ this is
essentially contained in \cite[Section~2]{Henniart-Vigneras.2012}, where
Henniart--Vign\'eras give a functorial description without, however, explicitly
mentioning the functor inducing $\Satake_V^0$. It was also known that
$\Satake_V^{\dim U}$ and $\Satake_V^0$ are related by taking duals, see
\cite[Lem.~2.11]{Herzig.2011b} or \cite[Prop.~2.4]{Henniart-Vigneras.2012}.

In the forthcoming work \cite{Heyer.2023} I use $\LL(U,\pholder)$ to prove a geometrical lemma \textit{\`a la} Bernstein--Zelevinsky and apply it to the
computation of some $\Ext$-groups of parabolically induced representations.

\subsection{Structure of the paper} 
In \S\ref{sec:foundation} we recall several fundamental concepts starting with a
recap of adjoint functors and the calculus of mates in \S\ref{subsec:category}.
In \S\ref{subsec:smooth} we introduce $p$-adic monoids and give a quick overview
of their smooth representation theory. The subject of \S\ref{subsec:derived} is
the derived category $\D(S)$ of smooth representations of a $p$-adic monoid $S$.
We recall Brown representability, the tensor triangulated structure on
$\D(S)$, and how to obtain spectral sequences using truncation
functors.\smallskip

In \S\ref{sec:derivedinf} we prove that the derived inflation functor
$\RR\Inf^T_S$ admits a left adjoint. We first prove this in the case where $T$
and $S$ are torsion-free compact $p$-adic Lie groups, see
\S\ref{subsec:compact}. The general case is then deduced from the compact case
in \S\ref{subsec:general}. In \S\ref{subsec:p-adic} and \S\ref{subsec:positive}
we specialize to $p$-adic Lie groups and show that $\LL_U$ satisfies a
projection formula and compute more explicit descriptions of the
$\LL_U^i$.\smallskip

In \S\ref{sec:Jacquet} we investigate the left adjoint of derived
parabolic induction. In \S\ref{subsec:global-adm} we show that $\RR i_M^G$ and
$\LL(U,\pholder)$ preserve global admissibility, and \S\ref{subsec:Satake} is concerned
with the derived Satake homomorphism.\smallskip

The explicit computations of $\LL^i(U,V)$, for irreducible smooth
$\Fpbar[\GL_2(\Q_p)]$-representations $V$, are the subject
of~\S\ref{sec:example}.

\subsection{Acknowledgments} 
First and foremost I thank Peter Schneider for his interest and support,
for some enlightening discussions, and for sharing with me the articles
\cite{Scherotzke-Schneider.2022} and \cite{Schneider-Sorensen.2023a}. I thank Jakob
Scholbach for some helpful suggestions and for pointing me to the article
\cite{Balmer.2016}. During the work on this project I have further profited from
discussions with Eugen Hellmann, Konstantin Ardakov, Emanuele Bodon, Yifei
Zhao, and Zhixiang Wu. I am most
grateful to Tobias Schmidt, for inviting me to speak at the Colloque Tournant
2021 du GDR TLAG, and to Claus Sorensen, for inviting me to present some of
these results at the UCSD Number Theory Seminar. Finally, I thank the referee 
for many helpful suggestions and corrections, and for pointing out to me the paper \cite{Chirvasitu-Kanda.2023}.
The project was funded by the Deutsche Forschungsgemeinschaft (DFG, German
Research Foundation) – Project-ID 427320536 – SFB 1442, as well as under
Germany's Excellence Strategy EXC 2044 390685587, Mathematics Münster:
Dynamics–Geometry–Structure.

\subsection{Notation and conventions} 
Throughout this article, $k$ will denote a coefficient field, which will be
assumed to be of characteristic $p>0$ from \S\ref{sec:derivedinf} onwards. The
category of $k$-vector spaces is denoted $\Vect_k$, and we write $\D(k)$ for the
unbounded derived category of $\Vect_k$.\smallskip

If $S$ is a locally profinite monoid, we denote $\Rep_k(S)$ the abelian category
of smooth $S$-re\-pre\-sen\-ta\-tions on $k$-vector spaces and by $\D(S)$ the
unbounded
derived category of $\Rep_k(S)$. Given $V\in \Rep_k(S)$, we denote by $V[0]$ (or
just $V$ if no confusion is likely) the representation $V$ viewed as a complex
concentrated in degree $0$.\smallskip

We make the convention that a diagram of functors, say,
\[
\begin{tikzcd}
\cat A \ar[r,"a"] \ar[d,"f"'] & \cat A' \ar[d,"f'"]\\
\cat B \ar[r,"b"'] & \cat B'
\end{tikzcd}
\]
commutes if there exists a natural isomorphism $b\circ f\xRightarrow{\cong}
f'\circ a$, which will usually be specified in the proof if it
is unclear which natural isomorphism is meant.\smallskip

If $G$ is a group, $K\subseteq G$ is a subgroup and $g\in G$, we write ${}^gK
\coloneqq gKg^{-1}$ and $K^g \coloneqq g^{-1}Kg$.

\section{Foundations}\label{sec:foundation} 
This section is devoted to collecting facts about 
adjoint functors (\S\ref{subsec:category}), smooth representations of locally
profinite monoids (\S\ref{subsec:smooth}), and the derived category of smooth
representations (\S\ref{subsec:derived}).
\subsection{About adjoint functors}\label{subsec:category} 
In this section we recall well-known facts about adjoint functors. For
references and more details we refer to \cite{KS} and \cite{Maclane}.

Let $\cat C, \cat D$ be categories. We fix two functors $L\colon \cat C\to \cat
D$ and $R\colon \cat D\to \cat C$.
\begin{defn} 
We say $L$ is \emph{left adjoint} to $R$ (or $R$ is \emph{right adjoint} to $L$),
written $L\dashv R$, if there is a bijection
\begin{equation}\label{eq:adjoint}
\Hom_{\cat D}\bigl(L(C), D\bigr) \xrightarrow\cong \Hom_{\cat C}\bigl(C,
R(D)\bigr)
\end{equation}
which is natural in both $C\in \cat C$ and $D\in \cat D$.
\end{defn} 

\begin{rmk} 
If $L\colon \cat C\to \cat D$ is an additive functor between additive
categories, then a right adjoint $R$ (if it exists) is additive as well and the
map \eqref{eq:adjoint} is an isomorphism of abelian groups, see \cite[IV.1,
Thm.~3]{Maclane}.
\end{rmk} 

It follows immediately from the definition that one can compose adjunctions:
\begin{lem}[{\cite[Prop.~1.5.5]{KS}}] 
\label{lem:adjoint-composition}
Consider the functors
\[
\begin{tikzcd}
\cat C \ar[r,shift left, "L"] & \ar[l,shift left,"R"] \cat D \ar[r,shift left,
"L'"] & \ar[l,shift left,"R'"] \cat E.
\end{tikzcd}
\]
Assume that $L$ and $L'$ are left adjoint to $R$ and $R'$, respectively. Then
$L'\circ L$ is left adjoint to $R\circ R'$.
\end{lem} 

\begin{lem}\label{lem:adjoint-defn} 
The following assertions are equivalent:
\begin{enumerate}[label=(\alph*)]
\item\label{lem:adjoint-defn-a} $L$ is left adjoint to $R$.
\item\label{lem:adjoint-defn-b} There exist natural transformations $\eta\colon
\id_{\cat C}\To RL$ and $\varepsilon\colon LR\To \id_{\cat D}$, called the
\emph{unit} and the \emph{counit} of the adjunction, respectively, satisfying
the ``triangle identities'' $R\varepsilon\circ \eta R = \id_R$ and $\varepsilon
L\circ L\eta = \id_L$.
\end{enumerate}
\end{lem} 
\begin{proof} 
For the details we refer to \cite[IV.1]{Maclane}.
If $L$ is left adjoint to $R$, the map $\eta_C$ is given by the image
of $\id_{L(C)}$ under \eqref{eq:adjoint} and $\varepsilon_D$ is given by the
preimage of $\id_{R(D)}$ under \eqref{eq:adjoint}. It is then easily checked
that $\eta$ and $\varepsilon$ are natural transformations and satisfy the
triangle identities. 

Conversely, if $\eta$ and $\varepsilon$ are as in \ref{lem:adjoint-defn-b} one
checks that the compositions
\begin{align*}
\Hom_{\cat D}\bigl(L(C),D\bigr) &\xrightarrow{R} \Hom_{\cat C}\bigl(RL(C),
R(D)\bigr) \xrightarrow{(\eta_C)^*} \Hom_{\cat C}\bigl(C, R(D)\bigr)\\
\intertext{and}
\Hom_{\cat C}\bigl(C,R(D)\bigr) &\xrightarrow{L} \Hom_{\cat D}\bigl(L(C),
LR(D)\bigr) \xrightarrow{(\varepsilon_D)_*} \Hom_{\cat D}\bigl(L(C),D\bigr)
\end{align*}
are inverse to each other.
\end{proof} 

\begin{lem}[{\cite[Prop.~1.5.6]{KS}}] 
\label{lem:adjoint-ffaithful}
Assume that $L$ is left adjoint to $R$.
\begin{enumerate}[label=(\alph*)]
\item The functor $R$ is fully faithful if and only if the counit $\varepsilon
\colon LR\To \id_{\cat D}$ is an isomorphism.
\item The functor $L$ is fully faithful if and only if the unit $\eta\colon
\id_{\cat C} \To RL$ is an isomorphism.
\end{enumerate}
\end{lem} 

The following result shows that adjoints are uniquely determined.

\begin{thm}[{\cite[Thm.~1.5.3]{KS}}] 
\label{thm:adjoint-unique}
The right adjoint of $L$ (if it exists) is unique up to unique isomorphism.
Moreover, $L$ admits a right adjoint if and only if the functor $\Hom_{\cat
D}\bigl(L(\pholder), D\bigr) \colon \cat D\to \Sets$ is representable for every $D\in
\cat D$.

Dually, the left adjoint of $R$ (if it exists) is unique up to unique
isomorphism. Moreover, $R$ admits a left adjoint if and only if the functor
$\Hom_{\cat C}\bigl(C,R(\pholder)\bigr)\colon \cat C\to \Sets$ is representable for
every $C\in \cat C$.
\end{thm} 

\begin{prop}[{\cite[Prop.~2.1.10]{KS}}] 
\label{prop:adjoint-limit}
If $L\colon \cat C\to \cat D$ admits a right adjoint, then $L$ commutes with all
colimits that exist in $\cat C$.

Dually, if $R\colon \cat D\to \cat C$ admits a left adjoint, then $R$ commutes
with all limits that exist in $\cat D$.
\end{prop} 

For a stronger (and dual) version of the next extremely useful result, see
\cite[II, Prop.~10.2]{Hilton.1971}.
\begin{lem}\label{lem:adjoint-injective} 
Assume that $\cat C$ and $\cat D$ are abelian categories and $R\colon \cat D\to
\cat C$ is an additive functor admitting an exact left adjoint. Then $R$
preserves injective objects.
\end{lem} 
\begin{proof} 
Let $I \in \cat D$ be an injective object. The functor $\Hom_{\cat C} 
\bigl(\pholder,R(I)\bigr) \cong \Hom_{\cat D}\bigl(L(\pholder),I\bigr)$ is exact,
which means that $R(I)\in \cat C$ is injective.
\end{proof} 

Finally, we discuss the formalism of ``mates'' of natural transformations
involving adjoint functors. As usual, given two functors $F,G\colon \cat C\to
\cat D$, we denote by $\Nat(F,G)$ the class of natural transformations from $F$ to
$G$.

\begin{prop}[{\cite[Prop.~2.1]{Kelly-Street.1974}}]\label{prop:mates} 
Consider a diagram of functors
\begin{equation}\label{eq:mates}
\begin{tikzcd}[column sep= 6em, row sep=4em, baseline=(current bounding
box.center)]
\cat C \ar[d,"F"'] \ar[r,bend left=20, "L", ""{name=L,below}]& \cat D \ar[d,"G"]
\ar[l,bend left=20, "R", ""{name=R,above}]\\
\arrow[from=L, to=R, symbol=\dashv]
\cat C' \ar[r,bend left=20, "L'", ""{name=L', below}] & \ar[l, bend left=20,
"R'", ""{name=R',above}] \cat D'
\arrow[from=L', to=R', symbol=\dashv]
\end{tikzcd} 
\end{equation}
and denote $\eta\colon \id_{\cat C} \To RL$, $\varepsilon\colon LR\To \id_{\cat D}$
and $\eta'\colon \id_{\cat C'} \To R'L'$, $\varepsilon'\colon L'R'\To \id_{\cat D'}$
the units and counits. There is a natural bijection
\begin{equation}\label{eq:mates-bijection}
\begin{tikzcd}[row sep=0em, baseline=(current bounding box.center)]
\Nat(L'F, GL) \ar[r,leftrightarrow, "\cong"] & \Nat(FR, R'G),\\
\alpha \ar[r,mapsto] & (R'G\varepsilon)\circ (R'\alpha_R) \circ (\eta'_{FR}),\\
(\varepsilon'_{GL})\circ (L'\beta_L)\circ (L'F\eta) & \ar[l,mapsto] \beta.
\end{tikzcd}
\end{equation}
\end{prop} 

If $\alpha \colon L'F\To GL$ corresponds to $\beta\colon FR \To R'G$ under
\eqref{eq:mates-bijection}, we say that $\alpha$ is the left mate of $\beta$ and
$\beta$ is the right mate of $\alpha$.

The ``naturality'' above means that the bijection \eqref{eq:mates-bijection} is
compatible with horizontal and vertical composition of squares of the
form \eqref{eq:mates}, see \cite[Prop.~2.2]{Kelly-Street.1974} for a precise
statement. 

\begin{ex}\label{ex:mates-vertical} 
Let us make explicit two special cases of vertical composition. Suppose we are
in the setting of Proposition~\ref{prop:mates}. 
\begin{enumerate}[label=(\roman*)]
\item Let $L_1\colon \cat C\to \cat D$ be a functor with right adjoint $R_1$, and
let $\alpha\colon L\To L_1$ be a natural transformation with right mate
$\beta\colon R_1\To R$. Then the diagram
\[
\begin{tikzcd}
\Nat(L'F, GL) \ar[r,leftrightarrow,"\cong"] \ar[d, "(G\alpha)_*"'] & 
\Nat(FR, R'G) \ar[d,"(F\beta)^*"]\\
\Nat(L'F,GL_1) \ar[r,leftrightarrow,"\cong"'] & \Nat(FR_1,R'G)
\end{tikzcd}
\]
commutes.

\item Let $L'_1\colon \cat C'\to \cat D'$ be a functor with right adjoint
$R'_1$, and let $\alpha'\colon L'_1\To L'$ be a natural transformation with
right mate $\beta'\colon R'\To R'_1$. Then the diagram
\[
\begin{tikzcd}
\Nat(L'F,GL) \ar[r,leftrightarrow,"\cong"] \ar[d,"(\alpha'_F)^*"'] &
\Nat(FR, R'G) \ar[d,"(\beta'_G)_*"]\\
\Nat(L'_1F,GL) \ar[r,leftrightarrow,"\cong"'] & \Nat(FR, R'_1G)
\end{tikzcd}
\]
commutes.
\end{enumerate}
\end{ex} 

It is not true in general that \eqref{eq:mates-bijection} preserves
isomorphisms. However, in very special cases this does hold.

\begin{ex}\label{ex:mates} 
Let $L_i\colon \cat C\to \cat D$ be a functor admitting a right adjoint $R_i$,
for $i=1,2,3$. Let $\alpha\colon L_3\To L_2$ be a natural transformation with
mate $\beta\colon R_2\To R_3$. It follows from Example~\ref{ex:mates-vertical}
that the diagram 
\[
\begin{tikzcd}
\Nat(L_2,L_1) \ar[r,leftrightarrow, "\cong"] \ar[d,"\alpha^*"'] & \Nat(R_1,R_2)
\ar[d,"\beta_*"]\\
\Nat(L_3,L_1) \ar[r,leftrightarrow,"\cong"'] & \Nat(R_1,R_3)
\end{tikzcd}
\]
commutes.
Hence, if $\alpha'\colon L_2\To L_1$ has mate $\beta'\colon R_1\To R_2$, then
the mate of $\alpha'\circ\alpha$ is $\beta\circ\beta'$. Note also that under
$\Nat(L_3,L_3) \xleftrightarrow{\hphantom{i}\cong\hphantom{i}} \Nat(R_3,R_3)$
the mate of $\id_{L_3}$ is $\id_{R_3}$; this is a reformulation of the fact
that the unit and counit of the adjunction $L_3\dashv R_3$ satisfy the triangle
identities.

It follows from this discussion that $\alpha\colon L_3\To L_2$ is an isomorphism
if and only if its mate $\beta\colon R_2\To R_3$ is an isomorphism. This is a
strengthening of the uniqueness statement in Theorem~\ref{thm:adjoint-unique}.
\end{ex}

\subsection{Smooth representations}\label{subsec:smooth} 
For our purposes it is useful to consider smooth representations over locally
profinite monoids. Everything in this section is well-known for locally
profinite groups, see, \eg, \cite{Bushnell-Henniart.2006} or
\cite{Vigneras.1996}, and so we will restrain ourselves to a very brief
treatment.
We fix a coefficient field $k$.

\subsubsection{General properties of smooth representations} 
\begin{defn} 
A \emph{locally profinite monoid} is an open submonoid of a locally profinite group. 
A \emph{$p$-adic monoid} is an open submonoid of a $p$-adic Lie group. 
A morphism between locally profinite monoids (resp.\ $p$-adic monoids) is a continuous
monoid homomorphism.
\end{defn} 

\begin{ex}\phantomsection\label{ex:p-adic-monoid} 
Let $\field$ be a finite extension of $\Q_p$ and consider a connected
reductive group $\alg G$ defined over $\field$. Then every closed subgroup of
$\alg G(\field)$ is a $p$-adic Lie group \cite[9.6 Thm.]{DDMS99}. Let $\alg P =
\alg M\alg U$ be a parabolic subgroup of $\alg G$. Fix a compact open subgroup
$K_P\subseteq \alg P(\field)$ such that $K_P = K_MK_U$, where $K_M \coloneqq
K_P\cap \alg M(\field)$ and $K_U\coloneqq K_P\cap \alg U(\field)$. Then
\[
M^+\coloneqq \set{m\in \alg M(\field)}{mK_Um^{-1} \subseteq K_U}
\]
contains $K_M$ and hence is a $p$-adic monoid. 

The topological monoid $P^+\coloneqq K_UM^+$ contains $K$ and is therefore a
$p$-adic monoid.\footnote{Observe, however, that $M^+K_U$ is \emph{not} a
monoid.} Note that there is a surjective monoid homomorphism
$P^+\longtwoheadrightarrow M^+$ with kernel $K_U$.
\end{ex} 

\begin{lem}\label{lem:opensubmon} 
Let $S$ be a locally profinite monoid and $S' \subseteq S$ a submonoid. Then
$S'$ is open if and only if $S'$ contains an open subgroup of $S$.
\end{lem} 
\begin{proof} 
One direction follows from the fact that the neutral element in a $p$-adic Lie group admits a fundamental system of open neighborhoods which are profinite subgroups. The other direction is clear.
\end{proof} 

\begin{defn} 
Let $S$ be a locally profinite monoid. Let $V$ be a representation of $S$ on a
$k$-vector space, \ie, a (left) module over the monoid algebra $k[S]$. 
\begin{enumerate}[label=(\roman*)]
\item A vector $v\in V$ is called \emph{smooth} if the stabilizer $\set{s\in
S}{sv=v}$ of $v$ is open in $S$. We denote by $V^\sm\subseteq V$ the subset of
smooth vectors in $V$.
\item $V$ is called \emph{smooth} if $V = V^\sm$.
\end{enumerate}

We denote by $\Rep_k(S)$ the full subcategory of the category $\Mod(k[S])$ of
$k[S]$-modules consisting of smooth representations of $S$.
\end{defn} 

\begin{lem}\label{lem:smooth} 
Let $S$ be a locally profinite monoid and let $V$ be a $k[S]$-module.
\begin{enumerate}[label=(\alph*)]
\item\label{lem:smooth-a} Let $v\in V$. The following assertions are
equivalent:
\begin{enumerate}[label=(\roman*)]
\item\label{lem:smooth-a-i} The vector $v\in V$ is smooth.
\item\label{lem:smooth-a-ii} There exists an open subgroup $S_0\subseteq S$
fixing $v$.
\end{enumerate}
\item\label{lem:smooth-b} $V^{\sm}$ is a smooth $S$-subrepresentation of $V$.
\end{enumerate}
\end{lem} 
\begin{proof} 
The equivalence in \ref{lem:smooth-a} follows immediately from
Lemma~\ref{lem:opensubmon}. We now prove \ref{lem:smooth-b}. Let $v,w\in V^\sm$
and $s\in S$. It suffices to show $v+w\in V^\sm$ and $sv\in V^\sm$.
By \ref{lem:smooth-a} there exist open subgroups $S_v$ and $S_w$ fixing $v$ and
$w$, respectively. Then $S_v\cap S_w$ fixes $v+w$, and then \ref{lem:smooth-a}
implies $v+w\in V^\sm$. Since multiplication in $S$ is continuous and $sS_v$ is
open in $S$, there exists an open subgroup $S'_v\subseteq S$ such that $S'_vs
\subseteq sS_v$. But then $S'_v$ fixes $sv$ which, by \ref{lem:smooth-a} again,
shows $sv\in V^\sm$. 
\end{proof} 

We list some immediate consequences of Lemma~\ref{lem:smooth}.
\begin{cor}\label{cor:abelian} 
Let $S$ be a locally profinite monoid.
\begin{enumerate}[label=(\alph*)]
\item A $k[S]$-module $V$ is smooth if and only if $V$ is smooth when
considered as a $k[S_0]$-module for some open profinite subgroup $S_0\subseteq
S$. 
\item $\Rep_k(S)$ is an abelian subcategory of $\Mod(k[S])$.
\end{enumerate}
\end{cor} 

Let $S'\subseteq S$ be a closed submonoid of a locally profinite monoid $S$.
The restriction functor 
\[
\Res^S_{S'}\colon \Rep_k(S) \to \Rep_k(S')
\]
is exact. It admits a right adjoint which we will now describe. Given a smooth
$k[S']$-module $V$, we obtain a smooth $k[S]$-module
\[
\Ind_{S'}^SV\coloneqq \Hom_{k[S']}\bigl(k[S], V\bigr)^\sm,
\]
called the \emph{smooth induction} of $V$.
Here, $S$ acts on $\Ind_{S'}^SV$ via right multiplication on $k[S]$. One easily
verifies that 
\[
\Ind_{S'}^SV \cong \set{f\colon S\to V}{\begin{array}{l}\text{$f(s's) = s'f(s)$
for all $s\in S$, $s'\in S'$,}\\
\text{$\exists K\subseteq S$ open with $f(sx) = f(s)$ for all $s\in S$, $x\in
K$}\end{array}},
\]
where $S$ acts on $f\colon S\to V$ via $(sf)(t) = f(ts)$, for all $s,t\in S$. We
will often implicitly make this identification.

\begin{lem}[Frobenius reciprocity I] 
\label{lem:FrobeniusI}
The smooth induction $\Ind_{S'}^S\colon \Rep_k(S')\to \Rep_k(S)$ is a right
adjoint of $\Res^{S}_{S'}$, that is, there is a natural bijection
\[
\Hom_{\Rep_k(S')}\bigl(\Res^S_{S'}V, W\bigr) \cong
\Hom_{\Rep_k(S)} \bigl(V, \Ind_{S'}^SW\bigr)
\]
for all $V\in \Rep_k(S)$, $W\in \Rep_k(S')$. In particular, $\Ind_{S'}^S$ is
left exact and preserves injective objects.
\end{lem} 
\begin{proof} 
The last assertion is a formal consequence of the fact that $\Ind_{S'}^S$ admits
an exact left adjoint, see Lemma~\ref{lem:adjoint-injective}.
\end{proof} 

\begin{rmk}\label{rmk:Ind-exact} 
If $S$ and $S'$ are profinite groups, the projection $S\to S/S'$ admits a
continuous section by \cite[Ch.~I, Prop.~1]{Serre.2013}. Hence,
$\Ind_{S'}^S\colon \Rep_k(S')\to \Rep_k(S)$ is exact by \cite[Prop.~5]{Modrep}.
\end{rmk} 

\begin{cor}\label{cor:enoughinj} 
The category $\Rep_k(S)$ has enough injectives.
\end{cor} 
\begin{proof} 
Let $V\in \Rep_k(S)$. The map $\varphi\colon V\to \Ind_{1}^S\Res^S_1V$,
$\varphi(v)(x) \coloneqq xv$, is clearly injective. Since $\Res^S_1V$ is
injective in the category of $k$-vector spaces, it follows from
Lemma~\ref{lem:FrobeniusI} that $\Ind_1^S\Res^S_1V$ is injective in
$\Rep_k(S)$.
\end{proof} 

If $S'\subseteq S$ is open, then $\Res^S_{S'}$ also admits a left
adjoint. For a smooth $k[S']$-module $V$ we define
\[
\ind_{S'}^SV \coloneqq k[S]\otimes_{k[S']} V,
\]
called the \emph{compact induction} of $V$. The smoothness of $\ind_{S'}^SV$
follows from Lemma~\ref{lem:smooth}.\ref{lem:smooth-b} and the fact that
$\ind_{S'}^SV$ is generated as a $k[S]$-module by $V$, which consists of smooth
vectors. 

\begin{rmk} 
Suppose that $S'\subseteq S$ is an open subgroup. Given $s\in S$ and $v\in V$, we
define $[s,v]\in \Ind_{S'}^S V$ by
\[
[s,v](t)\coloneqq \begin{cases}ts\cdot v, & \text{if $ts\in S'$,}\\ 0, &
\text{otherwise.} \end{cases}
\]
In other words, $[s,v]$ is the $S'$-equivariant function $S\to V$ supported on
$S's^{-1}$ with
value $v$ at $s^{-1}$.
In this case we have a canonical injective $S$-equivariant map
\begin{align*}
\ind_{S'}^SV &\to \Ind_{S'}^SV,\\
s\otimes v &\longmapsto [s,v].
\end{align*}
It is customary to identify $\ind_{S'}^SV$ with its image in
$\Ind_{S'}^SV$.
\end{rmk} 

\begin{lem}[Frobenius reciprocity II] 
\label{lem:FrobeniusII}
Assume that $S'\subseteq S$ is an open submonoid. Then compact induction
$\ind_{S'}^S\colon \Rep_k(S')\to \Rep_k(S)$ is left adjoint to
$\Res^S_{S'}$, that is, there is a natural bijection
\[
\Hom_{\Rep_k(S)}\bigl(\ind_{S'}^SV,W\bigr) \cong \Hom_{\Rep_k(S')} \bigl(V,
\Res^S_{S'}W\bigr),
\]
for all $V\in \Rep_k(S')$, $W\in \Rep_k(S)$. In particular, $\ind_{S'}^S$
is right exact.
\end{lem} 
\begin{proof} 
This follows from standard properties of the tensor product.
\end{proof} 

\begin{rmk*}
I do not know whether the extension $k[S']\subseteq k[S]$ is flat in general. But it is flat (and hence $\ind_{S'}^S$ is exact) in the following situations:
\begin{enumerate}[label=(\roman*)]
\item $S'$ is a group; in this case, $k[S]$ is even free over $k[S']$ since, by definition, $S$ is contained in a group.
\item In the context of Example~\ref{ex:p-adic-monoid}, the extensions $k[P^+]\subseteq k[P]$ and $k[M^+] \subseteq k[M]$ are flat (Lemma~\ref{lem:M+-flat}).
\end{enumerate}
\end{rmk*}

\begin{cor}\label{cor:FrobeniusII} 
Let $S'\subseteq S$ be an open submonoid. For all $V\in \Rep_k(S)$ the map
\begin{align*}
\Hom_{\Rep_k(S)}\bigl(\ind_{S'}^S\one, V\bigr) &\xrightarrow{\cong}
\H^0(S',V)\coloneqq \set{v\in V}{\text{$sv = v$ for all $s\in S'$}},\\
\varphi &\longmapsto \varphi(1\otimes 1)
\end{align*}
is a natural bijection. In particular, $\set{\ind_{S'}^S\one}{\text{$S'\subseteq
S$ open}}$ is a set of generators for $\Rep_k(S)$.
\end{cor} 
\begin{proof} 
For the definition of a set of generators we refer to
\cite[Def.~5.2.1]{KS}.
The first assertion is an easy consequence of Lemma~\ref{lem:FrobeniusII}.
Therefore, in order to prove the last statement, it suffices to prove that a
morphism $V\to W$ in $\Rep_k(S)$ is an isomorphism if the induced maps
$V^{S'}\to W^{S'}$ are bijective, for all open submonoids $S'$. But this
follows from the fact that $V$ and $W$ are smooth.
\end{proof} 

\begin{prop}\label{prop:Grothendieck} 
The category $\Rep_k(S)$ is Grothendieck abelian.
\end{prop} 
\begin{proof} 
By Corollaries~\ref{cor:abelian} and~\ref{cor:FrobeniusII} we know that
$\Rep_k(S)$ is abelian and admits a set of generators. Given a family
$\{V_i\}_{i\in I}$ of smooth representations, the direct sum $\bigoplus_{i\in
I}V_i$ is again smooth. It follows that $\Rep_k(S)$ has all colimits and that
they can be computed in $\Mod(k[S])$. Therefore, all filtered colimits are
exact, \ie, $\Rep_k(S)$ satisfies the Grothendieck axiom AB5. Moreover, the
representation $\bigl(\prod_{i\in I}V_i\bigr)^{\sm}$ is a categorical product in
$\Rep_k(S)$, so that $\Rep_k(S)$ satisfies the axiom AB3*.
\end{proof} 

\begin{rmk} 
If $\Char k$ divides the pro-order of every open profinite subgroup of $S$, the
formation of infinite direct products is not exact in $\Rep_k(S)$. As an
example, consider $S = \Z_p$ and $k = \F_p$. For each $n\ge1$ the map
$f_n\colon k[\Z_p/p^n\Z_p]\to k$, $x+(p^n)\mapsto 1$, is a surjection between
smooth $\Z_p$-representations, where $\Z_p$ acts trivially on $k$. Note that,
for $m< n$, the induced map $\H^0(p^m\Z_p, k[\Z_p/p^n\Z_p]) \to k$ on
$p^m\Z_p$-invariants is zero (this uses $\Char k = p$). Therefore, the map
\[
\Bigl(\prod_{n\ge1}k[\Z_p/p^n\Z_p]\Bigr)^\sm \xrightarrow{(f_n)_n}
\Bigl(\prod_{n\ge1}k\Bigr)^\sm = \prod_{n\ge1}k
\]
takes values in $\bigoplus_{n\ge1}k$ and hence is not surjective.

The failure of having exact direct products means that
$\Rep_k(S)$ does not have enough projective objects, see
\cite[Thm.~1.3]{Roos.2006}. (In fact, it is conjectured that $\Rep_k(S)$ has no non-zero projectives at all. For profinite groups this is resolved in \cite[Thm~3.1]{Chirvasitu-Kanda.2023}.) This is the source of many difficulties in the
theory of smooth mod $p$ representations of $p$-adic Lie groups.
\end{rmk}

\subsubsection{Inflation} 
Let $f\colon S\longtwoheadrightarrow T$ be a surjective open morphism of
locally profinite monoids with kernel $U$. Inflation along $f$ defines an
exact functor $\Inf^T_S\colon \Rep_k(T)\to \Rep_k(S)$. 

\begin{prop}\label{prop:Inf-adjoints} 
$\Inf^T_S$ admits a left and a right adjoint. More concretely:
\begin{enumerate}[label=(\alph*)]
\item\label{prop:Inf-adjoints-a} The functor
\begin{align*}
\LL^0_U\colon \Rep_k(S) &\to \Rep_k(T), \\
V &\longmapsto \LL^0_U(V)\coloneqq k[T]\otimes_{k[S]}V
\end{align*}
is a left adjoint of $\Inf^T_S$ and, in particular, right exact.
\item\label{prop:Inf-adjoints-b} The functor
\begin{align*}
\Pi_U\colon \Rep_k(S) &\to \Rep_k(T),\\
V&\longmapsto \Pi_U(V) \coloneqq \Hom_{k[S]}(k[T],V)^{\sm}
\end{align*}
is a right adjoint of $\Inf^T_S$ and, in particular, left exact.
\end{enumerate}
\end{prop} 
\begin{proof} 
\begin{itemize}
\item[\ref{prop:Inf-adjoints-a}] We write $\Inf^T_SW = \Hom_{k[T]}(k[T],W)$,
where $S$ acts on the right on $k[T]$. The assertion follows from the
$\otimes$-$\Hom$ adjunction, \cite[Lem.~2.8.2]{Benson.1991a}:
\[
\Hom_{k[S]}\bigl(V, \Hom_{k[T]}(k[T],W)\bigr) \cong \Hom_{k[T]}\bigl(k[T]
\otimes_{k[S]}V, W\bigr).
\]
\item[\ref{prop:Inf-adjoints-b}] We write $\Inf^T_SW = k[T]\otimes_{k[T]}W$,
where $S$ acts on the left on $k[T]$. By \loccit\ we have
\begin{align*}
\Hom_{k[S]}\bigl(k[T]\otimes_{k[T]}W,V\bigr) &\cong \Hom_{k[T]}\bigl(W,
\Hom_{k[S]}(k[T],V)\bigr)\\
&\cong \Hom_{k[T]}\bigl(W, \Hom_{k[S]}(k[T],V)^\sm\bigr),
\end{align*}
where the second bijection comes from the fact that every homomorphic image of a
smooth representation is smooth.
\end{itemize}
\end{proof} 

\begin{rmk}\label{rmk:adjoints-monoid} 
The description of the left and right adjoints of $\Inf^T_S$ is quite different
from the more familiar description in the case where $S$ (and hence $T$, $U$) is
a group. It is instructive to make precise the distinction between both
descriptions.
\begin{enumerate}[label=(\roman*)]
\item Given $V\in \Rep_k(S)$, put $V(U) \coloneqq \spann_k\{uv-v\mid u\in U, v\in
V\}$, where $\spann_k\{\ldots\}$ denotes the $k$-linear span. There
is a canonical $k$-linear surjection
\begin{equation}\label{eq:V_U}
V_U\coloneqq V/V(U) \longtwoheadrightarrow k[T]\otimes_{k[S]}V,
\end{equation}
given by mapping the coset of $v$ to $1\otimes v$. Note that, if $sU\subseteq
Us$ for all $s\in S$, then there is a natural $k[S]$-action on $V_U$. If
even $f^{-1}(f(s)) = Us$ for all $s\in S$, then this action induces a
$k[T]$-action on $V_U$. In this case, the kernel of $k[S]\longtwoheadrightarrow
k[T]$ is generated as a right ideal by $\set{u-1}{u\in U}$, and
hence \eqref{eq:V_U} is an isomorphism of smooth $T$-representations.

\item Evaluation at $1\in k[T]$ induces an injective $k$-linear map
\[
\ev_1\colon \Pi_U(V) \longhookrightarrow \H^0(U,V).
\]
We observe that, if $Us\subseteq sU$ for all $s\in S$, then there is a natural
$k[S]$-action on $\H^0(U,V)$. If even $f^{-1}(f(s)) = sU$ for all $s\in S$, then
this action induces a $k[T]$-action on $\H^0(U,V)$. In this case, the kernel of
$k[S]\longtwoheadrightarrow k[T]$ is generated as a left ideal by
$\set{u-1}{u\in U}$, and hence $\ev_1$ is an isomorphism of smooth
$T$-representations.
\end{enumerate}
\end{rmk} 

\begin{ex} 
In order to see where some of the above conditions may fail, let us consider the
positive monoid, cf.~Example~\ref{ex:p-adic-monoid}.
More concretely,
let $P \subseteq \GL_2(\Q_p)$ be the subgroup of upper triangular matrices, $M$
the subgroup of diagonal matrices, and $U$ the subgroup of upper triangular
unipotent matrices. Put $K_P = P\cap \GL_2(\Z_p)$, $K_M = M\cap K_P$, and $K_U =
U\cap K_P$. Consider now the positive monoid $M^+ = \set{m\in M}{mK_Um^{-1}
\subseteq K_U}$ and put $P^+ = K_UM^+$. Let $f\colon P^+\to M^+$ be the
canonical projection with kernel $K_U$. 

It is obvious that $f^{-1}(m) = K_Um$,
for all $m\in M^+$, and hence $V_{K_U}\cong \LL^0_{K_U}(V)$ as smooth
$M^+$-representations, for all $V\in \Rep_k(P^+)$.

However, for $m\coloneqq \diag(p,1) \in M^+$, we have $K_Um \not\subseteq
mK_U$. Therefore $\H^0(K_U,V)$ is not stable under the $M^+$-action and
$\ev_1\colon \Pi_{K_U}(V)\longhookrightarrow \H^0(K_U,V)$ is not surjective. We
remark that $\H^0(K_U,V)$ can be endowed with a ``Hecke action'',
cf.~\cite[Lem.~3.1.4]{EmertonI}, making it a smooth $M^+$-representation.
\end{ex}

\subsubsection{Symmetric monoidal structure} 
Given $V,W$ in $\Rep_k(S)$, we let $S$ act diagonally on the tensor product
$V\otimes_kW$; this is again a smooth representation. Therefore, $\Rep_k(S)$
carries the structure of a symmetric monoidal category, where the $\otimes$-unit
is given by the trivial representation $k$.

\begin{prop}\label{prop:innerHom} 
The symmetric monoidal category $\Rep_k(S)$ is closed, that is, for each $V$ in
$\Rep_k(S)$ the functor
\[
\pholder\otimes_kV\colon \Rep_k(S)\to \Rep_k(S)
\]
admits a right adjoint $\hom_{\Rep_k(S)}(V,\pholder)\colon \Rep_k(S)\to \Rep_k(S)$.
Moreover, there is a natural $S$-equivariant isomorphism
\[
\hom_{\Rep_k(S)}\bigl(U\otimes_kV, W\bigr) \cong \hom_{\Rep_k(S)}\bigl(U,
\hom_{\Rep_k(S)}(V,W)\bigr).
\]
\end{prop} 
\begin{proof} 
Let $V,W\in \Rep_k(S)$. We let $S$ act diagonally on $k[S]\otimes_kV$ and then on
the $k$-vector space $\Hom_{k[S]}\bigl(k[S]\otimes_kV,
W\bigr)$ via $(sf)(s'\otimes v) \coloneqq f(s's\otimes v)$. One easily verifies
that the maps
\[
\begin{tikzcd}[row sep=0em]
\Hom_{\Rep_k(S)}\bigl(U\otimes_kV, W\bigr) \ar[r,shift left] & \ar[l,shift left]
\Hom_{\Rep_k(S)}\bigl(U, \Hom_{k[S]}(k[S]\otimes_kV,W)^{\sm}\bigr),\\
\varphi \ar[r,mapsto] & {[u\mapsto [s\otimes v\mapsto \varphi(su\otimes v)]]}\\
{[u\otimes v\mapsto \psi(u)(1\otimes v)]} &\ar[l,mapsto] \psi
\end{tikzcd}
\]
are natural and inverse to each other. Hence,
$\hom_{\Rep_k(S)}(V,\pholder)\coloneqq \Hom_{k[S]}(k[S]\otimes_kV,\pholder)^\sm$ is right
adjoint to $\pholder\otimes_kV$. Let $X,U,V,W \in \Rep_k(S)$. Now, there are natural
bijections
\begin{align*}
\Hom_{\Rep_k(S)}\bigl(X, \hom_{\Rep_k(S)}(U\otimes_kV,W)\bigr) &\cong
\Hom_{\Rep_k(S)}\bigl(X\otimes_kU\otimes_kV, W\bigr)\\
&\cong \Hom_{\Rep_k(S)}\bigl(X\otimes_kU, \hom_{\Rep_k(S)}(V,W)\bigr)\\
&\cong \Hom_{\Rep_k(S)}\bigl(X, \hom_{\Rep_k(S)}(U,
\hom_{\Rep_k(S)}(V,W))\bigr).
\end{align*}
The last claim follows from this together with the Yoneda lemma.
\end{proof} 

\begin{rmk}\label{rmk:internal-hom-group} 
If $S$ is a group, the right adjoint of $\pholder\otimes_kV$ has a more familiar
description: The map $\sigma\colon k[S]\otimes_kV\to k[S]\otimes_kV$, $s\otimes
v\longmapsto s\otimes s^{-1}v$ is an isomorphism of $k[S]$-modules if we let $S$
act diagonally on the left hand side and on the first factor on the right hand
side. 
There is a natural $k[S]$-linear isomorphism
\[
\Hom_k(V,W)\cong \Hom_{k[S]}\bigl(k[S]\otimes_kV,W\bigr)
\xrightarrow[\sigma^*]{\cong} \Hom_{k[S]}\bigl(k[S]\otimes_kV,W\bigr)
\]
if we let $S$ act on the left hand side via $(sf)(v) = s\cdot f(s^{-1}v)$.
Therefore, the right adjoint of $\pholder\otimes_kV$ is in this case given by
$\Hom_k(V,\pholder)^\sm$.
\end{rmk} 

\subsection{The derived category}\label{subsec:derived} 
\subsubsection{Derived functors}  
Let $S$ be a locally profinite monoid. We denote by $\Chain(S)$ the category of
unbounded cochain complexes of smooth $S$-representations and by $\K(S)$ its
homotopy category. Then $\K(S)$ carries the structure of a triangulated
category \cite[Prop.~11.2.8]{KS}. We denote by 
\[
\D(S)\coloneqq \D\bigl(\Rep_k(S)\bigr)
\]
the derived category of unbounded complexes in $\Rep_k(S)$. It arises from
$\K(S)$ by localization at the class of quasi-isomorphisms
\cite[\S13.1]{KS}. Note that $\D(S)$ inherits the structure of a triangulated
category from $\K(S)$ such that the localization functor 
\[
\q\colon \K(S)\to \D(S)
\]
is triangulated. The shift of a complex $C$ (in $\K(S)$ or $\D(S)$) will be
denoted $C[1]$, defined by $(C[1])^n\coloneqq C^{n+1}$ and with differential
$d_{C[1]} = -d_C$. We recall the most important properties of $\D(S)$, which
are
summarized in \cite[Thm.~14.3.1]{KS}. The derived category $\D(S)$ admits
infinite direct sums, and they can be computed on complexes. By the main result
of \cite[Thm.~3.13]{Serpe.2003} or by \cite[Cor.~14.1.8]{KS} every
complex $X$ in $\K(S)$ admits a quasi-isomorphism $X\xrightarrow{\qis} I$ into a
K-injective complex $I$. (Recall that a complex $J$ in $\K(S)$ is called
\emph{K-injective} if $\Hom_{\K(S)}(Y,J) = 0$ for all acyclic complexes $Y$.)
This amounts to saying that $\q$ admits a fully faithful triangulated right
adjoint
\[
\ii\colon \D(S) \to \K(S)
\]
which identifies $\D(S)$ with the full subcategory of $\K(S)$ consisting of
K-injective complexes. In particular, $\D(S)$ is locally small. It also follows
that right derived functors always exist:

\begin{prop}\label{prop:derivedfunctor} 
Let $\cat A$ be an abelian category and $F\colon \Rep_k(S)\to \cat A$ a left
exact functor. A right derived functor $\RR F\colon \D(S)\to \D(\cat A)$
exists and is given by the composition
\[
\RR F\colon \D(S) \xrightarrow{\ii} \K(S) \xrightarrow{\K(F)} \K(\cat A)
\xrightarrow{\q_{\cat A}} \D(\cat A),
\]
where $\q_{\cat A}$ denotes the localization functor and $\K(F)$ the natural
extension of $F$ to the homotopy categories.
\end{prop} 

\begin{notation}\label{nota:derivedfunctor} 
If $F\colon \Rep_k(S)\to \cat A$ happens to be exact, we simply write $F\colon
\D(S)\to \D(\cat A)$ for the right derived functor of $F$. This is justified by
the fact that in this case the right adjoint can be computed by applying $F$
termwise to the complex.
\end{notation} 

Here is a useful and well-known lemma:
\begin{lem}\label{lem:K-injectives} 
Let $L\colon \cat A\to \cat B$ be an exact functor between abelian categories
admitting a right adjoint $R$.
\begin{enumerate}[label=(\alph*)]
\item\label{lem:K-injectives-a} $\K(R)\colon \K(\cat B)\to \K(\cat A)$ preserves
K-injective complexes.

\item\label{lem:K-injectives-b} Assume moreover that the localization functors
$\q_{\cat A}\colon \K(\cat A)\to \D(\cat A)$ and $\q_{\cat B}\colon \K(\cat
B)\to \D(\cat B)$ admit right adjoints $\ii_{\cat A}$ and $\ii_{\cat
B}$, respectively. Then there is
a natural isomorphism $\K(R)\ii_{\cat B} \xRightarrow{\cong}  \ii_{\cat A}\RR R$
of functors $\D(\cat B)\to \K(\cat A)$.

\item\label{lem:K-injectives-c} Under the hypotheses of
\ref{lem:K-injectives-b}, the functor $\RR R$ is a right adjoint of
$L\colon \D(\cat A)\to \D(\cat B)$.
\end{enumerate}
\end{lem} 
\begin{proof} 
Note that $L\dashv R$ extends to an adjunction $\K(L)\dashv \K(R)$.
Let $I\in \K(\cat B)$ be K-injective and $Y\in \K(\cat A)$ an acyclic complex.
Then $\K(L)(Y)$ is acyclic and hence $\Hom_{\K(\cat A)}(Y,\K(R)(I)) \cong
\Hom_{\K(\cat B)}(\K(L)(Y),I) = 0$. This proves \ref{lem:K-injectives-a}. 

We now prove \ref{lem:K-injectives-b}. Note that the unit $\id_{\K(\cat A)}
\To \ii_{\cat A}\q_{\cat A}$ is an isomorphism on the K-injective
complexes. Since $\K(R)$ preserves K-injective complexes, it follows that the
map $\K(R)\ii_{\cat B}\To \ii_{\cat A}\q_{\cat
A}\K(R)\ii_{\cat B}= \ii_{\cat A}\RR R$ is an isomorphism.

For \ref{lem:K-injectives-c}, note that $\K(R)$ is right adjoint to $\K(L)\colon
\K(\cat A)\to \K(\cat B)$. There are natural bijections
\begin{align*}
\Hom_{\D(\cat A)}\bigl(\q_{\cat A}X, \RR RY\bigr) &\cong \Hom_{\K(\cat A)}
\bigl(X,\ii_{\cat A}\RR RY\bigr)\\
&\cong \Hom_{\K(\cat A)}\bigl(X, \K(R) \ii_{\cat B}Y\bigr) & & \text{(by
\ref{lem:K-injectives-b})}\\
&\cong \Hom_{\K(\cat B)}\bigl(\K(L)X, \ii_{\cat B}Y\bigr)\\
&\cong \Hom_{\D(\cat B)}\bigl(\q_{\cat B}\K(L)X, Y\bigr)\\
&\cong \Hom_{\D(\cat B)}\bigl(L\q_{\cat A}X, Y\bigr)
\end{align*}
for all $X\in \K(\cat A)$ and $Y\in \D(\cat B)$. This implies
\ref{lem:K-injectives-c}.
\end{proof} 

\begin{ex} 
Let $K\subseteq S$ be an open subgroup. 
\begin{enumerate}[label=(\alph*)]
\item The functor $\Hom_{\Rep_k(K)}(\one, \pholder)\colon \Rep_k(K)\to \Vect_k$ is
canonically isomorphic to $\H^0(K,\pholder)$. This extends to a natural isomorphism
$\Hom_{\Rep_k(K)}^\bullet(\one,\pholder) \xRightarrow{\cong} \K(\H^0(K,\pholder))$ of functors
$\K(K)\to \K(\Vect_k)$, which induces a natural isomorphism
\[
\RHom_{\Rep_k(K)}(\one,\pholder)\xRightarrow{\cong} \RR\H^0(K,\pholder)
\]
of functors $\D(K)\to \D(k)$. 

\item The functor $\H^0(K,\pholder)\Res^S_K\colon \Rep_k(S)\to \Vect_k$ of
$K$-invariants is left exact and hence admits a right derived functor
$\RR\bigl(\H^0(K,\pholder)\Res^S_K\bigr)\colon \D(S)\to \D(k)$. We observe that the diagram
\[
\begin{tikzcd}
\D(S) \ar[r,"\Res^S_K"] \ar[dr,"{\RR(\H^0(K,\pholder)\Res^S_K)}"']& \D(K)
\ar[d,"{\RR\H^0(K,\pholder)}"]\\
& \D(k)
\end{tikzcd}
\]
commutes. Indeed, since $K\subseteq S$ is a subgroup, $\Res^S_K$ admits an
\emph{exact} left adjoint $\ind_K^S$ (Lemma~\ref{lem:FrobeniusII}). By
Lemma~\ref{lem:K-injectives} we have
\begin{align*}
\RR\H^0(K,\pholder)\Res^S_K &= \q\H^0(K,\pholder)\ii \Res^S_K \cong \q \H^0(K,\pholder)\Res^S_K
\ii \\
&= \RR\bigl(\H^0(K,\pholder)\Res^S_K\bigr).
\end{align*}
For this reason, we often allow ourselves to omit $\Res^S_K$ when it does not
lead to confusion.
\end{enumerate}
\end{ex} 

Recall that, by definition, every smooth $S$-representation $V$ satisfies $V =
\varinjlim_{K} \Hom_{\Rep_k(K)}(\one,V)$, where $K$ runs through the compact open
subgroups of $S$. This property continues to hold on the level of derived
categories as well as the following proposition makes precise.

\begin{prop}\label{prop:smooth-derived} 
Let $X\in \D(S)$. There exists a natural isomorphism
\begin{equation}\label{eq:smooth-derived}
\varinjlim_{\substack{K\le S\\\text{compact open}} }
\RHom_{\Rep_k(K)}\bigl(\one, \Res^S_KX\bigr) \xrightarrow{\cong} \Res^S_1 X
\qquad \text{in $\D(k)$.}
\end{equation}
\end{prop} 
\begin{proof} 
Since taking filtered colimits is exact in $\Vect_k$, the maps
\begin{align*}
\varinjlim_{\substack{K\le S\\\text{compact open}} }
\RHom_{\Rep_k(K)}\bigl(\one,\Res^S_KX\bigr) &\xrightarrow{\cong}
\varinjlim_{\substack{K\le S\\\text{compact open}} }
\RR\H^0(K, X) \\
&= \varinjlim_{\substack{K\le S\\\text{compact open}} } 
\q \H^0(K, \ii X)\\
&\xrightarrow{\cong} \q
\varinjlim_{\substack{K\le S\\\text{compact open}} } \H^0(K,\ii X)\\
&\xrightarrow{\cong} \q \Res^S_1 \ii X\\ 
&\cong \Res^S_1 X
\end{align*}
are natural isomorphisms.
\end{proof} 

\begin{rmk} 
If $S$ is a locally profinite group, one can make sense of the left hand side
of \eqref{eq:smooth-derived} even as an object of $\D(S)$ and show that there is
a natural isomorphism 
\[
\varinjlim_{\substack{K\le S\\\text{compact open}} }
\RHom_{\Rep_k(K)}(\one, \Res^S_KX) \xrightarrow{\cong} X \qquad
\text{in $\D(S)$.}
\]
% (??prove this??) 
\end{rmk} 

\subsubsection{Brown representability} 
We will now formulate the Brown representability theorems. We make the
assumption that $S$ is a $p$-adic monoid and $k$
has characteristic $p$. Then $S$ admits a compact open subgroup $K$
which carries the structure of a $p$-adic Lie group of dimension, say, $d$.
Replacing $K$ by an open subgroup, we may even assume that $K$ is torsion-free.
Then $K$ is a Poincar\'e group by \cite[Cor.~(1)]{Serre.1965} and
\cite[V.2.5.8]{Lazard}. In
particular, the cohomology groups $\H^i(K,k)$ are finite-dimensional, and
$\H^i(K,\pholder) = 0$ provided $i>d$.

\begin{prop}\label{prop:compactlygenerated} 
Let $S$ be a $p$-adic monoid and let $K\subseteq S$ be a torsion-free compact
open subgroup. Then $\D(S)$ is compactly generated by $\ind_K^S\one$.
\end{prop} 
\begin{proof} 
This is precisely \cite[Prop.~6]{Schneider.2015} in the case where $S$ is
a $p$-adic Lie group, and the proof is general. We only have to observe that
$\ind_K^S$ is an exact functor, since $K$ is a group, and hence the adjunction
$\ind_K^S \dashv \Res^S_K$ (Lemma~\ref{lem:FrobeniusII}) extends to the derived
categories.
\end{proof} 

\begin{thm}[Brown representability] 
Denote by $\Ab$ the category of abelian groups.
\label{thm:brown}
\begin{enumerate}[label=(\alph*)]
\item\label{thm:brown-a} Let $H\colon \D(S)^\op\to \Ab$ be a cohomological
functor commuting with small products.\footnote{This means that $H$ sends small direct sums in $\D(S)$ to products in $\Ab$.} Then $H$ is representable, that is, there
exists an object $X\in \D(S)$ and a natural isomorphism $\Hom_{\D(S)}(\pholder,X)\xRightarrow{\cong} H$. In particular, $\D(S)$ admits small
products.
\item\label{thm:brown-b} Let $H\colon \D(S)\to \Ab$ be a cohomological functor
commuting with small direct sums. Then $H$ is corepresentable, that is, there
exists an object $X\in \D(S)$ and a natural isomorphism $\Hom_{\D(S)}(X,\pholder)\xRightarrow{\cong} H$.
\end{enumerate}
\end{thm} 
\begin{proof} 
See \cite[Thm.~8.3.3]{Neeman.2001} or \cite[Thm.~A]{Krause} for
\ref{thm:brown-a} and \cite[Thm.~B]{Krause} for \ref{thm:brown-b}.
\end{proof} 

The following corollary is a standard application of Brown representability.

\begin{cor}\label{cor:brown} 
Let $\cat C$ be a triangulated category and $F\colon \D(S)\to \cat C$ a
triangulated functor. 
\begin{enumerate}[label=(\alph*)]
\item\label{cor:brown-a} $F$ admits a triangulated right adjoint if and only if
$F$ commutes with small direct sums.
\item\label{cor:brown-b} $F$ admits a triangulated left adjoint if and only if
$F$ commutes with small products.
\end{enumerate}
\end{cor} 
\begin{proof} 
We only prove \ref{cor:brown-b} and leave \ref{cor:brown-a} as an exercise. It
is clear that $F$ commutes with small products if $F$ admits a left adjoint
(see Proposition~\ref{prop:adjoint-limit}). 

Conversely, if $F$ commutes with
small products, the functor $\Hom_{\cat C}\bigl(X,F(\pholder)\bigr) \colon \D(S)\to
\Ab$ is cohomological and commutes with
small products. By Theorem~\ref{thm:brown} this functor is representable. From
Theorem~\ref{thm:adjoint-unique} it follows that $F$ admits a left adjoint,
which, by \cite[Lem.~5.3.6]{Neeman.2001}, is even triangulated.
\end{proof} 

\subsubsection{Tensor triangulated structure} 
The category $\D(S)$ carries the structure of a closed symmetric monoidal
category, which was introduced and investigated in \cite{Schneider-Sorensen.2023a} when $S$ is a group. 
The tensor product of cochain complexes endows the triangulated category $\K(S)$
with a symmetric monoidal structure. The $\otimes$-unit is given by $\one
\coloneqq k[0]$, that is, $k$ viewed as a complex concentrated in degree $0$.
Since, for any $X,Y\in \K(S)$, the complex
$X\otimes_kY$ is acyclic
whenever either $X$ or $Y$ is acyclic, the induced symmetric monoidal structure on
$\D(S)$ is again given by $\pholder\otimes_k\pholder$.

\begin{prop} 
The symmetric monoidal category $\D(S)$ is closed.
\end{prop} 
\begin{proof} 
Let $X\in \D(S)$. Then $\pholder\otimes_kX$ is a triangulated functor commuting with
small direct sums. By Corollary~\ref{cor:brown}.\ref{cor:brown-a}, it admits a
right adjoint.
\end{proof} 

\begin{rmk} 
The internal Hom objects in $\D(S)$ are explicitly given as follows: Let $X, Y \in
\Chain(S)$ be complexes and define a new complex $\iHom(X,Y)$ in $\Chain(S)$
by
\[
\iHom^n(X,Y) \coloneqq \Bigl(\prod_{i\in \Z} \Hom_{k[S]}\bigl(k[S]\otimes_k
X^i, Y^{i+n}\bigr)\Bigr)^{\sm},
\]
with differential $(df)^i = d_Y^{n+i}\circ f^i - (-1)^n f^{i+1}\circ
(1\otimes d_X^i)$, for $f\in \iHom^n(X,Y)$; here, $S$ acts diagonally on $k[S]\otimes_k X^i$. This yields a $k$-linear bifunctor
\begin{equation}\label{eq:Chain-iHom}
\iHom\colon \Chain(S)^\op\times \Chain(S) \to \Chain(S).
\end{equation}
A direct computation shows
\begin{equation}\label{eq:H(Hom)}
\H^0\bigl(\iHom(X,Y)\bigr) \cong \varinjlim_{\substack{S'\le S\\
\text{open} }} \Hom_{\K(S)}\bigl(k[S/S'] \otimes_kX,Y\bigr).
\end{equation}
As in the proof of Proposition~\ref{prop:innerHom} one constructs a natural
isomorphism
\[
\iHom\bigl(X\otimes_kY, Z\bigr) \cong \iHom\bigl(X, \iHom(Y,Z)\bigr).
\]
Taking cohomology and then $S$-invariants yields a natural isomorphism
\begin{equation}\label{eq:iHom-adjoint}
\Hom_{\K(S)}\bigl(X\otimes_kY,Z\bigr) \cong \Hom_{\K(S)}\bigl(X,
\iHom(Y,Z)\bigr).
\end{equation}
This has the following immediate consequences:
\begin{enumerate}[label=(\roman*)]
\item\label{iHom-i} The bifunctor \eqref{eq:Chain-iHom} descends to a bifunctor
\[
\iHom\colon \K(S)^\op \times \K(S) \to \K(S)
\] 
making $\K(S)$ a closed symmetric monoidal category.
\item\label{iHom-ii} As a right adjoint of $\pholder\otimes_kY$, the functor
$\iHom(Y,\pholder)$ is triangulated, by \cite[Lem.~5.3.6]{Neeman.2001}. It is also
true that $\iHom(\pholder,Y)$ preserves distinguished triangles, cf.~\cite[Prop.~11.6.4.(ii)]{KS}.
\item\label{iHom-iii} If $Z$ is K-injective, then the complex $\iHom(Y,Z)$
again is K-injective.
\item\label{iHom-iv} If $Z$ is K-injective and $Y$ is acyclic, then $\iHom(Y,Z)$
is contractible and, in particular, acyclic.
Consequently, $\iHom(\pholder,Z)$ preserves quasi-isomorphisms.
\end{enumerate}
The last property shows that the right derived functor
\[
\RR\iHom\colon \D(S)^\op \times \D(S)\to \D(S)
\]
exists and is computed by $\RR\iHom(X,Y) \cong \iHom(X, \ii Y)$. Using
\ref{iHom-iii} above, one shows that there is a natural isomorphism
\[
\Hom_{\D(S)}\bigl(X\otimes_kY, Z\bigr) \cong \Hom_{\D(S)}\bigl(X,
\RR\iHom(Y,Z)\bigr).
\]
Therefore, $\RR\iHom(Y,\pholder)$ is the right adjoint of $\pholder\otimes_kY$. Moreover,
\eqref{eq:H(Hom)}, applied to $\ii Y$, shows that there is a natural
isomorphism
\begin{equation}\label{eq:H(RHom)}
\H^0\bigl(\RR \iHom(X,Y)\bigr) \cong \varinjlim_{\substack{S'\le
S\\\text{open}} } \Hom_{\D(S)}\bigl(k[S/S']\otimes_kX,Y\bigr).
\end{equation}
If, in addition, $S$ is a group, then Remark~\ref{rmk:internal-hom-group}
implies that there is a natural isomorphism
\begin{equation}\label{eq:H(RHom)-group}
\H^0\bigl(\RR \iHom(X,Y)\bigr) \cong \varinjlim_{\substack{S'\le
S\\\text{open}} } \Hom_{\D(S')}(X,Y).
\end{equation}
\end{rmk} 

\begin{notation}\label{nota:dual} 
Given $X\in \D(S)$, we denote the right adjoint of $\pholder\otimes_kX$ by
$\hom_{\D(S)}(X,\pholder)$ or just $\hom(X,\pholder)$ if the category is clear from the
context.

We call $X^\vee \coloneqq \hom(X,\one)$ the \emph{dual} of $X$ in $\D(S)$.
\end{notation} 

\begin{warning} 
The internal Hom complex $\hom_{\D(S)}(X,Y) \in \D(S)$ should not be confused
with the $\RHom$ complex $\RHom_{\Rep_k(S)}(X,Y)$, which is only an object
of $\D(k)$. In general, there is the relation
\begin{equation}\label{eq:RHom-iHom}
\RR\H^0(S,\pholder)\circ \hom_{\D(S)}(X,Y) \cong \RHom_{\Rep_k(S)}(X,Y),
\end{equation}
which follows from \ref{iHom-iii} above, \ie, the fact that $\iHom(X,\pholder)$
preserves K-injective complexes.
\end{warning} 

\subsubsection{Rigidly-compact generation} 
Let $K$ be a compact torsion-free $p$-adic Lie group and $k$ a field of
characteristic $p$. The next crucial observation allows us to appeal to the
results of \cite{Balmer.2016}.

\begin{prop}\label{prop:rigid-compact} 
The category $\D(K)$ is rigidly-compactly generated.
\end{prop} 
\begin{proof} 
We know from Proposition~\ref{prop:compactlygenerated} that $\D(K)$ is compactly
generated by $\one$. It remains to show that the set of compact objects coincides
with the set of rigid objects. Recall that $X\in \D(K)$ is called \emph{compact}
if for every family $\{Y_i\}_{i\in I}$ in $\D(K)$ the natural map
\[
\bigoplus_{i\in I}\Hom_{\D(K)}(X,Y_i) \to \Hom_{\D(K)}\Bigl(X,
\bigoplus_{i\in I}Y_i\Bigr)
\]
is bijective. Recall also that $X$ is called \emph{rigid} if for all $Y\in
\D(K)$ the natural map
\[
X^\vee \otimes_k Y \to \hom(X,Y)
\]
is an isomorphism in $\D(K)$. Since $\one$ is compact and
$\Hom_{\D(K)}\bigl(\one, \hom(X,Y)\bigr)= \Hom_{\D(K)}(X,Y)$, it is immediate
that rigid objects are compact. For the converse, it is easy to see that $\one$ is
rigid and that the strictly full subcategory of rigid objects is triangulated and
closed under direct summands (closure under direct summands can either be checked directly or by observing that split idempotents are absolute equalizers \cite[Prop.~6.5.4]{Borceux.I}). Therefore, the strictly full, thick, triangulated
subcategory $\langle\one\rangle$ generated by $\one$ consists
of rigid objects. Now, \cite[Lem.~2.2]{Neeman.1992b} shows that $\langle
\one\rangle$ consists of all the compact objects in $\D(K)$.
\end{proof} 

\subsubsection{Spectral sequences} 
Lacking a good reference, we explain how to construct spectral sequences using
the canonical $t$-structure on $\D(\cat A)$, for $\cat A$ an abelian category.
Everything in this paragraph is well-known and holds in much greater generality,
cf.~\cite[1.2.2]{Lurie-HA}. 

Our main reference is \cite[3.2]{Benson.1991b}, which constructs spectral
sequences from filtered complexes. Our treatment is very similar, yet not the
same, and so we refer to \loccit\ for the proofs, which apply verbatim in our
case if the notation is reinterpreted accordingly.

First, let us recall the canonical $t$-structure on $\D(\cat A)$, for an
abelian category $\cat A$. Given a cochain complex $(C, \{d^i\}_{i}) \in
\Chain(\cat A)$, the truncated complexes $\tau^{\le n}C$ and $\tau^{\ge n}C$ in
$\Chain(\cat A)$ are defined by
\begin{align*}
\tau^{\le n}C\colon& \qquad \dotsb \xrightarrow{d^{n-3}} C^{n-2}
\xrightarrow{d^{n-2}} C^{n-1} \xrightarrow{d^{n-1}} \Ker(d^n) \to 0
\to \dotsb\\
\intertext{and}
\tau^{\ge n}C\colon& \qquad \dotsb \to 0\to
\Coker(d^{n-1}) \xrightarrow{d^{n}} C^{n+1} \xrightarrow{d^{n+1}} C^{n+2}
\xrightarrow{d^{n+2}} \dotsb
\end{align*}
The truncation functors $\tau^{\le n}$ and $\tau^{\ge n}$ preserve
quasi-isomorphisms and hence give rise to truncation functors on $\D(\cat A)$.
For each complex $C$ there is a distinguished triangle
\[
\tau^{\le n}C \to C \to \tau^{\ge n+1}C \xrightarrow{+}
\]
in $\D(\cat A)$, and we have $\Hom_{\D(\cat A)}(\tau^{\le n}C, \tau^{\ge n+1}D)
= 0$, for all $C,D\in \D(\cat A)$, $n\in\Z$. This defines the canonical
$t$-structure on $\D(\cat A)$. Note also that 
\[
\H^n(C)[-n]\cong \tau^{\le n}\tau^{\ge n}C \cong \tau^{\ge n}\tau^{\le
n}C,\qquad\text{for each $C\in \D(\cat A)$.}
\]
The following version of the hypercohomology spectral sequence will be useful
later. For the notion of way-out functors, we refer to
\cite[\S I.7]{Hartshorne.1966}

\begin{lem}\label{lem:spectral} 
Let $\cat A, \cat B$ be abelian categories and let $X\in \D(\cat A)$.
\begin{enumerate}[label=(\alph*)]
\item\label{lem:spectral-a} Let $F\colon \D(\cat A)\to \D(\cat B)$ be a
triangulated functor. Assume that one of the following properties is
satisfied:
\begin{enumerate}[label=$\bullet$]
\item $F$ is way-out right (\eg, left $t$-exact) and $X$ is left bounded;
\item $F$ is way-out left (\eg, right $t$-exact) and $X$ is right bounded;
\item $F$ is way-out in both directions (\eg, $t$-exact);
\item $X$ is bounded.
\end{enumerate}
Then there exists a convergent spectral sequence
\[
E_2^{i,j} = (\H^iF)\bigl(\H^j(X)\bigr) \To \H^{i+j}(FX).
\]

\item\label{lem:spectral-b} Let $F\colon \D(\cat A)^\op\to \D(\cat B)$ be a
triangulated functor. Assume that one of the following properties is satisfied:
\begin{enumerate}[label=$\bullet$]
\item $F$ is way-out right and $X$ is right bounded;
\item $F$ is way-out left and $X$ is left bounded;
\item $F$ is way-out in both directions;
\item $X$ is bounded.
\end{enumerate}
Then there exists a convergent spectral sequence
\[
E_2^{i,j} = (\H^iF)\bigl(\H^{-j}(X)\bigr) \To \H^{i+j}(FX).
\]
\end{enumerate}
\end{lem} 
\begin{proof} 
The following construction is adapted from \cite[\S3.2]{Benson.1991b} although we use slightly different notation. We refer
to \loccit\ for details about the proofs.
\begin{itemize}
\item[\ref{lem:spectral-a}] Put $X(r)\coloneqq \tau^{\le -r}X$ for each
$r\in\Z$. Given $s\ge r$, we define $X(r,s)$ by the distinguished triangle $X(s)\to
X(r)\to X(r,s)\xrightarrow{+}$ in $\D(\cat A)$; note that $X(r,r+1) \cong
\H^{-r}(X)[r]$. Applying $F$, we obtain a distinguished triangle
\[
FX(s) \to FX(r) \to FX(r,s)\xrightarrow{+} \quad \text{in $\D(\cat B)$.}
\]
For $s=r+1$ we contemplate the induced long exact sequence in cohomology
\[
\dotsb \to \H^n\bigl(FX(r+1)\bigr) \xrightarrow{i_1} \H^n\bigl(FX(r)\bigr)
\xrightarrow{j_1} \H^n\bigl(FX(r,r+1)\bigr) \xrightarrow{k_1}
\H^{n+1}\bigl(FX(r+1)\bigr) \to \dotsb.
\]
Define $D_1^{r,s} \coloneqq \H^{r+s}\bigl(FX(r)\bigr)$ and $E_1^{r,s} \coloneqq
\H^{r+s}\bigl(FX(r,r+1)\bigr)$, for all $r,s\in\Z$. We obtain an exact couple
\[
\begin{tikzcd}[column sep=0em]
D_1^{*,*} \ar[rr,"i_1"] & & D_1^{*,*} \ar[dl,"j_1"]\\
& E_1^{*,*} \ar[ul,"k_1"]
\end{tikzcd}
\]
that is, we have $\Image(i_1) = \Ker(j_1)$, $\Image(j_1) = \Ker(k_1)$, and
$\Image(k_1) = \Ker(i_1)$. We denote by $d_1 = j_1\circ k_1\colon E_1^{*,*}\to
E_1^{*+1,*}$ the associated differential. 
Taking derived couples yields a spectral sequence $\{(E_n,d_n)\}_{n\ge1}$.

For each $n\ge1$ we put
\begin{align*}
Z_n^{r,s} &\coloneqq \Image\bigl(\H^{r+s}\bigl(FX(r,r+n)\bigr) \to
\H^{r+s}\bigl(FX(r,r+1)\bigr)\bigr),\\
\intertext{and}
B_n^{r,s} &\coloneqq \Ker\bigl(\H^{r+s}\bigl(FX(r,r+1)\bigr) \to
\H^{r+s}\bigl(FX(r-n+1,r+1)\bigr)\bigr)
\end{align*}
and then $Z_{\infty}^{r,s} \coloneqq \bigcap_n Z_n^{r,s}$, $B_\infty^{r,s}
\coloneqq \bigcup_n B_n^{r,s}$. Now, \cite[Prop.~3.2.4]{Benson.1991b} shows
$E_n^{r,s} = Z_n^{r,s}/B_n^{r,s}$. Put $E_\infty^{r,s} \coloneqq Z_\infty^{r,s}
/B_\infty^{r,s}$. 
We define a descending filtration on $\H^{r+s}(FX)$ by
\[
\Fil^r \H^{r+s}(FX) \coloneqq \Image\bigl(\H^{r+s}(FX(r)) \to
\H^{r+s}(FX)\bigr).
\]
We now verify that $\{(E_n^{*,*},d_n)\}_n$ is \emph{biregular}, that is:
\begin{enumerate}[label=(\roman*)]
\item For all $r,s\in\Z$ there exists $n\gg0$ such that $Z_\infty^{r,s} =
Z_n^{r,s}$ and $B_\infty^{r,s} = B_n^{r,s}$.
\item For all $m\in\Z$, we have $\Fil^r\H^{m}(FX) = 0$, for $r\gg0$, and
$\Fil^r\H^{m}(FX) = \H^{m}(FX)$, for $r\ll0$.
\end{enumerate}
It then follows that the spectral sequence converges to $\H^*(FX)$.

\begin{enumerate}[label=$\bullet$, itemsep=1ex] 
\item If $X$ is left bounded, then we have $Z_n^{r,s} = Z_\infty^{r,s}$, for
$n\gg0$, and $\Fil^r\H^m(FX) = 0$, for fixed $m\in\Z$ and all $r\gg0$. Indeed,
there exists $n\gg0$ such that $X(r+n) = 0$. From
the distinguished triangle $X(r+n)\to X(r)\to X(r,r+n)\xrightarrow{+}$ we deduce
$X(r) \cong X(r,r+n)$. This implies $\H^{r+s}(FX(r,r+n)) \cong
\H^{r+s}(FX(r))$ and hence $Z_n^{r,s} = Z_\infty^{r,s}$, for $n\gg0$. Moreover,
for all $r\gg0$, we have $X(r) = 0$ and hence $\H^m(FX(r)) = 0$, that is,
$\Fil^r\H^{m}(FX(r)) = 0$, for $r\gg0$.

\item If $X$ is right bounded, then we have $B_n^{r,s} = B_\infty^{r,s}$, for
$n\gg0$, and $\Fil^r\H^{m}(FX) = \H^{m}(FX)$, for fixed $m\in\Z$ and all
$r\ll0$. Indeed, there exists $n\gg0$ such that $X(r-n+1) \cong
X(r-n'+1)$, for all $n'\ge n$. This implies $\H^{r+s}(FX(r-n+1,r+1)) \cong
\H^{r+s}(FX(r-n'+1,r+1))$ and hence $B_n^{r,s} = B_\infty^{r,s}$. Moreover, for
all $r\gg0$ we have $X(r) = X$ and hence $\Fil^r\H^{m}(FX) = \H^{m}(FX)$.

\item If $F$ is way-out right, then we have $B_n^{r,s} = B_\infty^{r,s}$, for
$n\gg r+s$, and $\Fil^r\H^{m}(FX) = \H^{m}(FX)$, for all $r\ll -m$. Indeed,
we have $\H^{r+s}(FX(r-n'+1,r-n+1)) = 0$ whenever $n'\ge n\gg r+s$. Hence, from
the distinguished triangle
\[
FX(r-n'+1,r+1)\to FX(r-n+1,r+1)\to FX(r-n+1,r-n'+1) \xrightarrow{+},
\]
we deduce an isomorphism
\[
\H^{r+s}\bigl(FX(r-n+1,r+1)\bigr) \cong \H^{r+s}\bigl(FX(r-n'+1,r+1)\bigr),
\]
which shows $B_n^{r,s} = B_\infty^{r,s}$. Similarly, we have $\H^{m}(FX(r)) \cong
\H^{m}(FX)$, whenever $r\ll -m$, and hence $\Fil^r\H^{m}(FX) = \H^{m}(FX)$.

\item If $F$ is way-out left, then we have $Z_n^{r,s} = Z_\infty^{r,s}$, for
all $n\gg -2r-s$, and $\Fil^r\H^{m}(FX) = 0$, for $r\gg-m$.

Indeed, we have $\H^{r+s}(FX(r+n,r+n')) = 0$ whenever $n'\ge n\gg -2r-s$.
From the distinguished triangle
\[
FX(r+n,r+n') \to FX(r,r+n') \to FX(r,r+n) \xrightarrow{+}
\]
we deduce an isomorphism $\H^{r+s}(FX(r,r+n')) \cong \H^{r+s}(FX(r,r+n))$, which
shows $Z_n^{r,s} = Z_\infty^{r,s}$. Moreover, for $r\gg -m$ we have $\H^{m}(FX(r))
= 0$, so that $\Fil^r\H^{m}(FX) = 0$.
\end{enumerate} 

By \cite[Thm.~3.2.9]{Benson.1991b} we have 
\[
\Fil^r\H^{r+s}(FX)/\Fil^{r+1}\H^{r+s}(FX) \cong E_\infty^{r,s}.
\]
Therefore, the spectral sequence indeed computes $\H^*(FX)$. Note that
\[
E_1^{r,s} = \H^{r+s}\bigl(FX(r,r+1)\bigr) = (\H^{r+s}F)\bigl(\H^{-r}(X)[r]\bigr)
= (\H^{2r+s}F)\bigl(\H^{-r}(X)\bigr).
\]
Hence, reindexing $(i,j)\coloneqq (2r+s,-r)$ yields the desired spectral
sequence $E_2^{i,j} = E_1^{r,s} = (\H^iF)\bigl(\H^j(X)\bigr) \To \H^{i+j}(FX)$.

\item[\ref{lem:spectral-b}] Of course, this follows by duality from
\ref{lem:spectral-a}, but let us nevertheless give the construction in this
case. Put $X(r) \coloneqq \tau^{\ge r}X$ for each
$r\in\Z$. Given $r\le s$, we have a distinguished triangle $X(r,s) \to X(r) \to
X(s) \xrightarrow{+}$ in $\D(\cat A)$; note that $X(r,r+1) \cong \H^{r}(X)[-r]$.
Applying $F$, we obtain a distinguished triangle
\[
FX(s) \to FX(r) \to FX(r,s) \xrightarrow{+} \quad \text{in $\D(\cat B)$.}
\]
Put $D_1^{r,s} \coloneqq \H^{r+s}(FX(r))$ and $E_1^{r,s}\coloneqq
\H^{r+s}(FX(r,r+1))$.
As in \ref{lem:spectral-a} this gives rise to a spectral sequence
$\{(E_n,d_n)\}_n$ which converges to $\H^*(FX)$. We compute
\[
E_1^{r,s} = \H^{r+s}\bigl(FX(r,r+1)\bigr) = (\H^{r+s}F)\bigl(\H^{r}(X)[-r]\bigr)
= (\H^{2r+s}F)\bigl(\H^r(X)\bigr).
\]
Again, reindexing $(i,j)\coloneqq (2r+s, -r)$ yields the desired spectral
sequence $E_2^{i,j} = E_1^{r,s} = (\H^iF)\bigl(\H^{-j}(X)\bigr) \To
\H^{i+j}(FX)$.
\end{itemize}
\end{proof}

\section{The left adjoint of derived inflation}\label{sec:derivedinf} 
From now on, $k$ denotes a field of characteristic $p>0$.
\subsection{The compact case}\label{subsec:compact} 
Let $f\colon K_P\longtwoheadrightarrow K_M$ be a surjective homomorphism of
torsion-free, compact $p$-adic Lie groups with kernel $K_U$. Note that $f$ is
automatically open as a surjection between profinite groups. Being
torsion-free $p$-adic Lie groups, these groups are pro-$p$, and
\cite[Cor.~(1)]{Serre.1965}
combined with \cite[V.2.5.8]{Lazard} shows that the groups $K_P, K_M$, and
$K_U$ are Poincar\'e, cf.~\cite[Ch.~I, \S4.5]{Serre.2013}. Recall that a pro-$p$
group $K$ is called \emph{Poincar\'e} of dimension $n$ if $\H^i(K,\F_p)$ is
finite for all $i$, $\dim_{\F_p}\H^n(K,\F_p) = 1$, and the cup-product
$\H^i(K,\F_p)\times \H^{n-i}(K,\F_p)\to \H^n(K,\F_p)$ is a non-degenerate
bilinear form.

\begin{notation} 
If $K'_P \subseteq K_P$ is an open subgroup, we write $K'_M \coloneqq f(K'_P)$
and $K'_U\coloneqq K_U\cap K'_P$.
\end{notation} 

By Proposition~\ref{prop:rigid-compact} the categories $\D(K_P)$ and $\D(K_M)$
are rigidly-compactly generated. Note that inflation along $f$
yields an exact functor $\Inf^{K_M}_{K_P}\colon \Rep_k(K_M)\to
\Rep_k(K_P)$ that is strictly monoidal, which means
$\Inf^{K_M}_{K_P}(\one) = \one$ and
$\Inf^{K_M}_{K_P}(V\otimes_kW) \cong \Inf^{K_M}_{K_P}(V)\otimes_k
\Inf^{K_M}_{K_P}(W)$, for all $V,W\in \Rep_k(K_M)$. The right adjoint is given
by the functor of invariants $\H^0(K_U,\pholder)$.

Passing to the derived categories, we obtain a strictly monoidal functor
\[
\Inf^{K_M}_{K_P}\colon \D(K_M)\to \D(K_P)
\]
with right adjoint $\RR\H^0(K_U,\pholder)$. For later use we record the following
result:

\begin{lem}[Projection formula]\label{lem:RH(U,-)-projection-formula} 
The natural map
\[
X\otimes_k \RR\H^0(K_U,Y) \xrightarrow{\cong} \RR\H^0\bigl(K_U,
\Inf^{K_M}_{K_P}(X)\otimes_k Y\bigr)
\]
is an isomorphism, for all $X\in \D(K_M)$ and $Y\in \D(K_P)$. In particular,
$X\otimes_k \RR\H^0(K_U,\one) \xrightarrow{\cong} \RR\H^0\bigl(K_U,
\Inf^{K_M}_{K_P}X\bigr)$, for all $X\in \D(K_M)$.
\end{lem} 
\begin{proof} 
This follows from \cite[2.15.~Prop.]{Balmer.2016} together with
Proposition~\ref{prop:rigid-compact}.
\end{proof} 

Our goal in this section is to prove the existence of the left adjoint of
$\Inf^{K_M}_{K_P}$ and to determine it explicitly. To start, we observe
that $\RR\H^0(K_U,\pholder)$ admits a right adjoint.

\begin{lem}\label{lem:F^M_P} 
The functor $\RR\H^0(K_U,\pholder)\colon \D(K_P)\to \D(K_M)$ admits a right adjoint,
denoted $F^{K_M}_{K_P}$. More precisely, there is a natural isomorphism
\begin{equation}\label{eq:F^M_P-adjunction}
\RHom_{\Rep_k(K_P)}\bigl(X, F^{K_M}_{K_P}(Y)\bigr) \cong
\RHom_{\Rep_k(K_M)}\bigl(\RR\H^0(K_U,X), Y\bigr) 
\end{equation}
in $\D(k)$, for all $X\in \D(K_P)$ and $Y\in \D(K_M)$.

Further, for each open subgroup $K'_P\subseteq K_P$, the following diagram of
functors commutes:
\[
\begin{tikzcd}[column sep=3em, row sep=3em]
\D(K_M) \ar[d,"\Res^{K_M}_{K'_M}"'] \ar[r,"F^{K_M}_{K_P}"] & \D(K_P)
\ar[d,"\Res^{K_P}_{K'_P}"]\\
\D(K'_M) \ar[r,"F^{K'_M}_{K'_P}"'] & \D(K'_P).
\end{tikzcd}
\]
\end{lem} 
\begin{proof} 
The functor $\Inf^{K_M}_{K_M}$ satisfies \cite[Hyp.~1.2]{Balmer.2016} and hence
\cite[Cor.~2.14]{Balmer.2016} shows that the right adjoint $F^{K_M}_{K_P}$ of
$\RR\H^0(K_U,\pholder)$ exists. 
Then \cite[(2.18)]{Balmer.2016} shows that there is a natural isomorphism
\[
\RR\H^0\bigl(K_U, \hom_{\D(K_P)}(X,
F^{K_M}_{K_P}(Y))\bigr) \xrightarrow{\cong} \hom_{\D(K_M)}\bigl(\RR\H^0(K_U,X),
Y\bigr)
\]
in $\D(K_M)$. Applying $\RR\H^0(K_M,\pholder)$, and using $\RR\H^0(K_M,\pholder)\circ
\RR\H^0(K_U,\pholder) \cong \RR\H^0(K_P,\pholder)$ and \eqref{eq:RHom-iHom}, yields the
isomorphism \eqref{eq:F^M_P-adjunction}.

Let now $K'_P\subseteq K_P$ be an open subgroup. It is obvious that there is a
natural isomorphism $\Inf^{K'_M}_{K'_P}\Res^{K_M}_{K'_M} \xRightarrow{\cong}
\Res^{K_P}_{K'_P}\Inf^{K_M}_{K_P}$. Passing to the right adjoints, see
Example~\ref{ex:mates}, yields a natural isomorphism $\RR\H^0(K_U,\pholder)
\Ind_{K'_P}^{K_P} \xRightarrow{\cong} \Ind_{K'_M}^{K_M}\RR\H^0(K'_U,\pholder)$. Since
$K'_P\subseteq K_P$ is open, the functor $\Res^{K_P}_{K'_P}$ is also right
adjoint to $\Ind_{K'_P}^{K_P}$. Similarly, $\Res^{K_M}_{K'_M}$ is the right
adjoint of $\Ind_{K'_M}^{K_M}$. Hence, passing to the right adjoints again
yields a natural isomorphism $F^{K'_M}_{K'_P}\Res^{K_M}_{K'_M}
\xRightarrow{\cong} \Res^{K_P}_{K'_P} F^{K_M}_{K_P}$.
\end{proof} 

\begin{lem}\label{lem:F^M_P-transitive} 
Let $f'\colon K_M\longtwoheadrightarrow K_L$ be another surjective homomorphism of
torsion-free $p$-adic Lie groups. Then the following diagram of functors
commutes:
\[
\begin{tikzcd}
\D(K_L) \ar[r,"F^{K_L}_{K_M}"] \ar[dr,"F^{K_L}_{K_P}"'] & \D(K_M)
\ar[d,"F^{K_M}_{K_P}"]\\
& \D(K_P)
\end{tikzcd}
\]
\end{lem} 
\begin{proof} 
Clearly, we have a natural isomorphism $\Inf^{K_M}_{K_P}\Inf^{K_L}_{K_M}
\xRightarrow{\cong} \Inf^{K_L}_{K_P}$. Passing to the right adjoints,
cf.~Example~\ref{ex:mates}, yields an isomorphism $\RR\H^0(\Ker(f'\circ f),\pholder)
\xRightarrow{\cong} \RR\H^0(\Ker(f'),\pholder)\circ\RR\H^0(\Ker(f),\pholder)$. The statement
follows by passing to the right adjoints again.
\end{proof} 

\begin{prop}\label{prop:RH(U,-)-compact} 
The functor $\RR\H^0(K_U,\pholder)\colon \D(K_P)\to \D(K_M)$ preserves compact objects. 
\end{prop} 
\begin{proof} 
Since $K_U$ a Poincar\'e group, the object $\RR\H^0(K_U,\one) \in \D(K_M)$ is
represented by a bounded complex with finite-dimensional cohomologies. Since
$\one \in \D(K_M)$ is compact by Proposition~\ref{prop:compactlygenerated}, and
because the subcategory of compact objects is triangulated, it
suffices to show that $\RR\H^0(K_U,\one)$ is contained in the strictly full
triangulated subcategory $\langle \one\rangle$ of $\D(K_M)$ generated by $\one$.
This follows from a standard argument which we recall for the benefit of the
reader. Let $X\in \D(K_M)$ be bounded with finite-dimensional cohomology groups.
There exist integers $a\le b$ with $\tau^{\le i}X \cong 0$ for $i<a$ and
$\tau^{\ge i}X \cong 0$ for $i>b$. Then $\tau^{\ge b}X \cong \H^b(X)[-b]$ is
finite-dimensional and concentrated in degree $b$. We obtain a distinguished
triangle $\tau^{\le b-1}X \to X \to \H^b(X)[-b] \xrightarrow{+}$.
Hence, by induction on $b-a$ we are reduced to the case where $X$ is a
finite-dimensional smooth $K_M$-representation concentrated in one degree,
say $X\cong V[i]$. Then $\H^0(K_M,V)$ is non-zero (since $K_M$ is a pro-$p$
group) and contained in $\langle \one\rangle$, because $\langle \one\rangle$ is
closed under finite direct sums. By induction on $\dim_kV$ also the quotient
$V/\H^0(K_M,V)$ is contained in $\langle \one\rangle$. As $\langle \one\rangle$
is triangulated and $\H^0(K_M,V)\to V\to V/\H^0(K_M,V)\xrightarrow{+}$ is a
distinguished triangle in $\D(K_M)$, it follows that $V$, hence also $X$, is
contained in $\langle \one\rangle$.
\end{proof} 

We draw two important consequences of Proposition~\ref{prop:RH(U,-)-compact},
but first we make a definition.
\begin{defn} 
The object $\omega_{K_P}\coloneqq F^{K_M}_{K_P}(\one)$ in $\D(K_P)$ is called
the \emph{dualizing object}.
\end{defn} 

\begin{cor}[Grothendieck duality]\label{cor:Grduality} 
For all $X,Y\in \D(K_M)$ there is a natural isomorphism
\[
F^{K_M}_{K_P}(X)\otimes_k \Inf^{K_M}_{K_P}(Y) \xrightarrow{\cong}
F^{K_M}_{K_P}(X\otimes_k Y).
\]
In particular, we have a natural isomorphism $\omega_{K_P}\otimes_k
\Inf^{K_M}_{K_P}(Y) \xrightarrow{\cong}F^{K_M}_{K_P}(Y)$.
\end{cor} 
\begin{proof} 
This follows from Proposition~\ref{prop:RH(U,-)-compact} and
\cite[3.2.~Prop.]{Balmer.2016}.
\end{proof} 

\begin{cor}\label{cor:Inf-product} 
The functor $\Inf^{K_M}_{K_P}\colon \D(K_M)\to \D(K_P)$ preserves small
products. In particular, it admits a left adjoint $\LL_{K_U}\colon \D(K_P)\to
\D(K_M)$. Moreover, there is a natural isomorphism 
\begin{equation}\label{eq:left-adjoint-explicit}
\RR\H^0(K_U,\omega_{K_P} \otimes_k\pholder) \xRightarrow{\cong} \LL_{K_U}.
\end{equation}
\end{cor} 
\begin{proof} 
By Proposition~\ref{prop:RH(U,-)-compact} the adjunction $\Inf^{K_M}_{K_P}
\dashv \RR\H^0(K_U,\pholder)$ restricts to an adjunction on the full subcategories of
compact objects $\D(K_M)^{\cpt} \subseteq \D(K_M)$ and $\D(K_P)^{\cpt} \subseteq
\D(K_P)$. The argument in \cite[Lem.~3.1]{Balmer.2016} shows that the inflation
functor $\D(K_M)^{\cpt} \to \D(K_P)^{\cpt}$ admits a left adjoint,
and then \cite[Lem.~2.6(b)]{Balmer.2016} shows that $\Inf^{K_M}_{K_P}$ preserves
small products. By the
Corollary~\ref{cor:brown}.\ref{cor:brown-b} to Brown representability it follows
that the left adjoint $\LL_{K_U}$ exists. The isomorphism
\eqref{eq:left-adjoint-explicit} is the ur-Wirthm\"uller isomorphism
\cite[(3.10)]{Balmer.2016}.
\end{proof} 

To finish our discussion of $\LL_{K_U}$ we determine the dualizing complex
$\omega_{K_P}$ explicitly. 

\begin{prop}\label{prop:dualizing} 
One has $\omega_{K_P} \cong k[\dim K_U]$ in $\D(K_P)$. In particular,
$\omega_{K_P}$ is $\otimes$-invertible.
\end{prop} 
\begin{proof} 
We apply Lemma~\ref{lem:F^M_P-transitive} to the surjection $f'\colon K_M\to
\{1\}$, and together with Corollary~\ref{cor:Grduality} we obtain isomorphisms
\[
F^1_{K_P}(\one) \cong F^{K_M}_{K_P}\bigl(F^{1}_{K_M}(\one)\bigr) \cong
\omega_{K_P} \otimes_k \Inf^{K_M}_{K_P}\bigl(F^1_{K_M}(\one)\bigr).
\]
Hence, if we can show $F^1_{K_P}(\one) \cong k[\dim K_P]$ and $F^1_{K_M}(\one)
\cong k[\dim K_M]$, the assertion follows from $\dim K_P = \dim K_M + \dim K_U$,
cf.~\cite[4.8 Thm.]{DDMS99}.

It remains to prove that $F^1_K(\one) \cong k[\dim K]$ provided $K$ is a
torsion-free $p$-adic Lie group.

We have the following isomorphisms in $\D(k)$:
\begin{align*}
\Res^{K}_1F^1_{K}(\one) &\cong \varinjlim_{K'\le K}
\RHom_{\Rep_k(K')}\bigl(\one, \Res^{K}_{K'}F^1_{K}(\one)\bigr) & &
\text{(Prop.~\ref{prop:smooth-derived})}\\
&\cong \varinjlim_{K'\le K} \RHom_{\Rep_k(K')}\bigl(\one,
F^1_{K'}(\one)\bigr) & & \text{(Lem.~\ref{lem:F^M_P})}\\
&\cong \varinjlim_{K'\le K} \RHom_{\Vect_k}\bigl(
\RR\H^0(K',\one), \one\bigr) & & \text{(by \eqref{eq:F^M_P-adjunction}).}
\end{align*}
Passing to the $i$-th cohomology yields isomorphisms
\begin{align*}
\H^i\bigl(F^1_{K}(\one)\bigr) &\cong \varinjlim_{K'\le K}
\Hom_{\D(k)}(\RR\H^0(K',\one),\one[i]\bigr)\\
&\cong \varinjlim_{K'\le K} \Hom_k\bigl(\H^{-i}(K',\one), k\bigr).
\end{align*}
Here, the transition maps are induced by the corestriction
\[
\cores\colon \H^{-i}(K'',\one) \to \H^{-i}(K',\one),
\] 
for $K''\subseteq K'$. Now, (5) in the proof of \cite[Ch.~I,
Prop.~30]{Serre.2013} implies $\H^i\bigl(F^1_{K}(\one)\bigr) = 0$ provided
$i\neq -\dim K$. By (4) in the proof of \loccit\ the above corestriction maps are
isomorphisms for $i = -\dim K$. Therefore, we have $F^1_{K}(\one) \cong
\H^{\dim K}(K,\one)^*[\dim K]$. Since $\H^{\dim K}(K,\one)$ is one-dimensional,
it follows that $F^1_{K}(\one) \cong \delta[\dim K]$, for some smooth character
$\delta \colon K\to k^\times$. Since $K$ is a pro-$p$ group and $\Char k =
p$, it follows that $\delta$ is the trivial character, which proves
$F^1_{K}(\one) \cong k[\dim K]$. 
\end{proof} 

Proposition~\ref{prop:dualizing} verifies condition (W4) in \cite[1.9.
Thm.]{Balmer.2016}. Precisely, this means:
\begin{cor}\label{cor:tower} 
Put $f^{(n)} \coloneqq \omega_{K_P}^{\otimes n}\otimes_k \Inf^{K_M}_{K_P}$ and
$f_{(n)} \coloneqq \RR\H^0\bigl(K_U, \omega_{K_P}^{\otimes n}\otimes_k \pholder\bigr)$,
for $n\in\Z$. There is an infinite chain of adjunctions
\[
\dotsb \dashv f^{(-1)} \dashv f_{(1)} \dashv f^{(0)} \dashv f_{(0)} \dashv
f^{(1)} \dashv f_{(-1)} \dashv \dotsb 
\]
\end{cor} 
\subsection{The general case}\label{subsec:general} 
In this section we fix an open surjective morphism $f\colon
S\longtwoheadrightarrow T$ of $p$-adic monoids with kernel $U$. The functor
$\Inf^T_S\colon \D(T)\to \D(S)$ is strictly monoidal and admits the right
adjoint $\RR\Pi_U$, cf.~Proposition~\ref{prop:Inf-adjoints} and
Lemma~\ref{lem:K-injectives}.\ref{lem:K-injectives-c}. The goal of this
section is to show that $\Inf^T_S$ also admits a left adjoint and study its
properties. We make the following observation:

\begin{prop}\label{prop:RPi-hom} 
For all $X\in \D(S)$ and $Y\in \D(T)$ the natural map
\[
\hom_{\D(T)}\bigl(Y, \RR\Pi_U(X)\bigr) \xrightarrow{\cong} \RR\Pi_U
\bigl(\hom_{\D(S)}(\Inf^T_SY,X)\bigr)
\]
is an isomorphism in $\D(T)$.
\end{prop} 
\begin{proof} 
Let $W\in \D(T)$. Since $\Inf^T_S$ is strictly monoidal, the natural
transformation
\[
(\Inf^T_SY\otimes_k\pholder)\circ \Inf^T_S \xRightarrow{\cong} \Inf^T_S \circ
(Y\otimes_k\pholder)
\]
is an isomorphism of functors $\D(T)\to \D(S)$. Passing to the right adjoints,
see Example~\ref{ex:mates}, shows that the natural map
\[
\hom_{\D(T)}(Y,\pholder)\circ \RR\Pi_U \xRightarrow{\cong} \RR\Pi_U\circ
\hom_{\D(S)}(\Inf^T_SY,\pholder)
\]
is an isomorphism of functors $\D(S)\to \D(T)$. 
\end{proof} 

\begin{lem}\label{lem:Res-conservative} 
Let $K_S\subseteq S$ be an open subgroup. The restriction functor
$\Res^S_{K_S}\colon \D(S)\to \D(K_S)$ is conservative and preserves small
products and small coproducts.
\end{lem} 
\begin{proof} 
Note that if $\varphi$ is a morphism in $\Chain(S)$, then
$\Res^S_{K_S}(\varphi)$ is a quasi-isomorphism if and only if $\varphi$ is a
quasi-isomorphism. This implies that $\Res^S_{K_S}\colon \D(S)\to \D(K_S)$ is
conservative.

Since $K_S\subseteq S$ is an open subgroup, the functor
$\Res^S_{K_S}\colon \Rep_k(S)\to \Rep_k(K_S)$ admits an exact left adjoint
$\ind_{K_S}^S$ and a left exact right adjoint $\Ind_{K_S}^S$, by
Lemmas~\ref{lem:FrobeniusII} and~\ref{lem:FrobeniusI}, respectively. But then
the induced functors $\ind_{K_S}^S, \RInd_{K_S}^S\colon \D(K_S)\to \D(S)$ are
left adjoint, resp.\ right adjoint, to $\Res^S_{K_S}$. By
Proposition~\ref{prop:adjoint-limit}, $\Res^S_{K_S}$ preserves small products
and small coproducts.
\end{proof} 

\begin{thm}\label{thm:Inf-prod-general} 
The functor $\Inf^T_S\colon \D(T)\to \D(S)$ preserves small products. In
particular, it admits a left adjoint, denoted $\LL_U$.
\end{thm} 
\begin{proof} 
Let $K_S \subseteq S$ be an open, compact, torsion-free $p$-adic Lie group such
that $K_T\coloneqq f(K_S) \subseteq T$ is open and torsion-free. 

Let $I$ be a set, and take $Y_i\in \D(T)$, for $i\in I$. We show that the
natural map
\[
\alpha_S\colon \Inf^T_S \prod_{i\in I}Y_i \to \prod_{i\in I} \Inf^T_SY_i.
\]
is an isomorphism in $\D(S)$. Note that $\Res^S_{K_S}\Inf^T_S \cong
\Inf^{K_T}_{K_S}\Res^T_{K_T}$ as functors $\D(T)\to \D(K_S)$. By
Lemma~\ref{lem:Res-conservative} the functors $\Res^S_{K_S}$ and $\Res^T_{K_T}$
preserve small products. Now, the diagram
\[
\begin{tikzcd}[column sep=5em,row sep=3em]
\Res^S_{K_S} \Inf^T_S \prod_{i\in I}Y_i \ar[r,"\Res^S_{K_S}(\alpha_S)"]
\ar[d,"\cong"'] & \Res^S_{K_S}\prod_{i\in I}\Inf^T_SY_i \ar[d,"\cong"]\\
\Inf^{K_T}_{K_S} \prod_{i\in I} \Res^T_{K_T}Y_i \ar[r,"\alpha_{K_S}"'] &
\prod_{i\in I} \Inf^{K_T}_{K_S}\Res^S_{K_S}Y_i
\end{tikzcd}
\]
is commutative. By Corollary~\ref{cor:Inf-product} the bottom map is an
isomorphism in $\D(K_S)$. Therefore, the top map $\Res^S_{K_S}(\alpha_S)$ is an
isomorphism in $\D(K_S)$. By Lemma~\ref{lem:Res-conservative} again,
$\Res^S_{K_S}$ is conservative, which shows that $\alpha_S$ is an isomorphism
in $\D(S)$. Hence, $\Inf^T_S$ preserves small products. By the
Corollary~\ref{cor:brown}.\ref{cor:brown-b} of Brown representability, it
follows that $\Inf^T_S$ admits a left adjoint.
\end{proof} 

In fact, the adjunction $\LL_U\dashv \Inf^T_S$ holds on the level of $\RHom$
complexes:
\begin{cor}\label{cor:RHom-LU-Inf} 
There is a natural isomorphism
\begin{equation}\label{eq:RHom-LU-Inf}
\RHom_{\Rep_k(T)}\bigl(\LL_U(X),Y\bigr) \xrightarrow{\cong}
\RHom_{\Rep_k(S)}\bigl(X,\Inf^T_S Y\bigr)\qquad \text{in $\D(k)$.}
\end{equation}
\end{cor} 
\begin{proof} 
Since $\Inf^T_S$ is an exact functor, it induces a natural morphism
\begin{equation}\label{eq:RHom-Inf}
\RHom_{\Rep_k(T)}(Y_1,Y_2) \to \RHom_{\Rep_k(S)}\bigl(\Inf^T_SY_1,
\Inf^T_SY_2\bigr),
\end{equation}
for all $Y_1,Y_2\in \D(T)$. Let now $X\in \D(S)$ and $Y\in \D(T)$, and denote
$\eta_X\colon X\to \Inf^T_S\LL_U(X)$ the unit of the adjunction.
Then \eqref{eq:RHom-LU-Inf} is given as the composite
\begin{align*}
\RHom_{\Rep_k(T)}\bigl(\LL_U(X),Y\bigr) &\xrightarrow{\eqref{eq:RHom-Inf}}
\RHom_{\Rep_k(S)}\bigl(\Inf^T_S\LL_U(X),\Inf^T_SY\bigr)\\
&\xrightarrow{(\eta_V)^*} \RHom_{\Rep_k(S)}\bigl(X, \Inf^T_SY\bigr).
\end{align*}
On the $n$-th cohomology it is given by the adjunction isomorphism
\[
\Hom_{\D(T)}\bigl(\LL_U(X),Y[n]\bigr) \xrightarrow{\cong} \Hom_{\D(S)}\bigl(X,
\Inf^T_SY[n]\bigr).
\]
Therefore, \eqref{eq:RHom-LU-Inf} is an isomorphism in $\D(k)$.
\end{proof} 

\begin{notation}\label{nota:L^n} 
Given $n\in\Z$, we denote by $\LL^n_U \colon \Rep_k(S)\to \Rep_k(T)$ the $n$-th
cohomology functor of $\LL_U$, \ie,
\[
\LL^n_U(V) \coloneqq \H^n(\LL_UV[0]),\qquad \text{for all $V\in
\Rep_k(S)$.}
\]
\end{notation} 

\begin{prop}\label{prop:Ln<0} 
The functor $\LL_U\colon \D(S)\to \D(T)$ is right $t$-exact. More precisely, we
have $\LL^n_U = 0$ for all $n>0$, and $\LL^0_U$ is the left adjoint of
$\Inf^T_S\colon \Rep_k(T)\to\Rep_k(S)$.
\end{prop} 
\begin{proof} 
The right $t$-exactness is a formal consequence of the (left) $t$-exactness of
$\Inf^T_S$. Let us show that $\LL^0_U$ is the left adjoint of $\Inf^T_S$. Let
$\D^{\ge0}(T)$ be the full subcategory of $\D(T)$ consisting of complexes $Y$
with $\H^n(Y) = 0$ for all $n<0$. The truncation functor $\tau^{\ge0} \colon
\D(T)\to \D^{\ge0}(T)$ is left adjoint to the inclusion. Hence, given $V\in
\Rep_k(S)$ and $W\in \Rep_k(T)$, we compute
\begin{align*}
\Hom_{\Rep_k(T)}\bigl(\LL^0_U(V),W\bigr) &\cong \Hom_{\D^{\ge0}(T)}
\bigl(\tau^{\ge0}\LL_U(V), W\bigr)\\
&= \Hom_{\D(T)}\bigl(\LL_U(V),W\bigr)\\
&\cong \Hom_{\D(S)}\bigl(V, \Inf^T_SW\bigr)\\
&\cong \Hom_{\Rep_k(S)}(V,\Inf^T_SW). \qedhere 
\end{align*}
\end{proof} 

\begin{rmk*} 
Proposition~\ref{prop:Ln<0} implies that there is a natural transformation
\[
\LL_U\circ \q_S \To \q_T\circ \LL^0_U
\]
of functors $\K(S)\to \D(T)$.
I do not know whether the left derived functor of $\LL^0_U$ exists and, if it
does exist, whether it coincides with $\LL_U$.
\end{rmk*} 

\begin{cor}\label{cor:RHom-LU-Inf-spectral} 
Given $V\in \Rep_k(S)$ and $W\in \Rep_k(T)$, there is a convergent
first-quadrant spectral sequence
\[
E_2^{i,j} = \Ext^i_{\Rep_k(T)}\bigl(\LL^{-j}_U(V),W\bigr) \To
\Ext^{i+j}_{\Rep_k(S)}\bigl(V,\Inf^T_SW\bigr).
\]
In particular, there is a five-term exact sequence
\[
\begin{tikzcd}[column sep=2em]
0\ar[r] & \Ext^1_{\Rep_k(T)}\bigl(\LL^0_U(V),W\bigr) \ar[r] &
\Ext^1_{\Rep_k(S)}\bigl(V,\Inf^T_SW\bigr) \ar[r] 
\ar[d,phantom,""{coordinate, name=Z}] &
\Hom_{\Rep_k(T)}\bigl(\LL^{-1}_U(V),W\bigr) 
\ar[dll,"{d_2^{0,1}}" description, rounded corners, at end, to path={
-- ([xshift=2ex]\tikztostart.east) |- (Z) \tikztonodes -|
([xshift=-2ex]\tikztotarget.west) -- (\tikztotarget)}]
\\
& \Ext^2_{\Rep_k(T)}\bigl(\LL^0_U(V),W\bigr) \ar[r] & \Ext^2_{\Rep_k(S)}\bigl(V,
\Inf^T_SW\bigr).
\end{tikzcd}
\]
\end{cor} 
\begin{proof} 
The functor $\RHom_{\Rep_k(T)}(\pholder,W)$ is left $t$-exact and $\LL_U(V)$ is
right bounded by Proposition~\ref{prop:Ln<0}. Now,
Lemma~\ref{lem:spectral}.\ref{lem:spectral-b} together with
Corollary~\ref{cor:RHom-LU-Inf} yields the desired spectral sequence.
\end{proof} 

\begin{prop}\label{prop:LU-transitive} 
Let $f'\colon T\longtwoheadrightarrow T'$ be another open surjective
morphism of $p$-adic monoids. Write $U' = \Ker(f')$ and $U'' = \Ker(f'\circ f)$.
\begin{enumerate}[label=(\alph*)]
\item\label{prop:LU-transitive-a} The diagram
\[
\begin{tikzcd}
\D(S) \ar[r,"\LL_U"] \ar[dr,"\LL_{U''}"'] & \D(T) \ar[d,"\LL_{U'}"]\\
& \D(T')
\end{tikzcd}
\]
is commutative. 
\item\label{prop:LU-transitive-b} Given $V\in \Rep_k(S)$, there is a convergent
third-quadrant spectral sequence
\[
E_2^{i,j} = \LL^i_{U'}\bigl(\LL^j_{U}(V)\bigr) \To
\LL^{i+j}_{U''}(V).
\]
In particular, there is a five-term exact sequence
\[
\begin{tikzcd}[column sep=1.2em]
\LL^{-2}_{U''}(V) \ar[r] & \LL^{-2}_{U'}\bigl(\LL^0_U(V)\bigr)
\ar[r,"{d^{-2,0}}"] &[1.3em] \LL^0_{U'}\bigl(\LL^{-1}_U(V)\bigr) \ar[r] &
\LL^{-1}_{U''}(V) \ar[r] & \LL^{-1}_{U'}\bigl(\LL^0_U(V)\bigr) \ar[r] & 0.
\end{tikzcd}
\]
\end{enumerate}
\end{prop} 
\begin{proof} 
\ref{prop:LU-transitive-a} follows from the natural isomorphism $\Inf^T_S
\Inf^{T'}_T \xRightarrow{\cong} \Inf^{T'}_S$ by passing to the left adjoints,
see Example~\ref{ex:mates}. 
Now, \ref{prop:LU-transitive-b} follows from \ref{prop:LU-transitive-a} by
applying Lemma~\ref{lem:spectral}.\ref{lem:spectral-a} and observing that the
functors $\LL_{U'}$ and $\LL_U$ are right $t$-exact by
Proposition~\ref{prop:Ln<0}.
\end{proof} 

\begin{prop}\label{prop:LU-ind} 
Let $S'\subseteq S$ be an open submonoid and put $T'\coloneqq f(S') \subseteq
T$. Put $U'\coloneqq U\cap S'$. Assume further that $k[S]$ is flat over
$k[S']$ and $k[T]$ is flat over $k[T']$. The following diagram is
commutative:
\[
\begin{tikzcd}
\D(S') \ar[d,"\ind_{S'}^S"'] \ar[r,"\LL_{U'}"] & \D(T') \ar[d,"\ind_{T'}^T"]\\
\D(S) \ar[r,"\LL_U"'] & \D(T).
\end{tikzcd}
\]
In particular, there is, for each $n\in\Z$, a natural isomorphism $\ind_{T'}^T
\LL^n_{U'} \xRightarrow{\cong} \LL^n_{U} \ind_{S'}^S$ of functors
$\Rep_k(S')\to \Rep_k(T)$.
\end{prop} 
\begin{proof} 
Clearly there is a natural isomorphism $\Res^S_{S'}\Inf^T_S \xRightarrow{\cong}
\Inf^{T'}_{S'}\Res^T_{T'}$. Passing to the left adjoints, see
Example~\ref{ex:mates}, yields the first assertion. The second assertion follows
from the first by passing to the $n$-th cohomology and using that $\ind_{T'}^T$
and $\ind_{S'}^S$ are exact.
\end{proof} 

\begin{prop}\label{prop:LU-character} 
Let $\chi\colon T\to k^\times$ be a smooth character viewed as a character of
$S$ via inflation along $f$. There is a natural isomorphism $\LL_U
(\chi\otimes_k V) \xrightarrow{\cong} \chi\otimes_k \LL_U(V)$, for each $V\in
\D(S)$.
\end{prop} 
\begin{proof} 
The functor $\chi\otimes_k\pholder$ is an equivalence of triangulated categories with
quasi-inverse
$\chi^{-1}\otimes_k\pholder$. In particular, $\chi\otimes_k\pholder$ is left
adjoint to $\chi^{-1}\otimes_k\pholder$. Moreover,
there is a natural isomorphism $\Inf^T_S(\chi^{-1}\otimes_k W)
\xrightarrow{\cong} \chi^{-1}\otimes_k \Inf^T_S(W)$, for $W\in \D(T)$. Passing
to the left adjoints, see Example~\ref{ex:mates}, yields the result.
\end{proof} 

\begin{rmk} 
If $S$ and $T$ are $p$-adic Lie groups, we will obtain in
Corollary~\ref{cor:projectionformula} a stronger version of the above
proposition.
\end{rmk} 

Next, we will study how $\LL_U$ behaves under restriction. This requires some
preparatory lemmas. We will employ a version of the Mackey decomposition for
derived categories.

\begin{notation}\label{nota:g_*} 
Let $G$ be locally profinite group and $S\subseteq G$ a closed submonoid. Given
$g\in G$ and $V\in \Rep_k(S)$, we define a smooth representation of $gSg^{-1}$
on $V$ via $(gsg^{-1})\star v \coloneqq sv$, for all $s\in S$, $v\in V$. We
denote this modified representation by $g_*V$. In this way, we obtain a functor
\[
g_*\colon \Rep_k(S) \to \Rep_k(gSg^{-1})
\]
which is easily seen to be an equivalence of categories. The induced equivalence
on the derived categories is again denoted $g_*$.
\end{notation} 

\begin{lem}[derived Mackey decomposition]\label{lem:derived-Mackey} 
Assume that $S$ is a locally profinite monoid. Let
$K,H\subseteq S$ be two subgroups with $K$ open and $H$ closed.
\begin{enumerate}[label=(\alph*)]
\item\label{lem:derived-Mackey-a} There is a natural
isomorphism of functors $\D(H)\to \D(K)$,
\[
\Res^S_K\RInd_H^S \xRightarrow{\cong} \prod_{s\in H\backslash S/K}
\Ind_{s^{-1}Hs\cap K}^K s_*^{-1}\Res^{H}_{H\cap sKs^{-1}}.
\]
\item\label{lem:derived-Mackey-b} There is a natural isomorphism of functors
$\D(K)\to \D(H)$,
\[
\Res^S_H \ind_K^S \xLeftarrow{\cong} \bigoplus_{s\in H\backslash S/K}
\ind_{H\cap sKs^{-1}}^H s_*\Res^{K}_{s^{-1}Hs\cap K}.
\]
\end{enumerate}
\end{lem} 
\begin{proof} 
Note that $\Ind^K_{s^{-1}Hs\cap K}$ is exact by Remark~\ref{rmk:Ind-exact}.
Both statements follow from the Mackey decomposition,
\cite[Thm.~1.1]{Yamamoto.2022}, by passing to the derived categories. For
\ref{lem:derived-Mackey-a} we observe that the functors $\Res^{H}_{H\cap
sKs^{-1}}$, $s_*^{-1}$, and $\Ind_{s^{-1}Hs\cap K}^K$ preserve
K-injective complexes by Lemma~\ref{lem:K-injectives}.\ref{lem:K-injectives-b},
since they all admit exact left adjoints.
\end{proof} 

Recall that we denoted by $U$ the kernel of $f\colon S\longtwoheadrightarrow T$.

\begin{lem}\label{lem:monoid-coset} 
Let $H_S\subseteq S$ be a subgroup and put $H_T \coloneqq f(H_S)$. Assume that
$f$ induces a bijection $H_S\backslash S\xrightarrow{\cong} H_T\backslash T$. 
\begin{enumerate}[label=(\alph*)]
\item\label{lem:monoid-coset-a} One has $sU\subseteq H_Ss$, for all $s\in S$. In
particular, $U\subseteq H_S$ is a group.

\item\label{lem:monoid-coset-b} Let $K_S\subseteq S$ be a subgroup and put
$K_T\coloneqq f(K_S)$. The map $f$ induces a
bijection $H_S\backslash S/K_S \xrightarrow{\cong} H_T\backslash T/K_T$.
\end{enumerate}
\end{lem} 
\begin{proof} 
The first assertion in \ref{lem:monoid-coset-a} is immediate from the
hypothesis, because $f(sU) = \{f(s)\} \subseteq H_Tf(s)$ is contained in a
single coset. For $s=1$ this shows $U \subseteq H_S$. Thus, $U = U\cap H_S$ is
the kernel of the group homomorphism $f\big|_{H_S}$, hence a group.
For~\ref{lem:monoid-coset-b}, note that $K_S$ acts on $H_S\backslash S$ by
right multiplication. By~\ref{lem:monoid-coset-a}, this action factors through
$K_T$, and this implies the assertion.
\end{proof} 

\begin{lem}\label{lem:Inf-Ind} 
Let $K_P$ be a profinite group and $K_U\subseteq K'_P\subseteq K_P$ closed
subgroups with $K_U$ normal in $K_P$. Write $K_M \coloneqq K_P/K_U$ and
$K'_M\coloneqq K'_P/K_U$. 

The natural transformation
\begin{equation}\label{eq:Inf-Ind}
\Inf^{K_M}_{K_P} \Ind_{K'_M}^{K_M} \xRightarrow{\cong} \Ind_{K'_P}^{K_P}
\Inf^{K'_M}_{K'_P}
\end{equation}
is an isomorphism of functors $\D(K'_M) \to \D(K_P)$.
\end{lem} 
\begin{proof} 
It suffices to show that \eqref{eq:Inf-Ind} is an isomorphism of functors
$\Rep_k(K'_M)\to \Rep_k(K_P)$, since by Remark~\ref{rmk:Ind-exact}, the
underived functors $\Ind_{K'_M}^{K_M}$ and $\Ind_{K'_P}^{K_P}$ are exact. But
this can be checked directly or by observing that
the natural map $\LL^0_{K_U}\Res^{K_P}_{K'_P} \xRightarrow{\cong}
\Res^{K_M}_{K'_M}\LL^0_{K_U}$ is obviously an isomorphism of functors
$\Rep_k(K_P)\to \Rep_k(K'_M)$, and then passing to the right adjoints,
see Example~\ref{ex:mates}.
\end{proof} 

\bigskip
For the rest of this section we assume that $S$ and $T$ are open submonoids of
$p$-adic Lie groups $P$ and $M$, respectively, and that $f$ arises from a
continuous surjection $P\longtwoheadrightarrow M$. This condition will be
satisfied for the $p$-adic monoids we consider. Note that, since $M$ with the
quotient topology is a $p$-adic Lie group, \cite[Thm.~9.6 (ii)]{DDMS99}, and the
$p$-adic analytic structure on $M$ is unique, \cite[Cor.~9.5]{DDMS99}, it
follows that $P\longtwoheadrightarrow M$ is a quotient map of topological
groups, hence is open. Therefore, $f$ is automatically open. 

\begin{prop}\label{prop:LU-Res} 
Let $H_S\subseteq S$ be a closed subgroup and put $H_T\coloneqq f(H_S)
\subseteq T$. Assume that $f$ induces a bijection $H_S\backslash
S\xrightarrow{\cong} H_T\backslash T$. The diagrams
\[
\begin{tikzcd}
\D(H_T) \ar[r,"\Inf^{H_T}_{H_S}"] \ar[d,"\RInd_{H_T}^T"'] & \D(H_S)
\ar[d,"\RInd_{H_S}^S"] &[2em] 
\D(S) \ar[d,"\Res^S_{H_S}"'] \ar[r,"\LL_U"] & \D(T) \ar[d,"\Res^T_{H_T}"]\\
\D(T) \ar[r,"\Inf^T_S"'] & \D(S) & 
\D(H_S) \ar[r,"\LL_U"'] & \D(H_T)
\end{tikzcd}
\]
commute.
\end{prop} 
\begin{proof} 
Note that $H_T$ is closed, since
Lemma~\ref{lem:monoid-coset}.\ref{lem:monoid-coset-a} implies $U\subseteq H_S$
so that $T\setminus H_T = f(S\setminus H_S)$ is open.
Consider the natural transformation
\begin{equation}\label{eq:Inf-Ind-1}
\beta\colon \Inf^T_S \RInd_{H_T}^T \To \RInd_{H_S}^S \Inf^{H_T}_{H_S},
\end{equation}
which arises as the right mate of the natural map $\Res^S_{H_S}\Inf^T_S
\xRightarrow{\cong} \Inf^{H_T}_{H_S}\Res^T_{H_T}$ by
Proposition~\ref{prop:mates}. It suffices to show that $\beta$ is an
isomorphism, because then its mate $\LL_{U}\Res^S_{H_S} \To
\Res^T_{H_T}\LL_{U}$ is an isomorphism as well by Example~\ref{ex:mates}, and
hence both diagrams commute.

Fix an open profinite subgroup $K_S\subseteq S$ so that $K_T\coloneqq f(K_S)$
is open in $T$. Since $\Res^S_{K_S}$ is conservative by
Lemma~\ref{lem:Res-conservative},
it suffices to show that $\Res^S_{K_S}(\beta)$ is an isomorphism. We compute
the source and target of $\Res^S_{K_S}(\beta)$ using the derived Mackey
decomposition, Lemma~\ref{lem:derived-Mackey}, and the fact that
$\Inf^{K_T}_{K_S}$ preserves small products,
Theorem~\ref{thm:Inf-prod-general}:
\begin{align*}
\Res^S_{K_S}\Inf^T_S\RInd^T_{H_T} &\xRightarrow{\cong} \Inf^{K_T}_{K_S}
\Res^T_{K_T}\RInd^T_{H_T}\\
&\xRightarrow{\cong}\Inf^{K_T}_{K_S} \prod_{t\in H_T\backslash T/K_T}
\Ind^{K_T}_{t^{-1}H_Tt\cap K_T} t_*^{-1}\Res^{H_T}_{H_T\cap tK_Tt^{-1}}
\\
&\xRightarrow{\cong}\prod_{t\in H_T\backslash T/K_T}\Inf^{K_T}_{K_S} 
\Ind^{K_T}_{t^{-1}H_Tt\cap K_T} t_*^{-1}\Res^{H_T}_{H_T\cap tK_Tt^{-1}}
\\
\intertext{and}
\Res^S_{K_S}\RInd_{H_S}^S\Inf^{H_T}_{H_S} &\xRightarrow{\cong} 
\prod_{s\in H_S\backslash S/K_S} \Ind_{s^{-1}H_Ss\cap K_S}^{K_S}
s_*^{-1} \Res^{H_S}_{H_S\cap sK_Ss^{-1}} \Inf^{H_T}_{H_S}
\end{align*}
By Lemma~\ref{lem:monoid-coset} we may identify the double coset spaces
$H_S\backslash S/K_S$ and $H_T\backslash T/K_T$. 

Let now $s\in S$ and write $\ol{s} \coloneqq f(s) \in T$. By
Lemma~\ref{lem:monoid-coset}.\ref{lem:monoid-coset-a}, we have $U\cap K_S
\subseteq s^{-1}H_Ss\cap K_S$, so that $f(s^{-1}H_Ss\cap K_S) = \ol{s}^{-1}H_T
\ol{s}\cap K_T$. To shorten the notation, we write
\begin{align*}
\Psi_s &\coloneqq \Ind^{K_S}_{s^{-1}H_Ss\cap K_S} s_*^{-1} \Res^{H_S}_{H_S\cap
sK_Ss^{-1}}\\
\intertext{and}
\Psi_{\ol s} &\coloneqq \Ind^{K_T}_{\ol{s}^{-1}H_T\ol{s} \cap K_T} \ol{s}_*^{-1}
\Res^{H_T}_{H_T\cap \ol{s}K_T\ol{s}^{-1}}.
\end{align*}
It follows from Lemma~\ref{lem:Inf-Ind} that the natural map
\[
\beta_s\colon \Inf^{K_T}_{K_S} \Psi_{\ol s}
\xRightarrow{\cong} \Psi_s
\Inf^{H_T}_{H_S}
\]
is an isomorphism. Let us verify that
$\Res^S_{K_S}(\beta)$ corresponds to $\prod_{s\in H_S\backslash S/K_S}
\beta_s$ under the above identifications; it then follows that
$\Res^S_{K_S}(\beta)$ is an isomorphism, which finishes the proof.

It remains to show that the diagram
\[
\begin{tikzcd}
\Res^S_{K_S}\Inf^T_S\RInd^T_{H_T} \ar[r,Rightarrow,"\cong"]
\ar[dr,Rightarrow,"\Res^S_{K_S}(\beta)"'] & \Inf^{K_T}_{K_S}
\Res^T_{K_T}\RInd^T_{H_T} \ar[d,Rightarrow,"\gamma"'] \ar[r,Rightarrow,"\cong"]
& \prod_{s\in H_S\backslash S/K_S} \Inf^{K_T}_{K_S} \Psi_{\ol s}
\ar[d,Rightarrow,"\prod_s\beta_s"]\\
& \Res^S_{K_S}\RInd_{H_S}^S \Inf^{H_T}_{H_S} \ar[r,Rightarrow, "\cong"] &
\prod_{s\in H_S\backslash S/K_S} \Psi_s \Inf^{H_T}_{H_S}
\end{tikzcd}
\]
commutes. The left vertical map $\gamma$ is defined so as to make the  triangle
commute. The commutativity of the square can be checked componentwise.
Therefore, it suffices to show that the diagram
\begin{equation}\label{eq:Inf-Mackey}
\begin{tikzcd}[column sep=4em]
\Inf^{K_T}_{K_S} \bigl(\Res^T_{K_T}\RInd^T_{H_T}\bigr)
\ar[d,Rightarrow,"\gamma"'] 
\ar[r,Rightarrow, "\Inf^{K_T}_{K_S}\rho_{\ol s}"] & \Inf^{K_T}_{K_S} 
\Psi_{\ol s} \ar[d,Rightarrow,"\beta_s"]\\
\bigl(\Res^S_{K_S}\RInd^S_{H_S}\bigr) \Inf^{H_T}_{H_S}
\ar[r,Rightarrow,"(\rho_s)_{\Inf^{H_T}_{H_S}}"'] & \Psi_s \Inf^{H_T}_{H_S}
\end{tikzcd}
\end{equation}
commutes, where $\rho_{\ol s}$ and $\rho_s$ are the natural maps coming from the
Mackey decomposition. The left mate of $\rho_s$ is the map
\[
\lambda_s\colon \Phi_s\coloneqq \ind_{H_S\cap sK_Ss^{-1}}^{H_S} s_*
\Res^{K_S}_{s^{-1}H_Ss\cap K_S} \To \Res^S_{H_S} \ind_{K_S}^S
\]
that is the $s$-th component of the isomorphism in the Mackey decomposition for
compact induction. The analogous statement holds for the left mate
\[
\lambda_{\ol s}\colon \Phi_{\ol s}\coloneqq \ind_{H_T\cap \ol s K_T\ol{s}^{-1}}
^{H_T} \ol{s}_* \Res^{K_T}_{\ol{s}^{-1}H_T\ol{s}\cap K_T} \To
\Res^T_{H_T} \ind_{K_T}^T
\]
of $\rho_{\ol s}$. By Example~\ref{ex:mates-vertical} the commutativity of
\eqref{eq:Inf-Mackey} is equivalent to the commutativity of
\[
\begin{tikzcd}[column sep=4em]
\bigl(\Res^S_{H_S}\ind_{K_S}^S\bigr) \Inf^{K_T}_{K_S}
\ar[d,Rightarrow,"\alpha"'] & 
\Phi_s \Inf^{K_T}_{K_S} \ar[l,Rightarrow,"(\lambda_s)_{\Inf^{K_T}_{K_S}}"']
\ar[d,Rightarrow,"\alpha_s"]\\
\Inf^{H_T}_{H_S} \bigl(\Res^T_{H_T}\ind_{K_T}^T\bigr) & \Inf^{H_T}_{H_S}
\Phi_{\ol s} \ar[l,Rightarrow, "\Inf^{H_T}_{H_S}\lambda_{\ol s}"],
\end{tikzcd}
\]
where $\alpha$ (resp.\ $\alpha_s$) is the left mate of $\gamma$ (resp.\ $\beta_s$).
But this can be checked explicitly on the underived level, because
all functors involved are exact.
\end{proof} 

\begin{rmk}\label{rmk:Inf-preserves-acyclics} 
Under the hypotheses of Proposition~\ref{prop:LU-Res} it follows that the
natural map
\[
\RR\bigl(\Ind_{H_S}^S \Inf^{H_T}_{H_S}\bigr) \xRightarrow{\cong}
\RInd_{H_S}^S \Inf^{H_T}_{H_S}
\]
is an isomorphism. Thus, if $I\in \Rep_k(H_T)$ is injective, then
$\Inf^{H_T}_{H_S}I$ is acyclic for $\Ind_{H_S}^S$.
\end{rmk} 

\begin{rmk} 
One can show the following analog of Proposition~\ref{prop:LU-Res}: Let $K_S
\subseteq S$ be an open subgroup and $K_T \coloneqq f(K_S)\subseteq T$. Assume
that $f$ induces a bijection $S/K_S \xrightarrow{\cong} T/K_T$. Then the
diagrams
\[
\begin{tikzcd}
\D(K_T) \ar[r,"\Inf^{K_T}_{K_S}"] \ar[d,"\ind_{K_T}^T"'] & \D(K_S)
\ar[d,"\ind_{K_S}^S"] &[2em] 
\D(S) \ar[d,"\Res^S_{K_S}"'] \ar[r,"\RR\Pi_U"] & \D(T) \ar[d,"\Res^T_{K_T}"]\\
\D(T) \ar[r,"\Inf^T_S"'] & \D(S) & 
\D(K_S) \ar[r,"\RR\Pi_U"'] & \D(K_T)
\end{tikzcd}
\]
commute. Indeed, the bijection $S/K_S\xrightarrow{\cong} T/K_T$ implies that $U$
is a subgroup of $K_S$ and that the natural map
\[
k[S] \otimes_{k[K_S]} \Inf^{K_T}_{K_S}(\pholder) \xRightarrow{\cong} \Inf^T_S \circ
(k[T]\otimes_{k[K_T]} \pholder)
\]
is an isomorphism of functors $\Rep_k(K_T)\to \Rep_k(S)$. As all functors
involved are exact, we deduce that the left diagram commutes. The commutativity
of the right diagram then follows by passing to the right adjoints, see
Example~\ref{ex:mates}.
\end{rmk} 
\begin{warning} 
Let $H_S\subseteq S$ be a subgroup and put $H_T \coloneqq f(H_S)$. Then
$H_S\backslash S\xrightarrow[\cong]{f} H_T\backslash T$ does \emph{not} imply
(and is not implied by) $S/H_S \xrightarrow[\cong]{f} T/H_T$. 

For example, this generally fails for the $p$-adic monoids considered in
\S\ref{subsec:positive}. Let us make this explicit:
Consider the $p$-adic submonoid $S = P^+$ of $\GL_2(\Q_p)$
consisting of upper triangular matrices $\begin{psmallmatrix}a & b\\0 & d
\end{psmallmatrix}$ satisfying $\val_p(a)\ge \val_p(d)$ and $\val_p(b)\ge
\val_p(d)$, where
$\val_p$ denotes the $p$-adic valuation on $\Q_p$. Let $T = M^+$ be the $p$-adic
submonoid of $P^+$ consisting of diagonal matrices. The canonical projection
$f\colon P^+\longtwoheadrightarrow M^+$ (given by forgetting the upper right
entry) is surjective with kernel $U =
\begin{psmallmatrix}1 & \Z_p\\ 0 & 1\end{psmallmatrix}$. Now, $P^+ = UM^+$ and
hence $U\backslash P^+ \xrightarrow{\cong} M^+$ via $f$. On the other hand,
for any $m = \diag(a,d)$ in $M^+$ with $\val_p(a)>\val_p(d)$ we have $mUm^{-1}
\subsetneqq U$, from which it follows that the map
\[
P^+/U = UM^+/U \cong \bigsqcup_{m\in M^+} U/mUm^{-1}
\longtwoheadrightarrow \bigsqcup_{m\in M^+}\{m\} = M^+
\]
is (surjective but) not injective.
\end{warning} 

\begin{cor}\label{cor:Hom-LU} 
Assume that $U$ is a group and that $f$ induces a bijection
$U\backslash S\xrightarrow{\cong} T$. There is a natural isomorphism
\[
\Hom_{\D(k)}\bigl(\Res^T_1(\LL_U X), Y\bigr) \cong \Hom_{\D(U)}\bigl(
\Res^S_U X, \Inf^1_U Y\bigr),
\]
for all $X\in \D(S)$ and $Y\in \D(k)$. 
In particular, whenever $n\in\Z$, $V\in \Rep_k(S)$, and $W\in \Vect_k$, there
is a natural isomorphism
\[
\Hom_k\bigl(\LL^{-n}_U(V),W\bigr) \cong \Ext^n_{U}(V,W).
\]
\end{cor} 
\begin{proof} 
Applying Proposition~\ref{prop:LU-Res} with $H_S = U$ and $H_T = \{1\}$ and
Theorem~\ref{thm:Inf-prod-general} yields natural isomorphisms
\begin{align*}
\Hom_{\D(k)}\bigl(\Res^T_1(\LL_UX), Y\bigr) &\cong \Hom_{\D(k)}\bigl(
\LL_U(\Res^S_UX), Y\bigr)\\
&\cong \Hom_{\D(U)}\bigl(\Res^S_UX, \Inf^1_UY\bigr),
\end{align*}
for each $X\in \D(S)$ and $Y\in \D(k)$. The last statement follows from the
first by taking $X = V[0]$ and $Y = W[n]$ for $V\in \Rep_k(S)$ and $W\in
\Vect_k$.
\end{proof} 

\begin{cor}\label{cor:LU-products} 
Retain the hypotheses of Proposition~\ref{prop:LU-Res} and suppose that $H_S$ and
$H_T$ are compact, open, and torsion-free. Then $\LL_U\colon \D(S)\to \D(T)$
preserves small products and, in particular, is also a right adjoint.
\end{cor} 
\begin{proof} 
Note that Lemma~\ref{lem:monoid-coset} implies that $U\subseteq H_S$ is compact
and torsion-free.
Let $I$ be a set, and take $X_i\in \D(S)$, for $i\in I$. We show that the
natural map
\[
\alpha\colon \LL_U \prod_{i\in I}X_i \to \prod_{i\in I} \LL_U(X_i)
\]
is an isomorphism in $\D(T)$. By Proposition~\ref{prop:LU-Res} there is a
natural isomorphism
\[
\Res^T_{H_T} \LL_U
\xRightarrow{\cong} 
\LL_U \Res^S_{H_S} 
.\]
Now observe that the diagram
\[
\begin{tikzcd}[row sep=3em]
\Res^T_{H_T} \LL_U \prod_{i\in I}X_i \ar[r,"\cong"]
\ar[d,"\Res^T_{H_T}(\alpha)"'] &
\LL_U \prod_{i\in I}\Res^S_{H_S}X_i \ar[d,"\alpha'"] \\
\Res^T_{H_T} \prod_{i\in I} \LL_UX_i \ar[r,"\cong"'] &
\prod_{i\in I} \LL_U\Res^S_{H_S}X_i
\end{tikzcd}
\]
is commutative, where $\alpha'$ is the canonical map. 
Since $\Res^S_{H_S}$ and $\Res^T_{H_T}$ are conservative and preserve small
products, by Lemma~\ref{lem:Res-conservative}, it suffices to verify that the
natural map
\[
\alpha' \colon \LL_U \prod_{i\in I} \Res^S_{H_S}(X_i) \to
\prod_{i\in I} \LL_U\bigl(\Res^S_{H_S}X_i\bigr)
\]
is an isomorphism in $\D(H_T)$. But this follows from the fact that $\LL_U\colon
\D(H_S)\to \D(H_T)$ admits a left adjoint (Corollary~\ref{cor:tower}), which
finishes the proof.
\end{proof} 
\subsection{The case of \texorpdfstring{$p$}{p}-adic Lie groups} 
\label{subsec:p-adic}
We fix a continuous surjection $f\colon P\longtwoheadrightarrow M$ of $p$-adic
Lie groups with kernel $U$. 

The crucial result we need is the following:

\begin{lem}\label{lem:Res-hom} 
Let $G$ be a locally profinite group and $K\subseteq G$ an open subgroup. The
natural map
\begin{equation}\label{eq:Res-hom}
\Res^G_K\bigl(\hom_{\D(G)}(X,Y)\bigr) \xrightarrow{\cong}
\hom_{\D(K)}\bigl(\Res^G_KX, \Res^G_KY\bigr)
\end{equation}
is an isomorphism, for all $X,Y\in \D(G)$.
\end{lem} 
\begin{proof} 
The map \eqref{eq:Res-hom} is an isomorphism on cohomology by
\eqref{eq:H(RHom)-group}, hence an isomorphism.
\end{proof} 

\begin{prop}\label{prop:Inf-hom} 
For all $X,Y\in \D(M)$ the natural map
\[
\alpha_{X,Y}\colon \Inf^M_P \hom_{\D(M)}(X,Y) \xrightarrow{\cong} \hom_{\D(P)}
\bigl(\Inf^M_PX, \Inf^M_PY\bigr)
\]
is an isomorphism in $\D(P)$.
\end{prop} 
\begin{proof} 
The natural transformation $\alpha_{X,\pholder}\colon \Inf^M_P\circ \hom_{\D(M)}(X,\pholder)
\To \hom_{D(P)}(\Inf^M_PX,\pholder)\circ \Inf^M_P$ arises as the mate of the canonical
isomorphism
\[
(\Inf^M_PX\otimes_k\pholder)\circ \Inf^M_P \xRightarrow{\cong} \Inf^M_P\circ
(X\otimes_k\pholder),
\]
see Proposition~\ref{prop:mates}.

Let $K_P\subseteq P$ be a compact open subgroup such that $K_M \coloneqq
f(K_P)$ is torsion-free. By Lemma~\ref{lem:Res-conservative} the functor
$\Res^P_{K_P}$ is conservative. 
Therefore, it suffices to show that $\Res^P_{K_P}(\alpha_{X,Y}) =
\alpha_{\Res^M_{K_M}\!X, \Res^M_{K_M}\!\!Y}$ is an isomorphism in $\D(K_P)$. 

Thus, we may assume from the beginning that $P$ is compact and $M$ is
torsion-free and compact. Let
$\cat T$ be the full subcategory of $\D(M)$ with objects those $X\in \D(M)$
for which $\alpha_{X,Y}$ is an isomorphism for all $Y\in \D(M)$. It is clear
that $\cat T$ is triangulated. Since $\Inf^M_P$ preserves small products, by
Theorem~\ref{thm:Inf-prod-general}, and small coproducts, it easily
follows that $\cat T$ is localizing, \ie, closed under small coproducts. Now,
$\hom_{\D(M)}(\one, Y) \cong Y$ and $\hom_{\D(P)}(\Inf^M_P\one, \Inf^M_PY) \cong
\Inf^M_PY$. Under these identifications we have $\alpha_{\one,Y} =
\id_{\Inf^M_PY}$, which shows $\one \in \cat T$. By
Proposition~\ref{prop:compactlygenerated}, $\D(M)$ is compactly generated by
$\one$. This implies $\cat T = \D(M)$, which proves the assertion.
\end{proof} 

\begin{warning} 
Despite the result in Proposition~\ref{prop:Inf-hom}, the functor $\Inf^M_P$ is
in general \emph{not} fully faithful. For example, the counit $\LL_U(\one) =
\LL_U(\Inf^M_P\one) \to \one$ is not an isomorphism by
Corollary~\ref{cor:Grduality} provided $P$ and $M$ are compact and torsion-free.
\end{warning} 

\begin{cor}[Projection formula]\label{cor:projectionformula} 
For all $X\in \D(P)$ and $Y\in \D(M)$ the natural map
\[
\LL_U\bigl(\Inf^M_PX\otimes_kY\bigr) \xrightarrow{\cong} X\otimes_k\LL_U(Y)
\]
is an isomorphism in $\D(M)$. In particular, $\LL_U\Inf^M_PY \cong
\LL_U(\one)\otimes_k Y$ for all $Y\in \D(M)$.
\end{cor} 
\begin{proof} 
By Proposition~\ref{prop:Inf-hom} the natural map
\[
\Inf^M_P\circ \hom_{\D(M)}(Y,\pholder) \xRightarrow{\cong} \hom_{\D(P)}(\Inf^M_P
Y,\pholder)\circ \Inf^M_P
\]
is an isomorphism of functors $\D(M)\to \D(P)$. Passing to the left adjoints,
see Example~\ref{ex:mates}, shows that the induced map
\[
\LL_U\circ (\Inf^M_P Y\otimes_k\pholder) \xRightarrow{\cong} (Y\otimes_k\pholder)\circ \LL_U
\]
is an isomorphism of functors $\D(P)\to \D(M)$.
\end{proof} 

\begin{prop}\label{prop:LU-dual} 
Given $X\in \D(P)$ and $Y\in \D(M)$, the natural map
\[
\RR\H^0\bigl(U, \hom_{\D(P)}(X, \Inf^M_PY)\bigr)
\xrightarrow{\cong} \hom_{\D(M)}(\LL_UX,Y) 
\]
is an isomorphism in $\D(M)$.
\end{prop} 
\begin{proof} 
The natural map
\[
\RR\H^0(U,\pholder)\circ \hom_{\D(P)}(X,\pholder)\circ \Inf^M_P \To
\hom_{\D(M)}\bigl(\LL_U(X),\pholder\bigr)
\]
arises from $\pholder\otimes_k\LL_U(X) \xRightarrow{\cong} \LL_U\circ (\pholder\otimes_kX)
\circ \Inf^M_P$ by passing to the left adjoints, see Example~\ref{ex:mates}. The
latter map is an isomorphism by Corollary~\ref{cor:projectionformula}, hence so
is the first.
\end{proof} 

Recall that we denote by $X^\vee = \hom(X,\one)$ the dual object, see
Notation~\ref{nota:dual}.

\begin{cor}\label{cor:LU-dual} 
There is a natural isomorphism $\RR\H^0(U,X^\vee) \cong (\LL_UX)^\vee$ for all
$X\in \D(P)$.
\end{cor} 
\begin{proof} 
Apply Proposition~\ref{prop:LU-dual} with $Y=\one$.
\end{proof} 
\subsection{Computing the left adjoint using positive 
monoids}\label{subsec:positive}
We suppose now that $P = U\rtimes M$ is a semidirect product of $p$-adic Lie
groups. Let $f\colon P\longtwoheadrightarrow M$ be the canonical projection. We
fix a compact open torsion-free subgroup $K_P$ of $P$. As before, we write $K_U
= K_P\cap U$ and $K_M = f(K_P)$. We make the following hypothesis:
\begin{hyp}\label{hyp:positive} 
There exists a \emph{strictly positive} element $z\in M$, \ie, an element $z$ in
the center of $M$ such that $zK_Uz^{-1} \subseteq K_U$ and
\[
\bigcup_{n>0} z^{-n} K_Uz^n = U.
\]
\end{hyp} 

We fix a strictly positive element $z\in M$. 
Note that the existence of a strictly positive element necessitates that $U$ be
the union of its compact open subgroups.

\begin{ex} 
The example one should have in mind is the following:
Let $\field/\Q_p$ be a finite extension and $\alg G$ a connected reductive group
defined over $\field$. Fix a parabolic subgroup $\alg P$ of $\alg G$ with
unipotent radical $\alg U$. Let $\ol{\alg P}$ be the opposite parabolic subgroup
with respect to $\alg P$, that is, the unique parabolic subgroup $\ol{\alg P}$
of $\alg G$ such that $\alg M\coloneqq \alg P\cap \ol{\alg P}$ is a Levi
subgroup for $\alg P$ and $\ol{\alg P}$. Denote $\ol{\alg U}$ the unipotent
radical of $\ol{\alg P}$. Let $K\subseteq \alg G(\field)$ be a
compact open subgroup with an Iwahori decomposition
\[
K = (K\cap \alg U(\field))(K\cap \alg M(\field))(K\cap \ol{\alg U}(\field)).
\]
By \cite[(6.14)]{Bushnell-Kutzko.1998} there
exists a strongly positive element for the choice $K_P \coloneqq (K\cap \alg
U(\field))(K\cap \alg M(\field))$. Note that the strongly positive elements in
the sense of \opcit\ are strictly positive in our sense. 
\end{ex} 

The following definition is standard, cf.~\cite[(6.5)]{Bushnell-Kutzko.1998},
\cite[II.4]{Vigneras.1998}, or \cite[\S3.1]{EmertonI}. 
\begin{defn} 
An element $m\in M$ is called \emph{positive} if $mK_Um^{-1}\subseteq K_U$. We
denote $M^+$ the monoid of positive elements in $M$. Note that $M^+$ is a
$p$-adic monoid, since $K_P\cap M\subseteq M^+$.

The $p$-adic monoid $P^+ \coloneqq K_UM^+$ satisfies $K_U\backslash P^+\cong
M^+$.
\end{defn} 

\begin{rmk} 
We allow for $U$ to be compact. In this case, Hypothesis~\ref{hyp:positive}
amounts to saying that $U=K_U$ is torsion-free and $M = M^+$. 
\end{rmk} 

Our goal in this section is to give a more precise description of the functor
$\LL_U$ using positive monoids. We first collect some elementary facts.

\begin{lem}\label{lem:M+-flat} 
The following assertions hold:
\begin{enumerate}[label=(\alph*)]
\item\label{lem:M+-flat-a} $M$ (resp.\ $P$) is generated as a monoid by $M^+$ and
$z^{-1}$ (resp.\ $P^+$ and $z^{-1}$).

\item\label{lem:M+-flat-b} Let $S = \{z^n\}_{n\ge0}$ be the multiplicative set
generated by $z$. Then $k[M] = S^{-1}k[M^+]$ is the left Ore localization of
$k[M^+]$ at $S$. In particular, the ring extension $k[M^+]\subseteq k[M]$ is flat. 

Similarly, $k[P] = S^{-1}k[P^+]$ is the left Ore localization of $k[P^+]$ at $S$,
and hence the ring extension $k[P^+]\subseteq k[P]$ is flat. 

\item\label{lem:M+-flat-c} The functors $\ind_{M^+}^M\colon \Rep_k(M^+)\to
\Rep_k(M)$ and $\ind_{P^+}^P\colon \Rep_k(P^+)\to \Rep_k(P)$ are exact.
\end{enumerate}
\end{lem} 
\begin{proof} 
\begin{itemize}
\item[\ref{lem:M+-flat-a}] Let $m \in M$. Since $mK_Um^{-1}$ is a compact
subgroup of $U = \bigcup_{n>0} z^{-n}K_Uz^n$, there exists $n$ such that
$mK_Um^{-1} \subseteq z^{-n}K_Uz^n$. Then $z^nm\in M^+$, proving the
assertion for $M$. Similarly, let $u\in U$ and $m\in M$ be arbitrary. Let
$n\gg0$ such that $u\in z^{-n}K_Uz^n$ and $z^nm\in M^+$. Then $z^num =
z^nuz^{-n}\cdot z^nm \in P^+$, proving the assertion for $P$.

\item[\ref{lem:M+-flat-b}] We prove the assertion for $k[P^+]$; the case of
$k[M^+]$ is similar but easier. Note that $k[P^+]$ embeds into the bigger
ring $k[P]$ in which the elements of $S$ are invertible. Moreover, $S$ satisfies
the left Ore condition, meaning that for all $s\in S$ and
$r\in k[P^+]$ there exist $s'\in S$ and $r'\in k[P^+]$ such that $s'r = r's$;
indeed this is clear for $s'\coloneqq s$ and $r' \coloneqq srs^{-1} \in k[P^+]$.
Now, \cite[Lem.~2.1.13.(ii)]{McConnell-Robson.2001} shows that $S$ is a left
denominator set in $k[P^+]$. 
Then \ref{lem:M+-flat-a} implies that $k[P]$ is
the left Ore localization of $k[P^+]$ with respect to $S$. The
flatness assertion follows from \opcit\ Proposition~2.1.16.

\item[\ref{lem:M+-flat-c}] This follows immediately from \ref{lem:M+-flat-b}.
\end{itemize}
\end{proof} 

Recall that $\Res^P_{P^+}\colon \D(P)\to \D(P^+)$ admits a left adjoint
$\ind_{P^+}^P$ (by Lemma~\ref{lem:M+-flat}.\ref{lem:M+-flat-c}) and a right
adjoint $\RInd_{P^+}^P$.

\begin{lem}\label{lem:P+-unit} 
The following equivalent assertions hold:
\begin{enumerate}[label=(\roman*)]
\item\label{lem:P+-unit-i} The unit $\eta\colon \id_{\D(P)} \xRightarrow{\cong}
\RInd_{P^+}^P\Res^P_{P^+}$ is an isomorphism.

\item\label{lem:P+-unit-ii} $\Res^P_{P^+}\colon\D(P)\to \D(P^+)$ is fully
faithful. 

\item\label{lem:P+-unit-iii} The counit $\varepsilon\colon
\ind_{P^+}^P\Res^P_{P^+} \xRightarrow{\cong} \id_{\D(P)}$ is an isomorphism.
\end{enumerate}
\end{lem} 
\begin{proof} 
The fact that \ref{lem:P+-unit-i}, \ref{lem:P+-unit-ii}, and
\ref{lem:P+-unit-iii} are equivalent follows from
Lemma~\ref{lem:adjoint-ffaithful}. Since both $\ind_{P^+}^P$ and
$\Res^P_{P^+}$ are exact, it suffices for \ref{lem:P+-unit-iii} to show that the
counit
\[
\varepsilon\colon \ind_{P^+}^P\Res^P_{P^+} \To \id_{\Rep_k(P)}
\]
of the underived adjunction is an isomorphism. But this follows from the fact
that $\ind_{P^+}^P = k[P]\otimes_{k[P^+]}\pholder$ is a localization functor by
Lemma~\ref{lem:M+-flat}.\ref{lem:M+-flat-b}.
\end{proof} 

By Theorem~\ref{thm:Inf-prod-general} the functor $\Inf^{M^+}_{P^+}\colon
\D(M^+)\to \D(P^+)$ admits a left adjoint $\LL_{K_U}$.
\begin{prop}\label{prop:LU-positive} 
There is a natural isomorphism
\[
\ind_{M^+}^M \LL_{K_U} \Res^P_{P^+} \xRightarrow{\cong} \LL_U
\]
of functors $\D(P)\to \D(M)$.
\end{prop} 
\begin{proof} 
The functor $\ind_{M^+}^M\colon \D(M^+)\to \D(M)$ exists by
Lemma~\ref{lem:M+-flat}.\ref{lem:M+-flat-c} and is left adjoint to
$\Res^M_{M^+}$. From Lemma~\ref{lem:P+-unit} it follows that the composite
\[
\Inf^M_P \xRightarrow[\eta]{\cong} \RInd_{P^+}^P \Res^P_{P^+} \Inf^M_P
\xRightarrow{\cong} \RInd_{P^+}^P \Inf^{M^+}_{P^+} \Res^M_{M^+}
\]
is a natural isomorphism. The assertion follows by passing to the left adjoints,
cf.~Example~\ref{ex:mates}.
\end{proof} 

\begin{cor}\label{cor:LU-cohdim} 
The functors $\LL^{-n}_U\colon \Rep_k(P)\to \Rep_k(M)$ vanish for
$n\notin\{0,1,\dotsc,\dim U\}$. In particular, $\LL_U$ preserves bounded
complexes.
\end{cor} 
\begin{proof} 
Let $n\in \Z\setminus\{0,1,\dotsc,\dim U\}$. Since $\dim K_U = \dim U$, 
Corollary~\ref{cor:Inf-product} and Proposition~\ref{prop:dualizing} imply 
\[
\LL_{K_U}^{-n} = \H^{\dim K_U-n}(K_U,\pholder) = 0.
\]
Now, Proposition~\ref{prop:LU-positive} shows $\LL^{-n}_U =
\ind_{M^+}^M\LL^{-n}_{K_U} \Res^P_{P^+} = 0$.
\end{proof} 

\begin{cor}\label{cor:LU-bounded} 
The functors $\begin{tikzcd}[cramped] \LL_U\colon \D^{\bounded}(P) \ar[r,shift
left] & \ar[l,shift left] \D^{\bounded}(M):\!\Inf^M_P\end{tikzcd}$ define an
adjunction on the derived categories of bounded complexes.
\end{cor} 
\begin{proof} 
This is immediate from Corollary~\ref{cor:LU-cohdim}.
\end{proof} 

We finish this section by giving another description of $\LL_{K_U}\colon
\D(P^+)\to \D(M^+)$ similar to Corollary~\ref{cor:Inf-product}. To start, let
$V\in \Rep_k(P^+)$ be a smooth $P^+$-representation. There is an action of $M^+$
on $\H^0(K_U,V)$ called the \emph{Hecke action}, cf.
\cite[3.1.3~Def.]{EmertonI}. Concretely, $m\in M^+$ acts on $\H^0(K_U,V)$ as the
composition
\[
\begin{tikzcd}[row sep=0em]
\H^0(K_U,V) \ar[r,"\conj_m"] &
\H^0(mK_Um^{-1}, V) \ar[r,"\cores"] &
\H^0(K_U,V),\\
v \ar[r,mapsto] & mv \ar[r,mapsto] & m\star v\coloneqq \sum_{u\in
K_U/mK_Um^{-1}} umv;
\end{tikzcd}
\]
note that the restriction of the Hecke $M^+$-action to $K_M$ gives the usual
action. 
We obtain a left exact functor $\H^0(K_U,\pholder)\colon \Rep_k(P^+)\to \Rep_k(M^+)$,
hence also a right derived functor
\begin{equation}\label{eq:RH-monoid}
\RR\H^0(K_U,\pholder)\colon \D(P^+)\to \D(M^+).
\end{equation}

\begin{warning} 
The functor \eqref{eq:RH-monoid} is \emph{not} the right adjoint of
$\Inf^{M^+}_{P^+}$, see also Remark~\ref{rmk:adjoints-monoid}. 
\end{warning} 

The following lemma resolves the ambiguity in the notation of $\RR\H^0(K_U,\pholder)$:

\begin{lem}\label{lem:RH-Res} 
The diagram
\[
\begin{tikzcd}[column sep=5em, row sep=3em]
\D(P^+) \ar[r,"{\RR\H^0(K_U,\pholder)}"] \ar[d,"\Res^{P^+}_{K_P}"'] &
\D(M^+) \ar[d,"\Res^{M^+}_{K_M}"]\\
\D(K_P) \ar[r,"{\RR\H^0(K_U,\pholder)}"'] \ar[d,"\Res^{K_P}_{K_U}"'] & \D(K_M)
\ar[d,"\Res^{K_M}_1"]\\
\D(K_U)\ar[r,"{\RR\H^0(K_U,\pholder)}"'] & \D(k)
\end{tikzcd}
\]
is commutative.
\end{lem} 
\begin{proof} 
Observe that $\Res^{P^+}_{K_P}$ admits an exact left adjoint, by
Lemma~\ref{lem:M+-flat}.\ref{lem:M+-flat-c} and hence preserves K-injective
complexes, see Lemma~\ref{lem:K-injectives}.\ref{lem:K-injectives-a}. Thus, the
commutativity of the top square can be checked on the underived level, where it
is clear. 

Since $\H^0(K_U,\pholder)\colon \Rep_k(K_P)\to \Rep_k(K_M)$ and
$\H^0(K_U,\pholder)\colon \Rep_k(K_U)\to \Vect_k$ have finite cohomological
dimension, \cite[Cor.~5.3.$\gamma$]{Hartshorne.1966} shows that both derived
functors can be computed using acyclic resolutions. As $\Res^{K_P}_{K_U}$
preserves $\H^0(K_U,\pholder)$-acyclic objects, see, \eg, \cite[(1.3.6)
Prop.~(ii)]{NSW}, it follows that also the bottom square commutes.
\end{proof} 

\begin{lem}\label{lem:F+} 
The functor~\eqref{eq:RH-monoid} admits a right adjoint $F^{M^+}_{P^+}\colon
\D(M^+)\to \D(P^+)$.
\end{lem} 
\begin{proof} 
By Brown representability, Corollary~\ref{cor:brown}, it suffices to show that
$\RR\H^0(K_U,\pholder)$ commutes with small direct sums. So let $\{Y_i\}_{i\in I}$ in
$\D(P^+)$ be a family. The inclusions $Y_j\to \bigoplus_{i\in I}Y_i$ induce a
canonical map
\[
\alpha\colon \bigoplus_{i\in I} \RR\H^0(K_U,Y_i) \to \RR\H^0\Bigl(K_U,
\bigoplus_{i\in I}Y_i\Bigr).
\]
By Lemma~\ref{lem:RH-Res}, and because $\Res^{M^+}_{K_M}$ and $\Res^{P^+}_{K_P}$
commute with small direct sums, we obtain a commutative diagram
\[
\begin{tikzcd}[column sep=4em]
\Res^{M^+}_{K_M} \bigoplus_{i\in I} \RR\H^0(K_U,Y_i)
\ar[r,"\Res^{M^+}_{K_M}\alpha"] \ar[d,"\cong"']
&
\Res^{M^+}_{K_M}\RR\H^0\bigl(K_U, \bigoplus_{i\in I}Y_i\bigr)
\ar[d,"\cong"]\\
\bigoplus_{i\in I} \RR\H^0\bigl(K_U, \Res^{P^+}_{K_P}Y_i\bigr)
\ar[r,"\cong"']
&
\RR\H^0\bigl(K_U, \bigoplus_{i\in I} \Res^{P^+}_{K_P}Y_i\bigr).
\end{tikzcd}
\]
The lower horizontal map is an isomorphism, since $\RR\H^0(K_U,\pholder)\colon
\D(K_P)\to \D(K_M)$ admits a right adjoint by Lemma~\ref{lem:F^M_P}. Therefore,
$\Res^{M^+}_{K_M}\alpha$ is an isomorphism. As $\Res^{M^+}_{K_M}$ is
conservative by Lemma~\ref{lem:Res-conservative}, it follows that $\alpha$ is an
isomorphism.
\end{proof} 

\begin{lem}\label{lem:F+-Res} 
The diagram
\[
\begin{tikzcd}
\D(M^+) \ar[r,"F^{M^+}_{P^+}"] \ar[d,"\Res^{M^+}_{K_M}"']
&
\D(P^+) \ar[d,"\Res^{P^+}_{K_P}"]
\\
\D(K_M) \ar[r,"F^{K_M}_{K_P}"'] & \D(K_P)
\end{tikzcd}
\]
is commutative.
\end{lem} 
\begin{proof} 
Lemma~\ref{lem:RH-Res} gives a natural isomorphism $\RR\H^0(K_U,\pholder)
\Res^{P^+}_{K_P} \xRightarrow{\cong} \Res^{M^+}_{K_M} \RR\H^0(K_U,\pholder)$. Passing
to the left mate, see Proposition~\ref{prop:mates}, we obtain a natural
transformation
\begin{equation}\label{eq:ind-RH-commute}
\alpha\colon \ind_{K_M}^{M^+} \RR\H^0(K_U,\pholder) \To \RR\H^0(K_U,\pholder)
\ind_{K_P}^{P^+}.
\end{equation}
We claim that $\alpha$ is an isomorphism. Since $\Res^{M^+}_{1}$ is
conservative, it suffices to check that 
\begin{equation}\label{eq:F+-Res-1}
\Res^{M^+}_1 \ind_{K_M}^{M^+}\RR\H^0(K_U,\pholder) \xRightarrow{\Res^{M^+}_{1}\alpha}
\Res^{M^+}_1\RR\H^0(K_U,\pholder) \ind_{K_P}^{P^+} \xRightarrow{\cong} \RR\H^0(K_U,\pholder)
\Res^{P^+}_{K_U}\ind_{K_P}^{P^+}
\end{equation}
is an isomorphism. 
Since $P^+ = K_UM^+$, the inclusion $M^+\subseteq P^+$ induces a
bijection $M^+/K_M \cong K_U\backslash P^+/K_P$. Let $\Gamma
\subseteq M^+$ be a complete representing system for $K_U\backslash P^+/K_P$. By
the Mackey decomposition~\ref{lem:derived-Mackey}, we have an isomorphism
\begin{equation}\label{eq:P+-Mackey}
\bigoplus_{m\in \Gamma} \ind_{mK_Um^{-1}}^{K_U}
\Res^{mK_Pm^{-1}}_{mK_Um^{-1}} m_* \xRightarrow{\cong}
\Res^{P^+}_{K_U}\ind_{K_P}^{P^+}
\end{equation}
of functors $\D(K_P)\to \D(K_U)$, where we have used $K_U\cap mK_Pm^{-1} =
mK_Um^{-1}$. Similarly, we have $\bigoplus_{m\in \Gamma}
\Res^{mK_Mm^{-1}}_1 m_* \xRightarrow{\cong} \Res^{M^+}_1\ind_{K_M}^{M^+}$. 
Since $\RR\H^0(K_U,\pholder)$ has finite cohomological dimension, it can be computed by
$\H^0(K_U,\pholder)$-acyclic complexes, see \cite[Cor.~5.3 $\gamma$]{Hartshorne.1966}.
Now, each of the functors $m_*$, $\Res^{mK_Pm^{-1}}_{mK_Um^{-1}}$, and
$\ind_{mK_Um^{-1}}^{K_U}$ preserves injective objects, since each admits an
exact left adjoint. As every $\H^i(K_U,\pholder)$ commutes with small direct sums by
Lemma~\ref{lem:F+}, it follows from \eqref{eq:P+-Mackey} that
$\Res^{P^+}_{K_U}\ind_{K_P}^{P^+}$ sends injective objects to
$\H^0(K_U,\pholder)$-acyclic objects. It follows that \eqref{eq:F+-Res-1} arises from
the natural map
\begin{align*}
\beta\colon \ind_{K_M}^{M^+} \H^0(K_U,V) &\to
\H^0\bigl(K_U,\ind_{K_P}^{P^+}V\bigr),\\
[m, v] &\longmapsto m\star [1,v] = \sum_{u\in K_U/mK_Um^{-1}} [um,v]
\end{align*}
by passing to the right derived functors. Thus, it suffices to show that $\beta$
is an isomorphism. Using the Mackey decomposition, we have a commutative diagram
\begin{equation}\label{eq:F+-Res-2}
\begin{tikzcd}[column sep=4em, row sep=3em]
\ind_{K_M}^{M^+}\H^0(K_U,V)
\ar[r,"\beta"]
&
\H^0\bigl(K_U,\ind_{K_P}^{P^+}V\bigr)
\\
\bigoplus_{m\in\Gamma} \H^0(K_U,V)
\ar[u,"\cong"] \ar[r,"\oplus_m\beta_m"']
&
\bigoplus_{m\in\Gamma} \H^0\bigl(K_U, \ind_{mK_Um^{-1}}^{K_U} m_*V\bigr)
\ar[u,"\cong"']
\end{tikzcd}
\end{equation}
of $k$-linear maps, where each $\beta_m\colon \H^0(K_U,V)\to \H^0(K_U,
\ind_{mK_Um^{-1}}^{K_U}m_*V)$ is defined by
\[
\beta_m(v) \coloneqq \sum_{u\in K_U/mK_Um^{-1}}[u,v].
\]
Now, it is easy to check that the map induced by $f\mapsto f(1)$ is the inverse
of $\beta_m$, hence $\beta_m$ is an isomorphism. We deduce from
\eqref{eq:F+-Res-2} that $\beta$ and then that $\alpha$ is an isomorphism.
Finally, passing to the right adjoints in \eqref{eq:ind-RH-commute}, we obtain 
the natural isomorphism $\Res^{P^+}_{K_P}F^{M^+}_{P^+}\xRightarrow{\cong}
F^{K_M}_{K_P}\Res^{M^+}_{K_M}$.
\end{proof} 

\begin{prop}\label{prop:F+-Wirthmueller} 
There is a natural isomorphism
\[
F^{M^+}_{P^+}(X)\otimes_k \Inf^{M^+}_{P^+}(Y) \xRightarrow{\cong}
F^{M^+}_{P^+}(X\otimes Y).
\]
In particular, $F^{M^+}_{P^+}\cong \omega_{P^+}\otimes_k \Inf^{M^+}_{P^+}$,
where $\omega_{P^+}\coloneqq F^{M^+}_{P^+}(\one)$ is a smooth character of
$P^+$ sitting in cohomological degree $\dim K_U$.
\end{prop} 
\begin{proof} 
Note that, for $V\in \Rep_k(P^+)$ and $W\in \Rep_k(M^+)$ the natural map
\begin{align}\label{eq:P+-pf}
\H^0(K_U,V) \otimes_k W &\to \H^0\bigl(K_U, V\otimes_k
\Inf^{M^+}_{P^+}W\bigr),\\
v\otimes w &\longmapsto v\otimes w \notag
\end{align}
is an isomorphism in $\Rep_k(M^+)$. Indeed, bijectivity is clear, and for every
$v\in \H^0(K_U,V)$, $w\in W$, and $m\in M^+$ we have
$m\star v\otimes mw = \sum_{u\in K_U/mK_Um^{-1}}
umv\otimes mw = \sum_{u\in K_U/mK_Um^{-1}} um (v\otimes w) = m\star (v\otimes
w)$, which shows that \eqref{eq:P+-pf} is $M^+$-equivariant. 

The map \eqref{eq:P+-pf} extends to a natural isomorphism
\begin{equation}\label{eq:P+-pf-1}
\H^0(K_U,X)\otimes_k Y \xrightarrow{\cong} \H^0\bigl(K_U, X\otimes_k
\Inf^{M^+}_{P^+}Y\bigr)
\end{equation}
of bifunctors $\K(P^+)\times \K(M^+)\to \K(M^+)$. For any $Y\in \D(M^+)$ we
obtain a natural transformation
\begin{align*}
\RR\H^0(K_U,\pholder)\otimes_k Y &\xRightarrow{\cong} \RR\bigl( \H^0(K_U,\pholder)
(\pholder\otimes_k \Inf^{M^+}_{P^+}Y)\bigr)\\
&\xRightarrow{\;\alpha} \RR\H^0(K_U,\pholder)\circ (\pholder\otimes_k \Inf^{M^+}_{P^+}Y).
\end{align*}
We claim that the map denoted $\alpha$ is an isomorphism. Recall (\eg, from the
proof of Lemma~\ref{lem:RH-Res}) that
$\RR\H^0(K_U,\pholder)$ can be computed using $\H^0(K_U,\pholder)$-acyclic resolutions. Note
that, if $V\in \Rep_k(P^+)$ is $\H^0(K_U,\pholder)$-acyclic and $W\in \Rep_k(M^+)$ is
arbitrary, then $V\otimes_k \Inf^{M^+}_{P^+}W$ is $\H^0(K_U,\pholder)$-acyclic: indeed,
by Lemma~\ref{lem:RH-Res} this can be checked after applying $\Res^{P^+}_{K_U}$,
in which case we may identify $\Res^{P^+}_{K_U}(V\otimes_k \Inf^{M^+}_{P^+}W)
\cong V^{\oplus I}$, where $I$ parametrizes a $k$-basis of $W$, and then the
claim follows from the fact that each $\H^i(K_U,\pholder)$ commutes with small direct
sums (Lemma~\ref{lem:F+}). It follows that if $X\in \K(P^+)$ is a complex
consisting of $\H^0(K_U,\pholder)$-acyclic objects and $Y\in \K(M^+)$ is arbitrary,
then $X\otimes_k \Inf^{M^+}_{P^+}Y$ consists of $\H^0(K_U,\pholder)$-acyclic objects.
We deduce that $\alpha$ is an isomorphism, cf. \cite[Cor.~14.3.5]{KS}.

Thus, we have constructed a natural isomorphism
\[
\RR\H^0(K_U,\pholder)\circ (\pholder\otimes_k \Inf^{M^+}_{P^+} Y) \xLeftarrow{\cong}
(\pholder\otimes_k Y) \circ \RR\H^0(K_U,\pholder).
\]
Passing to the right mate yields a natural map
$\beta\colon F^{M^+}_{P^+}X \otimes_k \Inf^{M^+}_{P^+}Y \Longrightarrow
F^{M^+}_{P^+}(X\otimes Y)$. Now, we have a commutative diagram
\[
\begin{tikzcd}
\Res^{P^+}_{K_P}\bigl(F^{M^+}_{P^+}X\otimes_k \Inf^{M^+}_{P^+}Y\bigr)
\ar[r,"\Res^{P^+}_{K_P}\beta"]
\ar[d,"\cong"']
&
\Res^{P^+}_{K_P}\bigl(F^{M^+}_{P^+}(X\otimes Y)\bigr)
\ar[d,"\cong"]
\\
F^{K_M}_{K_P}\Res^{M^+}_{K_M}X \otimes_k \Inf^{K_M}_{K_P}\Res^{M^+}_{K_M}Y
\ar[r,"\cong"']
&
F^{K_M}_{K_P}\bigl(\Res^{M^+}_{K_M}X\otimes_k \Res^{M^+}_{K_M}Y\bigr),
\end{tikzcd}
\]
where the vertical maps are isomorphisms by Lemma~\ref{lem:F+-Res} and the
bottom horizontal map is the isomorphism from Corollary~\ref{cor:Grduality}.
Hence, the top horizontal map is an isomorphism. As $\Res^{P^+}_{K_P}$ is
conservative by Lemma~\ref{lem:Res-conservative}, we deduce that $\beta$ is an
isomorphism.

Finally, Lemma~\ref{lem:F+-Res} and Proposition~\ref{prop:dualizing}
show $\Res^{P^+}_{K_P}(\omega_{P^+}) = \omega_{K_P} = k[\dim K_U]$, which proves
the last statement.
\end{proof} 

\begin{cor}\label{cor:RH(KU,-)=LKU} 
There is a natural isomorphism $\RR\H^0(K_U,\omega_{P^+}\otimes_k\pholder)
\xRightarrow{\cong} \LL_{K_U}$ of functors $\D(P^+)\to \D(M^+)$.
\end{cor} 
\begin{proof} 
The natural map in the assertion, say, $\alpha$ arises, by passing to the left
mate, from the natural isomorphism $(\omega_{P^+}\otimes_k\pholder) \Inf^{M^+}_{P^+}
\xRightarrow{\cong} F^{M^+}_{P^+}$. We then have a commutative diagram
\[
\begin{tikzcd}[column sep=4em, row sep=3em]
\Res^{M^+}_{K_M} \RR\H^0(K_U,\omega_{P^+}\otimes_k \pholder)
\ar[r,Rightarrow,"\Res^{M^+}_{K_M}\alpha"] \ar[d,Rightarrow,"\cong"']
&
\Res^{M^+}_{K_M} \LL_{K_U}
\ar[d,Rightarrow,"\cong"]
\\
\RR\H^0(K_U,\omega_{K_P}\otimes_k\pholder) \Res^{P^+}_{K_P}
\ar[r,Rightarrow,"\cong"']
&
\LL_{K_U} \Res^{P^+}_{K_P},
\end{tikzcd}
\]
where the left vertical map is an isomorphism by Lemmas~\ref{lem:RH-Res}
and~\ref{lem:F+-Res}, the right vertical map is an isomorphism by
Proposition~\ref{prop:LU-Res}, and the bottom horizontal map is an isomorphism
by Corollary~\ref{cor:Inf-product}. Hence, $\Res^{M^+}_{K_M}\alpha$ is an
isomorphism. As $\Res^{M^+}_{K_M}$ is conservative by
Lemma~\ref{lem:Res-conservative}, it follows that $\alpha$ is an isomorphism.
\end{proof} 

\begin{rmk}\label{rmk:deltaP} 
One can show $\omega_{P^+} = \Res^P_{P^+}\omega_P$, where $\omega_P =
\delta_P[\dim U]$ and $\delta_P$ is a smooth character of $P$ (which is
necessarily trivial on $U$). Hence, $\omega_{P^+}$ is in fact
$\otimes$-invertible. (Note that this fails for general characters of
monoids.)

In fact, $\delta_P$ is the (inflation to $P$ of the) character $\alpha^u$
defined in \cite[3.6]{EmertonII}.
\end{rmk} 

\begin{cor}\label{cor:LU-explicit} 
One has a natural isomorphism
\[
\ind_{M^+}^M \RR\H^0(K_U,\omega_{P^+}\otimes_k\pholder) \Res^P_{P^+}
\xRightarrow{\cong} \LL_U
\]
of functors $\D(P)\to \D(M)$.
\end{cor} 

\begin{cor}\label{cor:LU-trivial} 
Let $f\colon P\longtwoheadrightarrow M$ be a continuous surjection of $p$-adic
Lie groups with kernel $U$, and let $K_U\subseteq U$ be a torsion-free compact
open subgroup. 
Assume there is an automorphism $\varphi\colon U\xrightarrow{\cong} U$ of
topological groups such that $\bigcap_{n\ge0} \varphi^n(K_U) = \{0\}$ and
$\bigcup_{n\ge0}\varphi^{-n}(K_U) = U$. The following equivalent properties
are satisfied:
\begin{enumerate}[label=(\roman*)]
\item\label{cor:LU-trivial-i} $\Inf^M_P\colon \D(M)\to \D(P)$ is fully faithful;

\item\label{cor:LU-trivial-ii} The counit $\varepsilon\colon \LL_U
\Inf^{M}_{P}\xRightarrow{\cong} \id_{\D(M)}$ is an isomorphism;

\item\label{cor:LU-trivial-iii} $\LL_U(\one_P)\cong \one_M$ in $\D(M)$;

\item\label{cor:LU-trivial-iv} $\LL_U(\one_U)\cong \one$ in $\D(k)$.
\end{enumerate}
\end{cor} 
\begin{proof} 
For \ref{cor:LU-trivial-i}$\iff$\ref{cor:LU-trivial-ii}, see
Lemma~\ref{lem:adjoint-ffaithful}, and the equivalence
\ref{cor:LU-trivial-ii}$\iff$\ref{cor:LU-trivial-iii} follows from the
projection formula~\ref{cor:projectionformula}; note that
\ref{cor:LU-trivial-iii} is satisfied if and only if $\varepsilon_{\one_P}$ is
an isomorphism. The equivalence
\ref{cor:LU-trivial-iii}$\iff$\ref{cor:LU-trivial-iv} follows from
Proposition~\ref{prop:LU-Res}.\medskip

The hypothesis on $\varphi$ shows that
there exists $n_0\ge0$ with $\varphi^{n_0}(K_U) \subseteq K_U$. Hence, replacing
$\varphi^{n_0}$ by $\varphi$ if necessary, we may assume that
$\varphi$ is \emph{positive}, meaning that $\varphi(K_U)\subseteq K_U$.
By \ref{cor:LU-trivial-iii}$\iff$\ref{cor:LU-trivial-iv}, we may assume $P =
U\rtimes M$, where $M \cong \Z$ is generated by an
element $z$ acting through conjugation on $U$ via $\varphi$. Let $M^+$ be the
monoid generated by $z$ and put $P^+ = K_UM^+$.\medskip

If $U$ has the discrete topology, the compact group $K_U$ is finite. As
$\varphi$ is positive, we must have $\varphi^n(K_U)= K_U$ for all $n\in\Z$, and
hence the hypothesis shows $K_U = \{1\}$ and then $U = \{1\}$, in which case the
assertion is trivial.\medskip

If $U$ is not discrete, the hypothesis implies that $U$ contains no compact open
subgroup which is normalized by $M$. In particular, $U$ is not compact. By
Corollary~\ref{cor:LU-explicit}, there is a natural isomorphism
\[
\ind_{M^+}^{M} \H^{d+i}(K_U,\delta_{P^+}) \xrightarrow{\cong} \LL_U^i(\one),
\]
where $\omega_{P^+} = \delta_{P^+}[d]$ and $d = \dim U$. Fix $i\neq 0$. As $K_U$
is a Poincar\'e group, the $k$-vector space $\H^{-i}(K_U,\one)$ is
finite-dimensional, and the cup product $\H^{-i}(K_U,\one)\times
\H^{d+i}(K_U,\one)\to \one$ is a perfect pairing. Therefore, $\H^{-i}(K_U,\one)$
identifies with the $k$-linear dual of $\H^{d+i}(K_U,\one)$, and for each $m\in
M^+$ the diagram
\[
\begin{tikzcd}
\H^{d+i}(mK_Um^{-1},\one) \ar[r,"\cong"] \ar[d,"\cores"'] &
\H^{-i}(mK_Um^{-1},\one)^* \ar[d,"\res^*"]\\
\H^{d+i}(K_U,\one) \ar[r,"\cong"'] & \H^{-i}(K_U,\one)^*
\end{tikzcd}
\]
is commutative. As $\H^{-i}(K_U,\one)$ is finite-dimensional and $i\neq 0$, we
find (by the hypothesis) some $n\gg0$ such that for $m\coloneqq z^n$ the right
vertical map vanishes. Hence, also the left vertical map is zero. As the
Hecke-action of $m$ on $\H^{d+i}(K_U,\delta_{P^+})$ is given by the composite
\[
\H^{d+i}(K_U,\delta_{P^+}) \xrightarrow{\delta_{P^+}(m)\cdot\conj_m}
\H^{d+i}(mK_Um^{-1},\delta_{P^+}) \xrightarrow{\cores}
\H^{d+i}(K_U,\delta_{P^+}),
\]
where $\conj_m$ is the conjugation by $m$ on cohomology, we deduce that
$\ind_{M^+}^M \H^{d+i}(K_U,\delta_{P^+}) = \{0\}$. Therefore, $\LL_U^i(\one) =
\{0\}$. Finally, it is clear that $\LL_U^0(\one) = \one$
(cf.~Proposition~\ref{prop:Ln<0}).
\end{proof} 

\begin{ex} 
If $U$ is the group of $\Q_p$-points of a connected unipotent linear algebraic
group, then
the functor $\Inf^M_P\colon \D(M)\to \D(P)$ is fully faithful. Indeed, by
Proposition~\ref{prop:LU-transitive}.\ref{prop:LU-transitive-a} the class of
groups $U'$ satisfying $\LL_{U'}(\one_{U'})\cong \one$ is closed under extensions.
Now, $U$ admits a composition series with components isomorphic to $\Q_p$, so we
may assume $U=\Q_p$. But then the hypothesis of Corollary~\ref{cor:LU-trivial}
is satisfied for $K_U = \Z_p$ and $\varphi$ being multiplication by $p$.
\end{ex}

\section{The left adjoint of derived parabolic induction} 
\label{sec:Jacquet}
As in \S\ref{sec:derivedinf} we fix a field $k$ of characteristic $p>0$.

\subsection{General results}\label{subsec:parabolic-general} 
Let $G$ be a $p$-adic reductive group, \ie, the group of $\field$-points of a
connected reductive group defined over a finite field extension $\field /\Q_p$.
Let $P$ be a parabolic subgroup of $G$ with Levi quotient $M$ and
unipotent radical $U$. 

The functor of \emph{parabolic induction}
\[
i_M^G \coloneqq \Ind_P^G \circ \Inf^M_P\colon \Rep_k(M)\to \Rep_k(G) 
\]
is exact and hence defines a functor $i_M^G\colon \D(M)\to \D(G)$ on the derived
categories. 

\begin{thm}\label{thm:L(U,-)} 
The derived parabolic induction $i_M^G$ admits a left adjoint, denoted
$\LL(U,\pholder)$. In fact, there is a natural isomorphism
\[
\RHom_{\Rep_k(M)}\bigl(\LL(U,X), Y\bigr) \xrightarrow{\cong}
\RHom_{\Rep_k(G)}\bigl(X, i_M^G(Y)\bigr)\qquad \text{in $\D(k)$,}
\]
for all $X\in \D(G)$ and $Y\in \D(M)$.
\end{thm} 
\begin{proof} 
By Lemma~\ref{lem:adjoint-composition} and
Theorem~\ref{thm:Inf-prod-general} it follows that $\LL(U,\pholder) = \LL_U\circ
\Res^G_P$ is the left adjoint of $i_M^G$. The last statement is analogous to
Corollary~\ref{cor:RHom-LU-Inf} (use that $\Ind_P^G$ is exact).
\end{proof} 

\begin{notation}\label{nota:L(U,-)^n} 
For any $n\in\Z$ we denote by $\LL^n(U,\pholder)\colon \Rep_k(G)\to \Rep_k(M)$ the
$n$-th cohomo\-logy functor of $\LL(U,\pholder)$, that is,
\[
\LL^n(U,V) \coloneqq \H^n\bigl(\LL(U,V[0])\bigr), \qquad \text{for $V\in
\Rep_k(G)$.}
\]
\end{notation} 

We record the following consequence of Theorem~\ref{thm:L(U,-)}:

\begin{cor}\label{cor:L(U,-)-spectral} 
Given $V\in \Rep_k(G)$ and $W\in \Rep_k(M)$, there is a convergent
first-quadrant spectral sequence
\[
E_2^{i,j} = \Ext_{\Rep_k(M)}^i\bigl(\LL^{-j}(U,V), W\bigr) \To
\Ext_{\Rep_k(G)}^{i+j}\bigl(V, i_M^GW\bigr).
\]
In particular, there is a five-term exact sequence
\[
\begin{tikzcd}[column sep=1em]
0 \ar[r] 
&
\Ext_{\Rep_k(M)}^1\bigl(\LL^0(U,V), W\bigr) \ar[r]
&
\Ext_{\Rep_k(G)}^1\bigl(V, i_M^GW\bigr) \ar[r]
\ar[d,phantom,""{coordinate,name=Z}]
&
\Hom_{\Rep_k(M)}\bigl(\LL^{-1}(U,V), W\bigr) 
\ar[dll,"{d_2^{0,1}}" description, rounded corners, at end, to path={
-- ([xshift=2ex]\tikztostart.east) |- (Z) \tikztonodes -|
([xshift=-2ex]\tikztotarget.west) -- (\tikztotarget)}]
\\
&
\Ext_{\Rep_k(M)}^2\bigl(\LL^0(U,V), W\bigr) \ar[r]
&
\Ext_{\Rep_k(G)}^2\bigl(V, i_M^GW\bigr).
\end{tikzcd}
\]
\end{cor} 
\begin{proof} 
The functor $\RHom_{\Rep_k(M)}(\pholder,W)$ is left $t$-exact and $\LL(U,V)$ is right
bounded. We obtain the spectral sequence from
Lemma~\ref{lem:spectral}.\ref{lem:spectral-b} combined with
Theorem~\ref{thm:L(U,-)}.
\end{proof} 

\begin{prop}\label{prop:L(U,-)-transitive} 
Let $Q\subseteq P$ be a parabolic subgroup of $G$ with Levi quotient $L$ and
unipotent radical $U'$.

\begin{enumerate}[label=(\alph*)]
\item\label{prop:L(U,-)-transitive-a} The diagram
\[
\begin{tikzcd}[column sep=3em]
\D(G) \ar[r,"{\LL(U,\pholder)}"] \ar[dr,"{\LL(U',\pholder)}"'] & \D(M) \ar[d,"{\LL(U'/U,\pholder)}"]
\\
& \D(L)
\end{tikzcd}
\]
is commutative.

\item\label{prop:L(U,-)-transitive-b} Given $V\in \Rep_k(G)$, there is a
convergent third-quadrant spectral sequence
\[
E_2^{i,j} = \LL^i\bigl(U'/U, \LL^j(U,V)\bigr) \To \LL^{i+j}(U', V).
\]
In particular, there is a five-term exact sequence
\[
\begin{tikzcd}
\LL^{-2}(U',V) \ar[r]
&
\LL^{-2}\bigl(U'/U, \LL^0(U,V)\bigr) \ar[r,"d^{-2,0}"]
\ar[dr,phantom,""{coordinate,name=Z}]
&
\LL^0\bigl(U'/U, \LL^{-1}(U,V)\bigr) 
\ar[dl, rounded corners, at end, to path={
-- ([xshift=2ex]\tikztostart.east) |- ([xshift=2.5ex]Z) \tikztonodes -|
([xshift=-2ex]\tikztotarget.west) -- (\tikztotarget)}]
\\
&
\LL^{-1}(U',V\bigr) \ar[r]
&
\LL^{-1}\bigl(U'/U, \LL^0(U,V)\bigr) \ar[r]
&
0.
\end{tikzcd}
\]
\end{enumerate}
\end{prop} 
\begin{proof} 
\begin{itemize}
\item[\ref{prop:L(U,-)-transitive-a}] Clearly, there is a natural isomorphism
$\LL^0(U'/U,\pholder)\circ \LL^0(U,\pholder) \xRightarrow{\cong} \LL^0(U',\pholder)$ of functors
$\Rep_k(G)\to \Rep_k(L)$. Passing to the right adjoints, see
Example~\ref{ex:mates}, yields a natural isomorphism
$i_L^G\xRightarrow{\cong} i_M^G\circ i_L^M$. But this extends to a natural
isomorphism of functors $\D(L)\to \D(G)$. Hence, passing to the left adjoints
yields a natural isomorphism $\LL(U'/U,\pholder)\circ \LL(U,\pholder)\xRightarrow{\cong}
\LL(U',\pholder)$.

\item[\ref{prop:L(U,-)-transitive-b}] The functors $\LL(U'/U,\pholder)$ and $\LL(U,\pholder)$
are right $t$-exact (as left adjoints of $t$-exact functors). Now,
Lemma~\ref{lem:spectral}.\ref{lem:spectral-a} applied to
\ref{prop:L(U,-)-transitive-a} yields the desired spectral sequence.
\end{itemize}
\end{proof} 

Fix a compact, open, torsion-free subgroup $K$ of $G$. For any closed subgroup
$H$ of $G$ we write $K_H\coloneqq K\cap H$. 
Let $M^+$ and $P^+$ be the positive monoids of \S\ref{subsec:positive}
associated with $K_P$. Then \cite[(6.14)]{Bushnell-Kutzko.1998}
shows that Hypothesis~\ref{hyp:positive} is satisfied.

\begin{prop}\label{prop:L(U,-)-derived} 
There is a natural isomorphism
\[
\ind_{M^+}^M \RR\H^{0}\bigl(K_U,
\omega_{P^+}\otimes_k\pholder\bigr)\Res^G_{P^+}
\xRightarrow{\cong} \LL(U,\pholder)
\]
of functors $\D(G)\to \D(M)$. In particular, $\LL^{-n}(U,\pholder)$ vanishes
provided $n\notin \{0,1,\dotsc,\dim K_U\}$.
\end{prop} 
\begin{proof} 
Immediate from Corollary~\ref{cor:LU-explicit}. 
\end{proof} 

\begin{rmk} 
From Corollary~\ref{cor:LU-explicit} and \cite[Thm.~3.4.7(1) and
Lem.~3.2.1(1)]{EmertonII} it follows that on (locally) admissible
representations the $\delta$-functor $\LL^\bullet(U,\pholder)$ is isomorphic to the
$\delta$-functor $\H_\bullet(U,\pholder) = \H^{\dim U - \bullet}\Ord_P(\pholder)\otimes
\delta_P$ considered in \cite[\S4.2]{Hauseux.2018}.
\end{rmk} 

\begin{cor}\label{cor:L(U,-)-bounded} 
There are adjunctions 
\[
\begin{tikzcd}[row sep=0em,]
\LL(U,\pholder)\colon \D^+(G) \ar[r,shift left] & \ar[l,shift left] \D^+(M) :\!
i_M^G,\\
\LL(U,\pholder)\colon \D^{\bounded}(G) \ar[r,shift left] & \ar[l, shift left]
\D^{\bounded}(M) :\! i_M^G.
\end{tikzcd}
\]
\end{cor} 
\begin{proof} 
This is immediate from Proposition~\ref{prop:L(U,-)-derived}.
\end{proof} 

As an immediate consequence of Corollary~\ref{cor:Hom-LU} we have the following
result:
\begin{prop} 
There is a natural isomorphism
\[
\Hom_{\D(k)}\bigl(\Res^M_1\LL(U,X), Y\bigr) \cong \Hom_{\D(U)}\bigl(\Res^G_UX,
\Inf^1_UY\bigr)
\]
for all $X\in \D(G)$, $Y\in \D(k)$. In particular, for every $n\in\Z$ there
exists a natural $k$-linear isomorphism
\[
\Hom_k\bigl(\LL^{-n}(U,V), W\bigr) \xrightarrow{\cong} \Ext_{\Rep_k(U)}^n(V,
\Inf^1_UW),
\]
for all $V\in \Rep_k(G)$ and $W\in \Vect_k$.
\end{prop} 
\subsection{Preservation of global admissibility} 
\label{subsec:global-adm}
Retain the notation of \S\ref{subsec:parabolic-general}. Let $\ol P$ be the
parabolic
subgroup of $G$ opposite $P$ with unipotent radical $\ol U$ and Levi quotient
$M$. We make the further assumption that the compact, open, torsion-free
subgroup $K\subseteq G$ admits an Iwahori decomposition:
\[
K = K_{\ol U} K_M K_{U}.
\]

Recall from \cite[Def.~4.2]{Schneider-Sorensen.2023a} the following definition:
\begin{defn} 
Let $H$ be a $p$-adic Lie group and $H' \subseteq H$ a compact, open,
torsion-free subgroup. A complex $X\in \D(H)$ is called \emph{globally
admissible} if for every $i\in \Z$ the $k$-vector space $\H^i(H',X)$ is
finite-dimensional. 
By \cite[Cor.~4.6]{Schneider-Sorensen.2023a}, this definition is independent of the
choice of $H'$. 

The strictly full triangulated subcategory of $\D(H)$ consisting of the globally
admissible complexes is denoted $\D(H)^{\gadm}$.
\end{defn} 

In this section we prove the following result:

\begin{thm}\label{thm:globally-admissible} 
There is an adjunction
\[
\begin{tikzcd}
\LL(U,\pholder)\colon \D(G)^{\gadm} \ar[r,shift left] & \ar[l,shift left] \D(M)^{\gadm}
:\! i_M^G.
\end{tikzcd}
\]
\end{thm} 

We will prove separately that $\LL(U,\pholder)$ and $i_M^G$ preserve globally
admissible complexes. Before turning to the proofs, let us observe an immediate
consequence.

Given a $p$-adic Lie group $H$, we denote by $\D^+_{\adm}(H)$, resp.\ 
$\D^{\bounded}_{\adm}(H)$, the strictly full triangulated subcategory of
$\D^+(H)$, resp.\ $\D^{\bounded}(H)$, consisting of those complexes $X$ such that
$\H^i(X)$ is admissible, for every $i\in \Z$.

\begin{cor}\label{lem:L(U,-)-admissible} 
There are adjunctions
\[
\begin{tikzcd}[row sep=0em]
\LL(U,\pholder)\colon \D^+_{\adm}(G) \ar[r,shift left] & \ar[l,shift left]
\D^+_{\adm}(M) :\! i_M^G,\\
\LL(U,\pholder)\colon \D^{\bounded}_{\adm}(G) \ar[r,shift left] & \ar[l,shift left]
\D^{\bounded}_{\adm}(M) :\! i_M^G.
\end{tikzcd}
\]
\end{cor} 
\begin{proof} 
By \cite[Prop.~4.9]{Schneider-Sorensen.2023a} we
have $\D^+_{\adm}(G) = \D^+(G) \cap \D(G)^{\gadm}$ and, \textit{a fortiori},
$\D^{\bounded}_{\adm}(G) = \D^{\bounded}(G)\cap \D(G)^{\gadm}$; similarly with
$G$ replaced by $M$. As $\LL(U,\pholder)$ and $i_M^G$ preserve bounded complexes, the
claim follows from Theorem~\ref{thm:globally-admissible}.
\end{proof} 
\subsubsection{Parabolic induction preserves global 
admissibility}
It is well-known that the parabolic induction functor $i_M^G\colon \Rep_k(M)\to
\Rep_k(G)$ preserves admissibility, \cite[II.2.1]{Vigneras.1996}. In order to
prove that derived parabolic induction $i_M^G = \Ind_P^G \Inf^M_P$ preserves
global admissibility, we show separately that $\Inf^M_P$ and $\Ind_P^G$ preserve
global admissibility.

\begin{lem}\label{lem:Ind-gadm} 
One has $\Ind_P^G\bigl(\D(P)^{\gadm}\bigr) \subseteq \D(G)^{\gadm}$.
\end{lem} 
\begin{proof} 
Let $X\in \D(P)^{\gadm}$ be globally admissible. We need to show that
$\H^i(K, \Ind_P^GX)$ is finite-dimensional, for all $i\in\Z$. Note that, for
every $g\in G$ the subgroup $P\cap gKg^{-1} \subseteq P$ is compact, open, and
torsion-free. Applying the derived Mackey decomposition,
Lemma~\ref{lem:derived-Mackey}.\ref{lem:derived-Mackey-a}, we deduce
\begin{align*}
\RR\H^0(K,\Ind_P^GX) &\cong \RR\H^0\Bigl(K, \prod_{g\in P\backslash G/K}
\Ind_{g^{-1}Pg\cap K}^K g_*^{-1} \Res^P_{P\cap gKg^{-1}} X\Bigr)\\
&\cong \prod_{g\in P\backslash G/K} \RR\H^0\bigl(K, \Ind^K_{g^{-1}Pg\cap K}
g_*^{-1} \Res^P_{P\cap gKg^{-1}} X\bigr)\\
&\cong \prod_{g\in P\backslash G/K} \RR\H^0\bigl(g^{-1}Pg\cap K, g_*^{-1}
\Res^P_{P\cap gKg^{-1}} X\bigr)\\
&\cong \prod_{g\in P\backslash G/K} \RR\H^0\bigl(P\cap gKg^{-1}, \Res^P_{P\cap
gKg^{-1}}X\bigr).
\end{align*}
Note that the product is finite, because $P\backslash G$ is compact and $K$ is
open. The second isomorphism holds, since $\RR\H^0(K,\pholder)$ is an additive functor.
For the third isomorphism, observe that there is a natural isomorphism
$\Inf^1_{g^{-1}Pg\cap K} \xRightarrow{\cong} \Res^K_{g^{-1}Pg\cap K} \Inf^1_K$,
so that passing to the right adjoints (see Example~\ref{ex:mates}) yields an
isomorphism $\RR\H^0(K,\pholder) \Ind_{g^{-1}Pg\cap K}^K \xRightarrow{\cong}
\RR\H^0(g^{-1}Pg\cap K,\pholder)$.  The last isomorphism is given by conjugation. 

As $X$ is globally admissible, the right hand side in the above computation has
finite-dimensional cohomology, hence so does the left hand side. Consequently,
$\Ind_P^GX$ is globally admissible.
\end{proof} 

Before proving that $\Inf^M_P$ preserves global admissibility, we recall a
well-known general fact.

\begin{lem}\label{lem:tensor-ideal} 
Let $\cat C$ be a tensor triangulated category which is compactly generated by
the tensor unit $\one$. Let $\cat T$ be a full triangulated subcategory of $\cat
C$ which is closed under direct summands. One has $T\otimes X \in \cat T$ for
every $T\in \cat T$ and every compact $X \in\cat C$.
\end{lem} 
\begin{proof} 
Consider the full subcategory $\cat S$ of $\cat C$ consisting of
those $X$ such that for all $T\in \cat T$ it holds that $T\otimes X \in \cat T$.
As $\cat T$ is triangulated and closed under direct summands, it is clear that
$\cat S$ is triangulated and closed under direct summands. By
\cite[Lem.~2.2]{Neeman.1992b} the full triangulated subcategory $\langle
\one\rangle$ generated by $\one$ and closed under direct summands contains
precisely the compact objects. Since $\one \in \cat S$, we deduce $\langle
\one\rangle \subseteq \cat S$, that is,
$\cat S$ contains all compact objects. This implies the assertion.
\end{proof} 

\begin{lem}\label{lem:Inf-gadm} 
One has $\Inf^M_P\bigl(\D(M)^{\gadm}\bigr) \subseteq \D(P)^{\gadm}$.
\end{lem} 
\begin{proof} 
Given $X\in \D(M)$, we observe that $\Res^M_{K_M}(X)$ is globally admissible if
and only if $X$ is; similarly with $M$ replaced by $P$. Therefore, it suffices
to prove $\Inf^{K_M}_{K_P}\bigl(\D(K_M)^{\gadm}\bigr) \subseteq
\D(K_P)^{\gadm}$. Let $X\in \D(K_M)^{\gadm}$ be globally admissible. By
Proposition~\ref{prop:RH(U,-)-compact} we know that $\RR\H^0(K_U,\one)$ is
compact. Now, Lemma~\ref{lem:tensor-ideal} implies $X\otimes_k
\RR\H^0(K_U,\one) \in \D(K_M)^{\gadm}$. Applying the projection formula,
Lemma~\ref{lem:RH(U,-)-projection-formula}, we deduce that
\[
\RR\H^0\bigl(K_P, \Inf^{K_M}_{K_P}X\bigr) \cong \RR\H^0\bigl(K_M, \RR\H^0(K_U,
\Inf^{K_M}_{K_P}X)\bigr) \cong \RR\H^0(K_M, X\otimes_k\RR\H^0(K_U,\one)\bigr)
\]
has finite-dimensional cohomology. Therefore, $\Inf^{K_M}_{K_P}X$ is globally
admissible.
\end{proof} 

\begin{prop}\label{prop:i_M^G-gadm} 
Derived parabolic induction restricts to a functor $i_M^G\colon \D(M)^{\gadm}
\to \D(G)^{\gadm}$.
\end{prop} 
\begin{proof} 
Combine Lemmas~\ref{lem:Ind-gadm} and~\ref{lem:Inf-gadm}.
\end{proof} 

\subsubsection{\texorpdfstring{$\LL(U,\pholder)$}{L(U,-)} preserves global 
admissibility}
Fix a positive central element $z_0 \in M$ satisfying the conditions in
\cite[(6.14)]{Bushnell-Kutzko.1998}. In
particular, $z_0$ is strictly positive and satisfies $z_0K_{\ol U}z_0^{-1}
\supsetneq K_{\ol U}$. Let $C$ be the central subgroup of $M$ generated by
$z_0$. Then $C^+ \coloneqq M^+\cap C$ is the submonoid generated by $z_0$. Note
that $K_MC$ is a direct product, since $C\cap K_M = \{1\}$.

\begin{lem}\label{lem:gadm-1} 
The diagram
\[
\begin{tikzcd}[column sep=3em] 
\D(G)
\ar[r,"{\LL(U,\pholder)}"] \ar[d,"\Res^G_{K_PC^+}"']
&
\D(M)
\ar[r,"\Res^M_{K_MC}"]
&
\D(K_MC)
\ar[d,"\LL_{K_M}"]
\\
\D(K_PC^+)
\ar[r,"\LL_{K_P}"']
&
\D(C^+)
\ar[r,"\ind_{C^+}^C"']
&
\D(C)
\end{tikzcd} 
\]
is commutative. In particular, there is a natural $k$-linear isomorphism 
\[
\H^{i}\bigl(K_M,\LL(U,X)\bigr) \cong \ind_{C^+}^C \H^{i+\dim K_U}(K_P,X)
\]
for every $X\in \D(G)$.
\end{lem} 
\begin{proof} 
We compute
\begin{align*}
\LL_{K_M} \Res^M_{K_MC} \LL(U,\pholder) &= \LL_{K_M} \Res^M_{K_MC} \LL_U \Res^G_P\\
&\cong \LL_{K_M} \LL_U \Res^P_{UK_MC}\Res^G_P & &
\text{(by Prop.~\ref{prop:LU-Res})}\\
&\cong \LL_{K_M} \ind_{K_MC^+}^{K_MC} \LL_{K_U} \Res^{UK_MC}_{K_PC^+}
\Res^G_{UK_MC} & & \text{(by Prop.~\ref{prop:LU-positive})}\\
&\cong \ind_{C^+}^C \LL_{K_M} \LL_{K_U} \Res^G_{K_PC^+} & & \text{(by
Prop.~\ref{prop:LU-ind})}\\
&\cong \ind_{C^+}^C \LL_{K_P} \Res^G_{K_PC^+} & & \text{(by
Prop.~\ref{prop:LU-transitive})}
\end{align*}
which proves the first assertion. Now, applying
Corollary~\ref{cor:RH(KU,-)=LKU}, we deduce a
$C$-equivariant natural isomorphism
\[
\RR\H^0\bigl(K_M, \omega_{K_MC}\otimes \LL(U,X)\bigr) \cong
\ind_{C^+}^C \RR\H^0\bigl(K_P, \omega_{UK_MC^+}\otimes X\bigr).
\]
Restricting to $k$ and passing to cohomology yields the final assertion, since
$\omega_{K_MC}$ is concentrated in degree $-\dim K_M$ and $\omega_{UK_MC^+}$ is
concentrated in degree $-\dim K_P$, and because $\dim K_P-\dim K_M = \dim K_U$.
\end{proof} 

We recall the Hecke action of $C^+$ on $\H^0(K,V)$, for $V\in \Rep_k(G)$: 
Note first that the inclusion $K_U\subseteq K$ induces, for each $z\in C^+$, a
bijection
\begin{equation}\label{eq:KU-quotient}
K_U/zK_Uz^{-1} \cong K/K_{\ol P}zK_Uz^{-1},
\end{equation}
since $K = K_UK_{\ol P}$. (By our choice of $z_0$, $K_{\ol P}zK_Uz^{-1} = K\cap
zKz^{-1}$ is an open subgroup of $K$.) Then the composite
\[
\begin{tikzcd}[row sep=0em]
z\star (\pholder)\colon &[-3em] \H^0(K,V) \ar[r,"z\cdot"] & 
\H^0\bigl(K_{\ol P}zK_Uz^{-1}, V\bigr) \ar[r,"\cores"] & \H^0(K,V),\\
& v \ar[r,mapsto] & zv \ar[r,mapsto] & \sum_{u\in K_U/zK_Uz^{-1}} uzv
\end{tikzcd}
\]
defines the Hecke action of $C^+$ on $\H^0(K,V)$ such that the natural
inclusion $\H^0(K,V)\subseteq \H^0(K_P,V)$ is $C^+$-equivariant. We obtain a
left exact functor $\H^0(K,\pholder)\colon \Rep_k(G)\to \Rep_k(C^+)$ and hence a
derived functor
\[
\RR\H^0(K,\pholder)\colon \D(G)\to \D(C^+)
\]
on the derived categories.

\begin{lem}\label{lem:gadm-2} 
The inclusion $\H^0(K,\pholder) \hookrightarrow \H^0(K_P,\pholder)$ induces a natural
isomorphism
\[
\ind_{C^+}^C \RR\H^0(K,\pholder) \xRightarrow{\cong} \ind_{C^+}^C \RR\H^0(K_P,\pholder)
\Res^G_{K_PC^+}.
\]
\end{lem} 
\begin{proof} 
We first show that $\Res^G_{K_PC^+}$ sends injective objects to
$\H^0(K_P,\pholder)$-acyclic objects. As $K_P$ is open in $K_PC^+$, it follows that
$\Res^{K_PC^+}_{K_P}$ admits an exact left adjoint and hence preserves injective
objects (Lemma~\ref{lem:adjoint-injective}). As $\H^0(K_P,\pholder)$ has finite
cohomological dimension, the functor $\RR\H^0(K_P,\pholder)$ can be computed using
acyclic objects, see \cite[Cor.~5.3 $\gamma$]{Hartshorne.1966}. Hence, there is
an isomorphism
\[
\Res^{C^+}_1 \RR\H^0(K_P,\pholder) \cong \RR\H^0(K_P,\pholder)\Res^{K_PC^+}_{K_P}.
\]
Now, if $J$ in $\Rep_k(G)$ is injective, then so is $\Res^G_{K_P}(J)$ by
\cite[Prop.~2.1.11]{EmertonII} and therefore,
\[
\Res^{C^+}_1 \H^i\bigl(K_P, \Res^G_{K_PC^+}(J)\bigr) = \H^i\bigl(K_P,
\Res^G_{K_P}(J)\bigr) = 0,\qquad \text{for all $i>0$.}
\]
Hence, $\Res^G_{K_PC^+}(J)$ is $\H^0(K_P,\pholder)$-acyclic.\medskip

We are thus reduced to showing that the natural map
\begin{equation}\label{eq:ind+-invariant}
\ind_{C^+}^C \H^0(K,V) \to \ind_{C^+}^C \H^0(K_P,V)
\end{equation}
is bijective, for all $V\in \Rep_k(G)$. As $\ind_{C^+}^C$ is exact, the
injectivity is clear. The surjectivity follows from the same argument as in
\cite[Lem.~3.4.5]{EmertonII} which we recall here for the benefit of the reader:
Let $v\in \H^0(K_P,V)$ be arbitrary. As $V$ is smooth, there exists $z\in C^+$
such that $v$ is fixed by $z^{-1}K_{\ol U}zK_P = z^{-1}Kz\cap K$. Now, we have
\[
z\star v = \sum_{u\in K_U/zK_Uz^{-1}} uzv \in \H^0(K,V),
\]
because $zv$ is fixed by $K_{\ol P}zK_Uz^{-1}$ and $u$ runs through a
representing system of $K/K_{\ol P}zK_Uz^{-1}$, see~\eqref{eq:KU-quotient}.
Therefore, $z^{-1}\otimes z\star v$ maps to $1\otimes v \in \ind_{C^+}^C
\H^0(K_P,V)$ under \eqref{eq:ind+-invariant}. This finishes the proof.
\end{proof} 

\begin{lem}\label{lem:gadm-3} 
Let $V\in \Rep_k(C^+)$ be a finite-dimensional representation. The canonical map
\begin{align*}
V &\longtwoheadrightarrow \ind_{C^+}^C V,\\
v &\longmapsto 1\otimes v
\end{align*}
is surjective.
\end{lem} 
\begin{proof} 
This is immediate from the proof of \cite[Lem.~3.2.1 (1)]{EmertonII}.
\end{proof} 

\begin{prop}\label{prop:L(U,-)-gadm} 
The functor $\LL(U,\pholder)$ restricts to a functor $\D(G)^{\gadm} \to \D(M)^{\gadm}$.
\end{prop} 
\begin{proof} 
Let $X$ in $\D(G)^{\gadm}$ be a globally admissible complex and let $i\in\Z$ be
arbitrary. Since $\Res^{C^+}_1 \RR\H^0(K,\pholder) \cong
\RR\H^0(K,\pholder) \Res^G_K$, we know that $\H^i(K,X)$ is a finite-dimensional smooth
$C^+$-representation. Now, the sequence of maps
\begin{align*}
\H^i(K,X) &\longtwoheadrightarrow \ind_{C^+}^C \H^i(K,X) & &
\text{(Lemma~\ref{lem:gadm-3})}\\
&\xrightarrow{\cong} \ind_{C^+}^C \H^i(K_P,X) & &
\text{(Lemma~\ref{lem:gadm-2})}\\
&\xrightarrow{\cong} \H^{i-\dim K_U}\bigl(K_M, \LL(U,X)\bigr) & &
\text{(Lemma~\ref{lem:gadm-1})}
\end{align*}
shows that $\H^{i-\dim K_U}\bigl(K_M, \LL(U,X)\bigr)$ is finite-dimensional.
Therefore $\LL(U,X)$ is globally admissible.
\end{proof}

Finally, Propositions~\ref{prop:i_M^G-gadm} and~\ref{prop:L(U,-)-gadm} together
imply Theorem~\ref{thm:globally-admissible}.
\subsection{On the Satake homomorphism}\label{subsec:Satake} 
The functor $\LL(U,\pholder)\colon \D(G)\to \D(M)$ can be used to define a derived
version of the Satake homomorphism, which we will now describe.

We fix an open compact mod center subgroup $K\subseteq G$ satisfying the
Iwasawa decomposition $G = PK$ and $K_P = K_UK_M$. Writing $\LL(K_U,\pholder)\coloneqq
\LL_{K_U} \circ\Res^{K}_{K_P}\colon
\D(K)\to \D(K_M)$, there are natural isomorphisms 
\begin{align*}
\LL(U,\pholder) \ind_K^G &= \LL_U \Res^G_P \ind_K^G \xRightarrow{\cong} \LL_U
\ind_{K_P}^P\Res^{K}_{K_P} \xRightarrow{\cong} \ind_{K_M}^M \LL(K_U,\pholder),
\end{align*}
where the first isomorphism follows from the Iwasawa decomposition and the last
is the one from Proposition~\ref{prop:LU-ind}.

\begin{defn} 
Given $X\in \D(K)$, we call the $k$-algebra homomorphism
\[
\Satake_X\colon \End_{\D(G)}\bigl(\ind_K^GX\bigr) \to
\End_{\D(M)}\bigl(\ind_{K_M}^M \LL(K_U,X)\bigr)
\]
induced by $\LL(U,\pholder)$ the \emph{derived Satake homomorphism}.

If $V\in \Rep_k(K)$, composing $\Satake_V$ with the cohomology functor
$\H^{-n}$ yields a homomorphism
\[
\Satake_V^n\colon \End_{\Rep_k(G)}\bigl(\ind_K^G V\bigr) \to \End_{\Rep_k(M)}
\bigl(\ind_{K_M}^M \LL^{-n}(K_U,V)\bigr)
\]
of $k$-algebras, which we call the \emph{$n$-th Satake homomorphism}. 
(As usual, we write $\LL^{-n}(K_U,V)\coloneqq \H^{-n}(\LL(K_U,V[0]))$.)
Note that
$\Satake_V^n$ is zero whenever $n\notin \{0,1,\dotsc,\dim K_U\}$.
\end{defn} 

\begin{rmk}\label{rmk:Ronchetti} 
One could also define a variant of the derived Satake homomorphism via
\[
\RHom_{\Rep_k(G)}\bigl(\ind_K^GV, \ind_K^GV\bigr) \to
\RHom_{\Rep_k(M)}\bigl(\ind_{K_M}^M\LL(K_U,V), \ind_{K_M}^M
\LL(K_U,V)\bigr).
\]
It might be interesting to relate the induced map on the cohomology algebras to
Ronchetti's derived Satake homomorphism, \cite{Ronchetti.2019}.
\end{rmk} 

The aim of this section is to show that $\Satake^0_V$ and $\Satake^{\dim K_U}_V$
are the well-known variants of the Satake homomorphism that were introduced by
Herzig, \cite{Herzig.2011b}, and extensively studied by Herzig and
Henniart--Vign\'eras, \cite{Herzig.2011a,
Henniart-Vigneras.2015, Henniart-Vigneras.2012}.

\begin{thm}\label{thm:Satake} 
Let $V\in \Rep_k(K)$.
\begin{enumerate}[label=(\alph*)]
\item\label{thm:Satake-a} Denote $\eta\colon V\longtwoheadrightarrow
\LL^0(K_U,V)$ the natural projection map. The map $\Satake^0_V$ is explicitly
given by
\begin{align*}
\End_{\Rep_k(G)}(\ind_K^GV) &\to \End_{\Rep_k(M)}\bigl(\ind_{K_M}^M
\LL^0(K_U,V)\bigr),\\
\Satake^0_V(\Phi)([1,\eta(v)])(m) &= \eta\Bigl(\sum_{u\in K_U\backslash U}
\Phi([1,v])(um)\Bigr)
\end{align*}
for all $\Phi\in \End_{\Rep_k(G)}(\ind_K^GV)$, $v\in V$, and $m\in M$.

\item\label{thm:Satake-b} The map $\Satake^{\dim K_U}_V$ is explicitly given by
\begin{align*}
\End_{\Rep_k(G)}(\ind_K^GV) &\to \End_{\Rep_k(M)}\bigl(\ind_{K_M}^M
\H^0(K_U,V)\bigr),\\
\Satake_V^{\dim K_U}(\Phi)([1,v])(m) &= \sum_{u\in U/K_U} \Phi([1,v])(mu)
\end{align*}
for all $\Phi\in \End_{\Rep_k(G)}(\ind_K^GV)$, $v\in \H^0(K_U,V)$, and $m\in
M$.
\end{enumerate}
\end{thm} 

\begin{explanation} 
Since $\delta_P\otimes\pholder$ is an equivalence of categories, there are
natural isomorphisms $\H^0(K_U,\pholder)\circ (\delta_P\otimes \pholder) \xRightarrow{\cong}
(\delta_P\otimes \pholder)\circ \H^0(K_U,\pholder)$ and $(\delta_P\otimes \pholder)\circ \ind_{K_M}^M
\xRightarrow{\cong} \ind_{K_M}^M\circ (\delta_P\otimes \pholder)$. By
Corollary~\ref{cor:RH(KU,-)=LKU} we have a natural isomorphism $\LL^{-\dim
K_U}(K_U,V) \cong \H^0(K_U,\delta_P\otimes V)\cong \delta_P\otimes \H^0(K_U,V)$.
Consequently, the target of $\Satake_V^{\dim K_U}$ may be identified with
\[
\End_{\Rep_k(M)}\bigl(\ind_{K_M}^M\LL^{-\dim K_U}(K_U,V)\bigr) \cong
\End_{\Rep_k(M)}\bigl(\ind_{K_M}^M \H^0(K_U,V)\bigr),
\]
which justifies the formulation in
Theorem~\ref{thm:Satake}.\ref{thm:Satake-b}.
\end{explanation} 

The proof of Theorem~\ref{thm:Satake}.\ref{thm:Satake-b} needs some
preparation and is deferred until the end of the section. Statement
\ref{thm:Satake-a} is well-known. Since I could not find a precise
reference to the literature, I will recall the easy proof for the convenience of
the reader.

\begin{proof}[Proof of Theorem~\ref{thm:Satake}.\ref{thm:Satake-a}] 
By definition the diagram
\[
\begin{tikzcd}
\ind_{K}^GV \ar[d,"\Phi"'] \ar[r,two heads,"\eta_U"] & \LL^0(U,\ind_{K}^GV)
\ar[d,"{\LL^0(U,\Phi)}"] \ar[r,"\cong"] & \ind_{K_M}^M\LL^0(K_U,V)
\ar[d,"\Satake_V^0(\Phi)"] \\
\ind_{K}^GV \ar[r, two heads, "\eta_U"'] & \LL^0(U,\ind_{K}^GV) \ar[r,"\cong"']
& \ind_{K_M}^M\LL^0(K_U,V)
\end{tikzcd}
\]
is commutative, where the horizontal compositions are given explicitly by
\[
[g,v] \longmapsto [\pr_{M}(g), \eta(v)],\qquad \text{for $g\in P$ and $v\in V$.}
\]
(Recall that $G = PK$.) Fix a complete representing
system $\mathcal{M}$ of $K_M\backslash M$ in $M$. Then the map
\begin{align*}
K_U\backslash U \times \mathcal{M} &\xrightarrow{\cong} K_P\backslash P,\\
(K_Uu, m) &\longmapsto K_Pum
\end{align*}
is bijective. Using $\Phi([1,v]) = \sum_{u\in K_U\backslash U}
\sum_{m\in \mathcal{M}} \bigl[(um)^{-1}, \Phi([1,v])(um)\bigr]$, we compute
\begin{align*}
\Satake^0_V(\Phi)\bigl([1,\eta(v)]\bigr) &= \sum_{u\in K_U\backslash U}
\sum_{m\in \mathcal{M}} \bigl[m^{-1}, \eta\bigl(\Phi([1,v])(um)\bigr)\bigr]\\
&= \sum_{m\in \mathcal{M}} \Bigl[m^{-1}, \eta\Bigl(\sum_{u\in K_U\backslash U}
\Phi([1,v])(um)\Bigr)\Bigr],
\end{align*}
which shows $\Satake^0_V(\Phi)\bigl([1,\eta(v)]\bigr)(m) = \eta\bigl(\sum_{u\in
K_U\backslash U} \Phi([1,v])(um)\bigr)$ for all $m\in M$.
\end{proof} 

\subsubsection{Finishing the proof of 
Theorem~\ref{thm:Satake}.\ref{thm:Satake-b}}
Fix $W\in \Rep_k(K_P)$ and recall from Corollary~\ref{cor:LU-explicit} the
isomorphism
\[
\LL^{-\dim K_U}_U(\ind_{K_P}^P W) \cong \ind_{M^+}^M \H^0\bigl(K_U, \delta_P
\otimes \ind_{K_P}^PW\bigr) \cong \delta_P \otimes \ind_{M^+}^M \H^0(K_U,
\ind_{K_P}^PW).
\]
As was already remarked, the character $\delta_P$ can be neglected for the
computation of the Satake homomorphism. The $M^+$-action on $\H^0(K_U,
\ind_{K_P}^PW)$ is given by
\[
m\star f \coloneqq \sum_{u\in K_U/mK_Um^{-1}} umf,
\]
for $m\in M^+$ and $f\in \ind_{K_P}^PW$. By the Mackey decomposition,
Lemma~\ref{lem:derived-Mackey}, we have a natural isomorphism
\begin{equation}\label{eq:Mackey-H^0-ind}
\H^0\bigl(K_U, \ind_{K_P}^PW\bigr) \cong \bigoplus_{g\in K_U\backslash P/K_P}
\H^0(K_U^g\cap K_P, W),
\end{equation}
where $K_U^g\coloneqq g^{-1}K_Ug$. We deduce that $\H^0(K_U,\ind_{K_P}^PW)$ is
generated by elements of the form
\[
[K_Ug,w] \coloneqq \sum_{u\in K_U/K_U\cap K_P^{g^{-1}} } [ug, w],\quad \text{for
$w\in \H^0(K_U^g\cap K_P,W)$,}
\]
where $g$ runs through a representing system for $K_U\backslash P/K_P$. In order
to describe the $M^+$-action on the $[K_Ug,w]$ we make the following definition:

\begin{defn} 
Given $m\in M^+$ and $g\in P$, we define the ``projection map''
\begin{align*}
\mu_{m,g} \colon \H^0(K_U^g\cap K_P, W) &\to \H^0(K_U^{mg}\cap K_P,W),\\
w &\longmapsto \sum_{u\in K_U^{mg}\cap K_P/K_U^g\cap K_P} uw.
\end{align*}
As $m$ is positive, we have $K_U^{mg}\cap K_P \supseteq K_U^g\cap K_P$, and
hence $\mu_{m,g}$ is well-defined.
\end{defn} 

\begin{lem}\label{lem:mu_mg-rules} 
Let $g\in P$ and $m,m_1,m_2\in M^+$.
\begin{enumerate}[label=(\alph*)]
\item\label{lem:mu_mg-rules-a} $\mu_{m,g} =
\id_{\H^0(K_U,W)}$ provided $g\in P^+ = K_UM^+$.
\item\label{lem:mu_mg-rules-b} $\mu_{m_1,g} = \mu_{m_2,g}$ provided
$m_1g,m_2g\in P^+$.
\item\label{lem:mu_mg-rules-c} $\mu_{1,g} = \id_{\H^0(K_U^g\cap K_P,W)}$.
\item\label{lem:mu_mg-rules-d} $\mu_{m_1m_2,g} = \mu_{m_1,m_2g} \circ
\mu_{m_2,g}$.
\item\label{lem:mu_mg-rules-e} $\mu_{m,ugh}(w) = h^{-1} \mu_{m,g}(hw)$ for each
$u\in K_U$, $h\in K_P$, $w\in \H^0(K_U^{ugh}, W)$.
\end{enumerate}
\end{lem} 
\begin{proof} 
Note that, if $h\in P^+$, then $K_U^h\cap K_P = K_U$. Now,
\ref{lem:mu_mg-rules-a} follows by applying this observation to $h = g$ and $h =
mg$, and \ref{lem:mu_mg-rules-b} follows by applying it to $h = m_1g$ and $h =
m_2g$.

\ref{lem:mu_mg-rules-c} is trivial and \ref{lem:mu_mg-rules-d} is
straightforward. 
For \ref{lem:mu_mg-rules-e} note that $K_U^u = K_U$ and $K_U^{mu} = K_U^m$,
because $u, mum^{-1} \in K_U$.
\end{proof} 

\begin{notation} 
Given $g\in P$, we put 
\[
\mu_g \coloneqq \mu_{m,g}\colon \H^0(K_U^g\cap K_P,W)\to
\H^0(K_U, W),
\]
where $m\in M^+$ is any element satisfying $mg\in P^+$. (Such elements exist by
Hypothesis~\ref{hyp:positive}.) By
Lemma~\ref{lem:mu_mg-rules}.\ref{lem:mu_mg-rules-b} the definition of $\mu_g$ is
independent of the chosen $m$.
\end{notation} 

\begin{lem}\label{lem:Hecke-action-ind} 
For all $g\in P$, $w\in \H^0(K_U^g\cap K_P,W)$ and $m\in M^+$ we have
\[
m\star [K_Ug,w] = [K_Umg, \mu_{m,g}(w)]\qquad \text{in $\H^0(K_U,
\ind_{K_P}^PW)$.}
\]
\end{lem} 
\begin{proof} 
We compute
\begin{align*}
m\star [K_Ug,w] &= \sum_{n\in K_U/K_U^{m^{-1}} } \sum_{u\in K_U/K_U \cap
K_P^{g^{-1}} } [nmug, w]\\
&= \sum_{n\in K_U^{mg}/K_U^g} \sum_{u\in K_U^g/K_U^g\cap K_P} [mgnu, w]\\
&= \sum_{n\in K_U^{mg}/K_U^g\cap K_P} [mgn,w]\\
&= \sum_{n\in K_U^{mg}/K_U^{mg}\cap K_P} \sum_{u\in K_U^{mg}\cap K_P/ K_U^g\cap
K_P} [mgnu,w]\\
&= \sum_{n\in K_U/K_U\cap K_P^{(mg)^{-1}} } \Bigl[nmg, \sum_{u\in K_U^{mg}\cap
K_P/K_U^g\cap K_P} uw\Bigr]\\
&= [K_Umg, \mu_{m,g}(w)].
\end{align*}
\end{proof} 

\begin{prop}\label{prop:ind-H^0} 
The natural map
\begin{align*}
\ind_{M^+}^M \H^0\bigl(K_U, \ind_{K_P}^PW\bigr) &\xrightarrow{\cong}
\ind_{K_M}^M \H^0(K_U,W),\\
1\otimes [K_Ug,w] &\longmapsto \bigl[\pr_M(g), \mu_{g}(w)\bigr]
\end{align*}
is an $M$-equivariant isomorphism.
\end{prop} 
\begin{proof} 
We first observe that the inclusion $\ind_{K_P}^{P^+}W \subseteq \ind_{K_P}^P W$
induces an isomorphism
\[
\ind_{M^+}^M \H^0(K_U,\ind_{K_P}^{P^+}W) \xrightarrow{\cong} \ind_{M^+}^M
\H^0(K_U,\ind_{K_P}^PW).
\]
Indeed, injectivity is clear, since $\H^0(K_U,\pholder)$ and $\ind_{M^+}^M$ are left
exact. The surjectivity follows from the fact that for every $g\in P$ there
exists $m\in M^+$ such that $mg\in P^+$. Concretely, the inverse map is given by
\begin{align}\label{eq:ind-H^0-1} 
\ind_{M^+}^M \H^0(K_U, \ind_{K_P}^PW) &\to \ind_{M^+}^M \H^0(K_U,
\ind_{K_P}^{P^+}W),\\
1\otimes [K_Ug,w] &\longmapsto m^{-1}\otimes [K_Umg, \mu_{m,g}(w)],
\end{align} 
where $m\in M^+$ is chosen such that $mg\in P^+$. (Note that the same $m$ works
if $g$ is replaced by any other element in $K_Ug$.)

We now prove that the map
\begin{align}\label{eq:ind-H^0-2}
\H^0(K_U, \ind_{K_P}^{P^+}W) &\xrightarrow{\cong} \ind_{K_M}^{M^+} H^0(K_U,
W),\\
[K_Um,w] &\longmapsto [m,w],
\end{align}
where $m\in M^+$, is an $M^+$-equivariant isomorphism. The $M^+$-equivariance
of the map follows from Lemma~\ref{lem:Hecke-action-ind} together with
$\mu_{m',m} = \id$,
for all $m',m\in M^+$ (Lemma~\ref{lem:mu_mg-rules}.\ref{lem:mu_mg-rules-a}). As
for \eqref{eq:Mackey-H^0-ind} the
Mackey decomposition provides a canonical isomorphism
\[
\H^0(K_U, \ind_{K_P}^{P^+}W) \cong \bigoplus_{m\in K_U\backslash
P^+/K_P} \H^0(K_U,W).
\]
Since the projection map $K_U\backslash P^+/K_P \to M^+/K_M$ is
bijective, we deduce that \eqref{eq:ind-H^0-2} is an isomorphism.

The map in the proposition is now given by \eqref{eq:ind-H^0-1} composed with
$\ind_{M^+}^M\eqref{eq:ind-H^0-2}$ and hence is an isomorphism.
\end{proof} 

\begin{lem}\label{lem:Satake} 
The following assertions hold:
\begin{enumerate}[label=(\alph*)]
\item\label{lem:Satake-a} 
\[
K_P gK_U = \bigsqcup_{u\in K_U/ K_U\cap gK_Ug^{-1}} K_MugK_U,\qquad
\text{for all $g\in P$.}
\]
\item\label{lem:Satake-b}
\[
K_M\backslash P/K_U \cong K_M\backslash M \times U/K_U.
\]
\end{enumerate}
\end{lem} 
\begin{proof} 
\ref{lem:Satake-a} is easy and \ref{lem:Satake-b} is trivial.
\end{proof} 

Finally, we finish the proof of Theorem~\ref{thm:Satake}.

\begin{proof}[Proof of Theorem~\ref{thm:Satake}.\ref{thm:Satake-b}] 
Let $V\in \Rep_k(K)$ and $\Phi\in \End_{\Rep_k(G)}(\ind_K^GV)$. Recall that we
want to show
\[
\Satake^{\dim K_U}_V(\Phi)([1,v])(m) = \sum_{u\in U/K_U} \Phi([1,v])(mu)
\]
for all $v\in \H^0(K_U,V)$ and $m\in M$. Recall also that the Iwasawa
decomposition $G = PK$ allows us to identify $\ind_K^GV$ with $\ind_{K_P}^PV$.

Note that $K_PgK_U = \bigsqcup_{u\in K_U\cap K_P^g\backslash K_U} K_Pgu$.
Therefore, we have
\begin{align*}
\Phi([1,v]) &= \sum_{g\in K_P\backslash P} [g^{-1}, \Phi([1,v])(g)] \\
&= \sum_{g\in K_P\backslash P/K_U} \sum_{u\in K_U\cap K_P^g\backslash K_U}
\bigl[u^{-1} g^{-1}, \Phi([1,v])(gu)\bigr]\\
&= \sum_{g\in K_P\backslash P/K_U} \sum_{u\in K_U\cap K_P^g\backslash K_U}
\bigl[u^{-1} g^{-1}, \Phi([1,v])(g)\bigr],
\end{align*}
where the last equality uses $v\in \H^0(K_U,V)$ and the $K_U$-equivariance of
$\Phi$. 
Viewing $\Satake^{\dim K_U}_V(\Phi)$ as an endomorphism on $\ind_{M^+}^M
\H^0(K_U,\ind_{K_P}^PV)$ we deduce
\[
\Satake^{\dim K_U}_V(\Phi)\bigl(1\otimes [K_U,v]\bigr) = \sum_{g\in
K_P\backslash P/K_U} 1\otimes [K_Ug^{-1}, \Phi([1,v])(g)].
\]
Now, applying the isomorphism in Proposition~\ref{prop:ind-H^0}, we view
$\Satake^{\dim K_U}_V(\Phi)$ as an endomorphism of $\ind_{K_M}^M \H^0(K_U,V)$
and compute
\begin{align*}
\Satake^{\dim K_U}_V(\Phi)([1,v]) &= \sum_{g\in K_P\backslash P/K_U}
\bigl[\pr_M(g^{-1}), \mu_{g^{-1}}(\Phi([1,v])(g))\bigr]\\
&= \sum_{\substack{g\in K_P\backslash P/K_U,\\ u\in K_U/K_U\cap gK_Ug^{-1}} }
\bigl[\pr_M(g^{-1}), \Phi([1,v])(ug)\bigr] & & \text{(def. of $\mu_{g^{-1}}$)}\\
&= \sum_{g\in K_M\backslash P/K_U} \bigl[\pr_M(g^{-1}), \Phi([1,v])(g)\bigr] & &
\text{(Lem.~\ref{lem:Satake}.\ref{lem:Satake-a})}\\
&= \sum_{m\in K_M\backslash M} \Bigl[m^{-1}, \sum_{u\in U/K_U} 
\Phi([1,v])(mu)\Bigr] & & \text{(Lem.~\ref{lem:Satake}.\ref{lem:Satake-b})},
\end{align*}
where for the second equality we have used $K_P\cap gK_Ug^{-1} = K_U \cap
gK_Ug^{-1}$ (recall that $U$ is normalized by $P$). This implies the claim.
\end{proof}

\section{The example of \texorpdfstring{$\GL_2(\Q_p)$}{GL(2,Qp)}} 
\label{sec:example}
In this section we put $G \coloneqq \GL_2(\Q_p)$ and let $B$ be the Borel
subgroup of upper triangular matrices with unipotent radical $U$ and $T$
the Levi subgroup of diagonal matrices. Further, we fix the maximal compact open
subgroup $K \coloneqq \GL_2(\Z_p)$ and denote by $Z$ the center of
$\GL_2(\Q_p)$. We fix an algebraically closed coefficient field $k$ of
characteristic $p$. 

We will compute $\LL^\bullet(U,V)$ for all
irreducible smooth representations $V$ of $G$. 

\subsection{The weights of \texorpdfstring{$K$}{K}} 
The irreducible smooth representations of $K$ have been classified by
Barthel--Livn\'e in \cite[Prop.~4]{Barthel-Livne.1994}. We recall the
classification and set up some notation.

\begin{defn} 
Given $0\le r\le p-1$ and $0\le e <p-1$, we consider
\begin{equation}\label{eq:weights}
\Sym^r(k^2) \otimes {\det}^e.
\end{equation}
Every irreducible smooth representation of $K$ is isomorphic to some
representation of the form \eqref{eq:weights}.
We identify $\Sym^r(k^2)$ with the subvector space of homogeneous polynomials of
degree $r$ in $k[x,y]$. Under this identification, $K$ acts through its quotient
$\GL_2(\F_p)$ via
\[
\begin{pmatrix}
a & b\\c & d
\end{pmatrix} \cdot f(x,y) \coloneqq f(ax+cy, bx+dy),\quad \text{for all
$\begin{pmatrix}a&b\\c&d\end{pmatrix} \in \GL_2(\F_p)$.}
\]
\end{defn} 

Thus, a $k$-basis of $\Sym^r(k^2)$ is given by $x^r, x^{r-1}y, \dotsc,y^r$, and
we have
\begin{align*}
\H^0\bigl(K_U, \Sym^r(k^2)\bigr) &= \spann_k\{x^r\}, \\
\LL^0\bigl(K_U,\Sym^r(k^2)\bigr) &= \Sym^r(k^2)/\spann_k\{x^r,x^{r-1}y,\dotsc,
xy^{r-1}\},
\end{align*}
where $\spann_k\{\ldots\}$ denotes the $k$-linear span.
\subsection{Irreducible smooth representations of 
\texorpdfstring{$\GL_2(\Q_p)$}{GL(2,Qp)}}
We view $\Sym^r(k^2)$ as a smooth $ZK$-representation by letting $p$ act
trivially.
Denote by $X$ (resp.\ $Y$) the $k$-linear projection of $\Sym^r(k^2)$ onto
$\spann_k\{x^r\}$ (resp.\ $\spann_k\{y^r\}$) with respect to the basis
$x^r,x^{r-1}y,\dotsc,y^r$. We define the endomorphism $\Phi \in
\End_{\Rep_k(G)}\bigl(\ind_{ZK}^G \Sym^r(k^2)\bigr)$ by
\[
\Phi([1,v]) \coloneqq \Bigl[\begin{pmatrix}p&0\\0&1\end{pmatrix}^{-1},
Y(v)\Bigr] + \sum_{i=0}^{p-1} \Bigl[\begin{pmatrix}1 &i\\0 &
p\end{pmatrix}^{-1}, X(\begin{pmatrix}1&i\\0&1\end{pmatrix}\cdot v)\Bigr].
\]
Then \cite[Prop.~8]{Barthel-Livne.1994} shows that
$\End_{\Rep_k(G)}\bigl(\ind_{ZK}^G\Sym^r(k^2)\bigr) \cong k[\Phi]$ as
$k$-algebras.

\begin{defn} 
For $0\le r\le p-1$ and $\lambda\in k$ we define the smooth $G$-representation
$V(r,\lambda)$ by the short exact sequence
\begin{equation}\label{eq:irreducible}
0\to \ind_{ZK}^G\Sym^r(k^2) \xrightarrow{\Phi-\lambda} \ind_{ZK}^G\Sym^r(k^2)
\to V(r,\lambda) \to 0.
\end{equation}
Note: $\Phi-\lambda$ is injective by \cite[Thm.~19]{Barthel-Livne.1994}. Given a
smooth character $\chi\colon \Q_p^\times \to k^\times$, we put
\[
V(r,\lambda,\chi) \coloneqq (\chi\circ\det)\otimes V(r,\lambda).
\]
\end{defn} 

\begin{notation}\label{nota:characters} 
For each $\lambda\in k$ we consider the characters 
\begin{alignat}{4}
\mu_\lambda\colon \Q_p^\times &\to k^\times,&\hspace{4em}  x&\longmapsto
\lambda^{\val_p(x)}\quad \text{and}\\
\omega\colon \Q_p^\times\cong p^{\Z}\times\Z_p^\times &\to \F_p^\times
\subseteq k^\times, &
(p^n,x) &\longmapsto x\bmod p.
\intertext{If $\chi_1,\chi_2\colon \Q_p^\times \to k^\times$ are smooth
characters, we write}
\chi_1\boxtimes \chi_2 \colon T &\to k^\times, & 
\begin{pmatrix}a&0\\0&d\end{pmatrix} &\longmapsto \chi_1(a)\chi_2(d).
\end{alignat}
\end{notation} 

The representations
$V(r,\lambda,\chi)$ are now described as follows:

\begin{thm}[{\cite[Thm.~30]{Barthel-Livne.1994}, 
\cite[Thm.~1.1]{Breuil.2003}}]\label{thm:irreducible}
Let $(r,\lambda)\in \{0,1,\dotsc,p-1\}\times k$ and let $\chi\colon
\Q_p^\times\to k^\times$ be a smooth character.
\begin{enumerate}[label=(\alph*)]
\item\label{thm:irreducible-a} If $(r,\lambda)\notin \{(0, \pm 1), (p-1,\pm
1)\}$, then $V(r,\lambda,\chi)$ is irreducible. 
\item\label{thm:irreducible-b} There are non-split short exact sequences
\begin{gather}
0\to (\chi\mu_{\pm 1}\circ\det) \otimes \Sp \to V(0,\pm1, \chi) \to
\chi\mu_{\pm1}\circ\det \to 0\\
\intertext{and}
0\to \chi\mu_{\pm1}\circ\det \to V(p-1,\pm1,\chi) \to (\chi\mu_{\pm1}\circ\det)
\otimes \Sp \to 0,
\end{gather}
where $\Sp = i_T^G(\one_T)/\one_G$ is the Steinberg representation, and where
$\one_T$ and $\one_G$ are the trivial representations of $T$ and $G$,
respectively.

\item\label{thm:irreducible-c} If $(r,\lambda)\neq (0,\pm1)$ and $\lambda\neq
0$, there is a $G$-equivariant isomorphism
\[
V(r,\lambda,\chi) \xrightarrow{\cong} i_T^G(\chi\mu_{\lambda^{-1}}\boxtimes
\chi\mu_\lambda\omega^r).
\]

\item\label{thm:irreducible-d} $V(r,\lambda,\chi)$ is supersingular if and only
if $\lambda=0$.
\end{enumerate}
\end{thm} 

By \cite[Thm.~33]{Barthel-Livne.1994}, every smooth irreducible representation
of $G$ admitting a central character is a quotient of some $V(r,\lambda,\chi)$.

\subsection{Computation of the left adjoint on irreducibles} 
Recall the endomorphism $\Phi$ on the representation $\ind_{ZK}^G\Sym^r(k^2)$
from the previous section. 
The character $\delta_B\colon B\to k^\times$ is (by Remark~\ref{rmk:deltaP})
explicitly given by
\[
\delta_B = \Inf^T_B\bigl( \omega\boxtimes\omega^{-1}\bigr).
\]

\begin{lem}\label{lem:Satake(Phi)} 
Let $0\le r\le p-1$ and recall the map $\eta\colon
\Sym^r(k^2)\longtwoheadrightarrow \LL^0(K_U,\Sym^r(k^2)) \cong
\one\boxtimes\omega^r$.
\begin{enumerate}[label=(\alph*)]
\item\label{lem:Satake(Phi)-a} The endomorphism $\Satake_{\Sym^r(k^2)}^0(\Phi)$
on $\ind_{ZK_T}^T (\one\boxtimes\omega^r)$ is given by
\[
\Satake_{\Sym^r(k^2)}^0(\Phi)\bigl([1,\eta(y^r)]\bigr) = \Bigl[\begin{pmatrix}
p & 0\\0 & 1\end{pmatrix}^{-1}, \eta(y^r)\Bigr].
\]

\item\label{lem:Satake(Phi)-b} The endomorphism
$\Satake_{\Sym^r(k^2)}^{1}(\Phi)$ on $\ind_{ZK_T}^T \delta_B\otimes 
\spann_k\{x^r\}\cong \ind_{ZK_T}^T(\omega^{r+1}\boxtimes \omega^{-1})$ is given by
\[
\Satake_{\Sym^r(k^2)}^1(\Phi)\bigl([1,x^r]\bigr) = 
\Bigl[\begin{pmatrix}1 & 0\\0 & p\end{pmatrix}^{-1}, x^r\Bigr].
\]
\end{enumerate}
\end{lem} 
\begin{proof} 
By Theorem~\ref{thm:Satake} we have
\begin{align*}
\Satake^0_{\Sym^r(k^2)}(\Phi)\bigl([1,\eta(y^r)]\bigr) &= \sum_{t\in
ZK_T\backslash T} \Bigl[t^{-1}, \eta\Bigl(\sum_{u\in K_U\backslash U}
\Phi([1,y^r])(ut)\Bigr)\Bigr]\\
&= \Bigl[\begin{pmatrix}p&0\\0&1\end{pmatrix}^{-1}, \eta(Y(y^r))\Bigr] +
\Bigl[\begin{pmatrix}1 & 0\\0&p\end{pmatrix}^{-1}, \eta\Bigl(\;\sum_{i=0}^{p-1}
X(\begin{pmatrix}1&i\\0&1\end{pmatrix}\cdot y^r)\Bigr)\Bigr]\\
&= \Bigl[\begin{pmatrix}p&0\\0&1\end{pmatrix}^{-1}, \eta(y^r)\Bigr],
\end{align*}
where we have used that the second summand in the second line is zero: If $r>0$,
this follows from $\Image(X) \subseteq\Ker(\eta)$; if $r=0$, we use
$\Char(k) = p$. This proves \ref{lem:Satake(Phi)-a}. For \ref{lem:Satake(Phi)-b}
we compute in a similar fashion:
\begin{align*}
\Satake^1_{\Sym^r(k^2)}(\Phi)\bigl([1,x^r]\bigr) &= \sum_{t\in ZK_T\backslash T}
\Bigl[ t^{-1}, \sum_{u\in U/K_U} \Phi([1,x^r])(tu)\Bigr]\\
&= \Bigl[\begin{pmatrix}p & 0\\0&1\end{pmatrix}^{-1},
\sum_{i=0}^{p-1}\begin{pmatrix}1 &i\\0&1\end{pmatrix}\cdot Y(x^r)\Bigr] +
\Bigl[\begin{pmatrix}1&0\\0&p\end{pmatrix}^{-1}, X(x^r)\Bigr]\\
&= \Bigl[\begin{pmatrix}1&0\\0&p\end{pmatrix}^{-1}, x^r\Bigr]. 
\qedhere
\end{align*}
\end{proof} 

\begin{prop}\label{prop:L(U,V(r,lambda))} 
Let $(r,\lambda)\in \{0,1,\dotsc,p-1\}\times k$ and let $\chi\colon
\Q_p^\times \to k^\times$ be a smooth character. One has
\[
\LL^{-i}\bigl(U, V(r,\lambda,\chi)\bigr) = \begin{cases}
\chi\mu_{\lambda^{-1}}\boxtimes \chi\mu_\lambda\omega^r, & \text{if $i=0$ and
$\lambda\neq 0$,}\\
\chi\mu_\lambda\omega^{r+1}\boxtimes \chi\mu_{\lambda^{-1}}\omega^{-1},&
\text{if $i=1$ and $\lambda\neq 0$,}\\
0,& \text{otherwise.}
\end{cases}
\]
\end{prop} 
\begin{proof} 
By Proposition~\ref{prop:LU-character} we have $\LL\bigl(U,
V(r,\lambda,\chi)\bigr) \cong (\chi\circ\det)\otimes \LL\bigl(U,
V(r,\lambda)\bigr)$. We may therefore assume from the start that $\chi =
\one$ is the trivial character. As $\dim U = 1$, we know from
Proposition~\ref{prop:L(U,-)-derived} that $\LL^{-i}\bigl(U,V(r,\lambda)\bigr) =
0$ provided $i\notin \{0,1\}$.
Consider the short exact sequence
\begin{equation}\label{eq:V(r,lambda)}
0 \to \ind_{ZK}^G \Sym^r(k^2) \xrightarrow{\Phi -\lambda} \ind_{ZK}^G
\Sym^r(k^2) \to V(r,\lambda) \to 0.
\end{equation}
It is immediate from Lemma~\ref{lem:Satake(Phi)}.\ref{lem:Satake(Phi)-a} that
$\Satake^0_{\Sym^r(k^2)}(\Phi) - \lambda$ is injective. Hence, the long exact
sequence associated with $\LL(U,-)$ and \eqref{eq:V(r,lambda)} splits into the
following two short exact sequences:
\begin{gather}
0 \to \ind_{ZK_T}^T (\one\boxtimes \omega^r)
\xrightarrow{\begin{psmallmatrix}p&0\\0&1\end{psmallmatrix}^{-1} - \lambda}
\ind_{ZK_T}^T
(\one\boxtimes \omega^r) \to \LL^0\bigl(U,V(r,\lambda)\bigr) \to 0,\\
\intertext{and}
0 \to \ind_{ZK_T}^T(\omega^{r+1}\boxtimes \omega^{-1})
\xrightarrow{\begin{psmallmatrix}1&0\\0&p\end{psmallmatrix}^{-1}-\lambda}
\ind_{ZK_T}^T (\omega^{r+1}\boxtimes \omega^{-1}) \to \LL^{-1}\bigl(U,
V(r,\lambda)\bigr) \to 0.
\end{gather}
The assertion now follows easily.
\end{proof} 

\begin{cor}\label{cor:irreducibles} 
Let $\chi,\chi_1,\chi_2\colon \Q_p^\times\to k^\times$ be smooth characters with
$\chi_1\neq \chi_2$, and let $0\le r\le p-1$. The $T$-representations
$\LL^{-i}(U,V)$, for irreducible smooth $G$-representations $V$ and $i\in
\{0,1\}$, are as presented in Table~\ref{tab:irreducibles}.
\begin{table}[ht!] 
\begin{tabular}{lccc}\toprule
Type & $V$ & $\LL^{-1}(U,V)$ & $\LL^{0}(U,V)$ \\ \midrule
character & $\chi\circ\det$ & $0$ & $\chi\boxtimes\chi$ \\
special series & $(\chi\circ\det)\otimes \Sp$ &
$\chi\omega\boxtimes\chi\omega^{-1}$ & $0$ \\
principal series & $i_T^G(\chi_1\boxtimes \chi_2)$ & $\chi_2\omega\boxtimes
\chi_1\omega^{-1}$ & $\chi_1\boxtimes\chi_2$ \\
supersingular & $V(r,0,\chi)$ & $0$ & $0$ \\\bottomrule
\end{tabular}
\caption{Computing $\LL^{-i}(U,V)$ for irreducible $V$.}
\label{tab:irreducibles}
\end{table} 
\end{cor} 
\begin{proof} 
If $V \cong \chi\circ\det$, the assertion follows from
Corollary~\ref{cor:LU-trivial}. If $V$ is supersingular, then $\LL(U,V) = \{0\}$
by Proposition~\ref{prop:L(U,V(r,lambda))}. Let now $V\cong
i_T^G(\chi_1\boxtimes \chi_2)$ for two (not necessarily distinct) smooth
characters $\chi_1,\chi_2\colon \Q_p^\times\to k^\times$. Write $\chi_j =
\mu_{\lambda_j}\omega^{r_j}$ for some $\lambda_j\in k^\times$ and $0\le r_j\le
p-1$ and take $\lambda\in k^\times$ such that $\lambda^2 = \lambda_1\lambda_2$.
We have $i_T^G(\chi_1\boxtimes \chi_2) \cong V(r_2-r_1, \lambda_2\lambda^{-1},
\mu_\lambda\omega^{r_1})$ by
Theorem~\ref{thm:irreducible}.\ref{thm:irreducible-c}. Applying
Proposition~\ref{prop:L(U,V(r,lambda))} we compute
\begin{align*}
\LL^0(U,V) &\cong \mu_{\lambda}\omega^{r_1}\mu_{\lambda\lambda_2^{-1}} \boxtimes
\mu_\lambda\omega^{r_1}\mu_{\lambda_2\lambda^{-1}} \omega^{r_2-r_1}\\ 
&\cong \chi_1 \boxtimes \chi_2\\
\intertext{and}
\LL^{-1}(U,V) &\cong \mu_{\lambda}\omega^{r_1}\mu_{\lambda_2\lambda^{-1}}
\omega^{r_2-r_1+1} \boxtimes \mu_{\lambda}\omega^{r_1}
\mu_{\lambda \lambda_2^{-1}}\omega^{-1} \\
& \cong \chi_2\omega \boxtimes \chi_1\omega^{-1}.
\end{align*}
This settles the principal series case. Finally, we deduce the special series
case as follows: By Proposition~\ref{prop:LU-character} we may assume
$\chi=\one$. By Theorem~\ref{thm:irreducible}.\ref{thm:irreducible-b} there
is a short exact sequence
\[
0 \to \Sp \to V(0,1) \to \one \to 0.
\]
The long exact sequence associated with $\LL(U,-)$ yields
\[
\begin{tikzcd}
0 \ar[r] & \LL^{-1}(U,\Sp) \ar[r] & \LL^{-1}\bigl(U,V(0,1)\bigr)
\ar[r] \ar[d,phantom,""{coordinate, name=Z}]&
\LL^{-1}(U,\one) 
\ar[dll, rounded corners, at end, to path={
-- ([xshift=2ex]\tikztostart.east) |- (Z) \tikztonodes -|
([xshift=-2ex]\tikztotarget.west) -- (\tikztotarget)}] \\
& \LL^0(U,\Sp) \ar[r] & \LL^0\bigl(V(0,1)\bigr) \ar[r,"\varphi"'] &
\LL^0(U,\one) \ar[r] & 0
\end{tikzcd}
\]
Since $\LL^{-1}(U,\one) = \{0\}$, we deduce $\LL^{-1}(U,\Sp) \cong
\LL^{-1}\bigl(U,V(0,1)\bigr) \cong \omega\boxtimes \omega^{-1}$.
As $\LL^0\bigl(V(0,1)\bigr) \cong \one\boxtimes\one \cong \LL^0(U,\one)$, it
follows that $\varphi$ is an isomorphism, whence $\LL^0(U,\Sp) = \{0\}$.
\end{proof}

\bibliographystyle{alphaurl}
\bibliography{../references}{}
\end{document}